%% file: LimitFunctions2.tex
\documentclass{memo-l}
\usepackage[utf8]{inputenc}
\usepackage{amsmath}
\usepackage{amssymb}

\usepackage[english]{babel}
\usepackage{fancyhdr}
\usepackage[export]{adjustbox}
\usepackage[colorlinks=true]{hyperref}

\usepackage[usenames,dvipsnames]{color}
\usepackage{array,longtable,calc}
\usepackage{tikz}
\usetikzlibrary{arrows}
\usepackage{float}
\usepackage{afterpage}
\usepackage{wrapfig}
\usepackage{lscape}
\usepackage{longtable}
\usepackage{rotating}
\usepackage{epstopdf}
\usepackage{array}

\usepackage{enumitem}   
\makeatletter
\def\namedlabel#1#2{\begingroup
	#2%
	\def\@currentlabel{#2}%
	\phantomsection\label{#1}\endgroup
}
\makeatother

\usepackage{}

\newtheorem{theorem}{Theorem}[chapter]
\newtheorem{lemma}[theorem]{Lemma}
\newtheorem{corollary}[theorem]{Corollary}
\newtheorem{proposition}[theorem]{Proposition}

\theoremstyle{definition}
\newtheorem{definition}[theorem]{Definition}

\newtheorem{claim}{Claim}[chapter]
\newenvironment{claimproof}[1]{\par\noindent\underline{Proof:}\space#1}{\hfill $\Diamond$}

\theoremstyle{remark}

\numberwithin{section}{chapter}
\numberwithin{equation}{chapter}

\definecolor{darkgreen}{rgb}{0,0.6,0}
\definecolor{darkblue}{rgb}{0.0,0.0,0.8}
\definecolor{orange}{rgb}{1.0,0.45,0.0}

\def \C {\mathbb C}

\def \D {\mathbb D}

\def \R {\mathbb R}

\def \N {\mathbb N}

\def \R {\mathbb R}

\def \T {\mathbb T}

\def \chat {\hat {\mathbb C}}

\def \interior {\mathrm{int\,}}

\def \epsilon {\varepsilon}

\def \Car {Carath\'eodory }

\def \phi {\varphi}

\def \pt {{\rm pt}}

\def \Am {{\mathcal A}_{\infty, m}}

\def \Fm {{\mathcal F}_m}

\def \Fn {{\mathcal F}_n}

\def \Jm {{\mathcal J}_m}

\def \Jn {{\mathcal J}_n}

\def \Km {{\mathcal K}_m}

\def \Pm {\{P_m \}_{m=1}^\infty}

\def \Pmt {\{\tilde P_m \}_{m=1}^\infty}

\def \calE {{\mathcal E}}

\def \calK {\mathcal K}

\def \calS {\mathcal S}

\def \calU {\mathcal U}

\def \calV {\mathcal V}

\def \calW {\mathcal W}

\def \Uu {(U,u)}

\def \Umum {(U_m,u_m)}

\makeindex

\begin{document}

\frontmatter

\title{On Possible Limit Functions on a Fatou Component in non-Autonomous Iteration}


\author{Mark Comerford}
\address{Department of Mathematics,
   University of Rhode Island,
   5 Lippitt Road, Room 102F,
   Kingston, RI 02881, USA}
\curraddr{}
\email{mcomerford@math.uri.edu}
\thanks{}

\author{Christopher Staniszewski}
\address{Department of Mathematics,
   Framingham State University,
   100 State Street
   Framingham, MA 01701, USA}
\curraddr{}
\email{cstaniszewski@framingham.edu}
\thanks{}

\date{}




\begin{abstract}
The possibilities for limit functions on a Fatou component for the iteration of a single polynomial or rational function are well understood and quite restricted. In non-autonomous iteration, where one considers compositions of arbitrary polynomials with suitably bounded degrees and coefficients, one should observe a far greater range of behaviour. We show this is indeed the case and we exhibit a bounded sequence of quadratic polynomials which has a bounded Fatou component on which one obtains as limit functions every member of the classical Schlicht family of normalized univalent functions on the unit disc. The proof is based on quasiconformal surgery and the use of high iterates of a quadratic polynomial with a Siegel disc which closely approximate the identity on compact subsets. Careful bookkeeping using the hyperbolic metric is required to control the errors in approximating the desired limit functions and ensure that these errors ultimately tend to zero.
\end{abstract}

\maketitle

\tableofcontents


\mainmatter
\include{Introduction}

\include{Background}

\include{PIL}

\include{PhaseI}

\include{PhaseII}

\include{MainProof}

\appendix
\include{Appendices}

\backmatter
\bibliographystyle{amsalpha}
\bibliography{references}
\printindex

\end{document}

%% file: Introduction.tex
\chapter{Introduction}

This work is concerned with non-autonomous iteration of bounded sequences of polynomials, a relatively new area of complex dynamics. In classical complex dynamics, one studies the iteration of a (fixed) rational function on the Riemann sphere.  Often in applications of dynamical systems, noise is introduced, and thus it is natural to consider a scheme of iteration where the function at each stage is allowed to vary.  Here we study the situation where the functions being applied are polynomials with appropriate bounds on the degrees and coefficients.  

Non-Autonomous Iteration, in our context, was first studied by Fornaess and Sibony \cite{FS}
and also by Sester, Sumi and others who were working in the closely related area of skew-products \cite{Ses, Sumi1, Sumi2, Sumi3, Sumi4}. Further work was done by Rainer Br\"uck, Stefan Reitz, Matthias B\"uger \cite{BR2, BR1, BBR, BUG}, Michael Benedicks, and the first author \cite{Com4, Com5, Com6, Com8, ComW}, among others. 

One of the main topics of interest in non-autonomous iteration is discovering which results in classical complex dynamics generalize to the non-autonomous setting and which do not.  For instance, the former author proved there is a generalization of the Sullivan Straightening Theorem \cite{CG,Com12,DH2}, while Sullivan's Non-Wandering Theorem \cite{CG, Sul} no longer holds in this context \cite{Com3}.  One can thus construct polynomial sequences which either provide counterexamples or have interesting properties in their own right.

\section{Non-Autonomous Iteration}

Following \cite{Com12, FS}, let $d \ge 2$, $M \geq 0$, $K \geq 1$, and let $\Pm$ be a sequence of polynomials where each $P_m(z) = a_{d_m,m}z^{d_m} + a_{d_m-1,m}z^{d_m-1} + \cdots \cdots + 
a_{1,m}z + a_{0,m}$ is a polynomial of degree $2 \leq d_m \leq d$ whose coefficients satisfy 
\[\qquad 1/K \leq |a_{d_m,m}| \leq K,\: m \geq 1, \quad |a_{k,m}| \leq M,\: m \geq 1,\: 0 \leq k \leq d_m -1. \]
Such sequences are called \emph{bounded sequences of polynomials} or simply \emph{bounded sequences}.  For a constant $C\ge1$, we will say that a bounded sequence is \textit{$C$-bounded} if all of the coefficients in the sequence are bounded above in absolute value by $C$ while the leading coefficients are also bounded below in absolute value by $\tfrac{1}{C}$.

For each $1 \le m$, let $Q_m$ be the composition $P_m \circ \cdots \cdots \circ P_2 \circ P_1$ and for each $0 \le m \le n$, let $Q_{m,n}$ be the composition $P_n \circ \cdots \cdots \circ P_{m+2} \circ P_{m+1}$ (where we set $Q_{m,m} = Id$ for each $m \ge 0$). Let the degrees of these compositions be $D_m$ and $D_{m,n}$ respectively so that $D_m = \prod_{i=1}^m d_i$, $D_{m,n} = \prod_{i=m+1}^n d_i$.

For each $m \ge 0$ define the \emph{$m$th iterated Fatou set} or simply the \emph{Fatou set} at time $m$, $\Fm$, by 
\[ \Fm = \{z \in \chat : \{Q_{m,n}\}_{n=m}^\infty \;
\mbox{is a normal family on some neighborhood of}\; z \}\]
where we take our neighborhoods with respect to the spherical topology on $\chat$. We then set the \emph{$m$th iterated Julia set} or simply the \emph{Julia set} at time $m$, $\Jm$, to be the complement $\chat \setminus \Fm$. 

It is easy to show that these iterated Fatou and Julia sets are completely invariant in the following sense. 

\begin{theorem}
\label{completelyinvariant}
For any $m \le n \in \N$, $Q_{m,n}(\Jm) = \Jn$  
and $Q_{m,n}(\Fm) = \Fn$, with Fatou 
components of $\Fm$ being mapped surjectively onto those of $\Fn$ by 
$Q_{m,n}$. 
\end{theorem}

It is easy to see that, given bounds $d$, $K$, $M$ as above, we can find some radius $R$ 
depending only on $d$, $K$, $M$ so that, for any sequence $\Pm$ with these bounds and any $m \ge 0$, 
\[ |Q_{m,n}(z)| \to \infty \qquad \mbox{as} \quad n \to \infty, \quad |z| > R\]
which shows in particular that, as for classical polynomial Julia sets, there 
will be a \emph{basin of infinity at time $m$}, denoted $\Am$ on which all points escape locally uniformly to infinity under
iteration. Such a radius will be called an \emph{escape radius} for the bounds $d$, $K$, $M$. Note that the maximum principle shows that just as in 
the classical case (see \cite{CG}), there can be only one component on which $\infty$ is 
a limit function and so the sets $\Am$ are 
completely invariant in the sense given in Theorem \ref{completelyinvariant}.

The complement of $\Am$ is called the {\it filled Julia set} at time $m$ for
the sequence $\Pm$ and is denoted by $\Km$. The same 
argument using Montel's theorem as in the classical case then shows that 
$\partial \Km = \Jm$ (see \cite{CG}). In view of the existence of the escape radius above we have the following obvious result which we will use in proving our main result (see Theorem \ref{thetheorem} below).

\begin{proposition}
\label{VerySimpleProposition}
If $V$ is an open connected set for which there exists a subsequence $\{m_k\}_{k=1}^\infty$ such that the sequence of forward images $\{Q_{m_k}(V)\}_{k=1}^\infty$ is uniformly bounded, then $V$ is contained in a bounded Fatou component for $\Pm$. 
\end{proposition}

\section{The Schlicht Class} 

The \textit{Schlicht} class of functions, commonly denoted by $\mathcal{S}$, is the set of univalent functions defined on the unit disk such that, for all $f \in \calS$, we have $f(0)=0$ and $f'(0)=1$.  This is a classical class of functions for which many useful results are known. A common and useful technique is to use scaling or conformal mapping to apply results for $\mathcal{S}$ to arbitrary univalent functions defined on arbitrary domains (see for example Theorem \ref{abitmore} below).

\subsection{Statement of the Main Theorem}

We now give the statement of the main result of this paper.

\begin{theorem}
\label{thetheorem}
There exists a bounded sequence of quadratic polynomials $\{P_m \}_{m=1}^{\infty}$ and a bounded Fatou component $V$ for this sequence containing $0$ such that for all $f\in \mathcal{S}$ there exists a subsequence of iterates $\{Q_{m_k}\}_{k=1}^{\infty}$ which converges locally uniformly to $f$ on $V$.
\end{theorem}

The strength of this statement is that \textit{every} member of $\calS$ is a limit function on the \emph{same} Fatou component for a \textit{single} polynomial sequence. 

\begin{figure}
\label{siegeldiscjuliaset1}
\begin{center}
\scalebox{.32}{\includegraphics{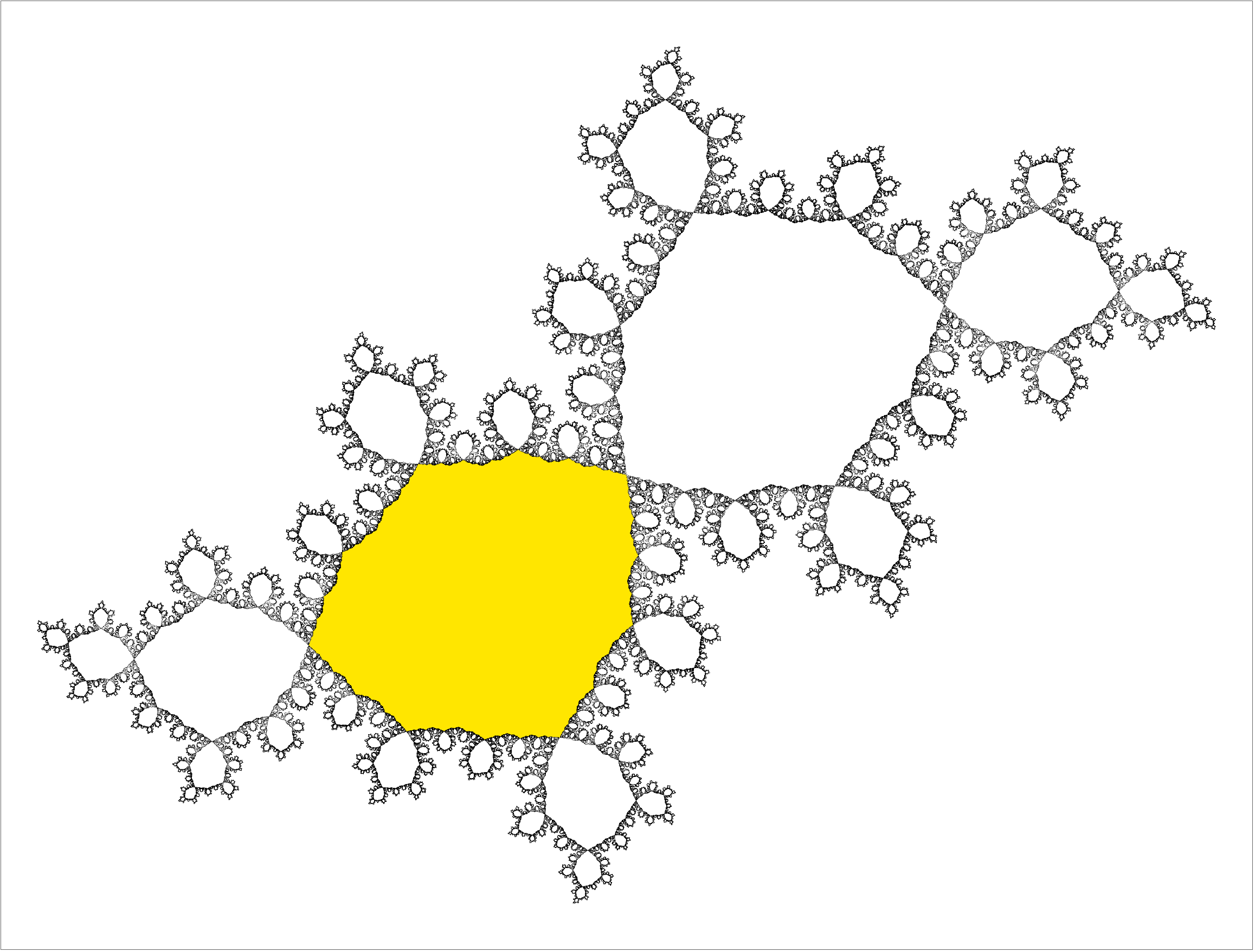}}
\caption{The filled Julia set $\calK_\lambda$ for $P_\lambda$ with Siegel disc highlighted}
\end{center}
\end{figure}

The proof relies on a scaled version of the polynomial \label{Plambda}$P_\lambda(z)= \lambda z(1-z)$ where $\lambda= e^{\frac{2\pi i (\sqrt{5}-1)}{2}}$.  As $P_\lambda$ is conjugate to an irrational rotation on its Siegel disk about $0$, which we denote by $U_\lambda$, we may find a subsequence of iterates which converges uniformly to the identity on compact subsets of $U_\lambda$.  We will rescale $P_\lambda$ so that the filled Julia set for the scaled version \label{overview1}$P$ of $P_\lambda$, is contained in a small Euclidean disc about $0$.  This is done so that, for any $f \in \mathcal S$, we can use the distortion theorems to control 
$|f'|$ on a relatively large hyperbolic disk inside $U$, the scaled version of the Siegel disc $U_\lambda$.  

The initial inspiration for this proof came from the concept of L\"owner chains (see e.g. \cite{CDG, Dur}), particularly the idea that a univalent function can be expressed as a composition of many functions which are close to the identity. Given   our remarks above about iterates of $P_\lambda$ which converge to the identity locally uniformly on $U_\lambda$, this encouraged us to think we might be able to approximate these univalent functions which are close to the identity in some way with polynomials and then compose these
polynomials to get an approximation of the desired univalent function on some suitable subset of $U_\lambda$, a principle which we like to summarize as `Do almost nothing and you can do almost anything'. As a matter of fact, there is now only one point in our proof where we make use of a L\"owner chain, although it is not necessary to know this: the interested reader can find this in the `up' section in the proof of Phase II (Lemma \ref{PhaseII}).

The proof of Theorem \ref{thetheorem} will follow from an inductive argument, and each step in the induction will be broken up into two phases:

\begin{itemize}
\item Phase I: Construct a bounded polynomial composition which approximates given functions from $\mathcal{S}$ on a subset of the unit disk.

\item Phase II: Construct a bounded polynomial composition which corrects the error of the previous Phase I composition to arbitrary accuracy on a slightly smaller subset.
\end{itemize}

Great care is needed to control the error in the approximations and to ensure that the domain loss that necessarily occurs in each Phase II correction eventually stabilizes, so that we are left with a non-empty region upon which the desired approximations hold.  

To create our polynomial approximations, we use what we call the Polynomial Implementation Lemma (Lemma \ref{PIL}).  \label{overview2}Suppose we want to approximate a given univalent function $f$ with a polynomial composition. Let $\calK$ be the filled Julia set for $P$ and let $\gamma$, $\Gamma$ be two Jordan curves enclosing $\calK$ with $\gamma$ lying inside $\Gamma$. In addition, we require that $f$ be defined inside and on a neighbourhood of $\gamma$, and that $f(\gamma)$ lie inside $\Gamma$.  We construct a homeomorphism of the sphere as follows:  define it to be $f$ inside $\gamma$, the identity outside $\Gamma$, and extend by interpolation to the region between $\gamma$ and $\Gamma$.  The homeomophism can be made quasiconformal, with non-zero dilation (possibly) only on the region between $\gamma$ and $\Gamma$.  If we then pull back with a high iterate of $P$ which is close to the identity, the support of the dilatation becomes small, which will eventually allow us to conclude that, when we straighten, we get a polynomial composition that approximates $f$ closely on a large compact subset of $U$.  In Phase I (Lemma \ref{PhaseI}), we then apply this process repeatedly to create a polynomial composition which approximates a finite set of functions from $\calS$. 

In Phase II (Lemma \ref{PhaseII}), we wish to correct the error from the Phase I composition.  This error is defined on a subset of the Siegel disk, but in order to apply the Polynomial Implementation Lemma to create a composition which corrects the error, we need the error to be defined on a region which contains $\calK$. 

To get around this, we conjugate so that the conjugated error is defined on a region which contains $\calK$.  This introduces a further problem, namely that we must now cancel the conjugacy with polynomial compositions.  A key element of the proof is viewing the expanding part of the conjugacy as a dilation in the correct conformal coordinates.  An inevitable loss of domain occurs in using these conformal coordinates, but we are, in the end, able to create a Phase II composition which corrects the error of the Phase I approximation on a (slightly smaller) compact subset of $U$.  What allows us to control the loss of domain, is firstly that, while some loss of domain is unavoidable, the accuracy of the Phase II correction is completely at our disposal. Secondly, one can show that the loss of domain will tend to zero as the size of the error to be corrected tends to zero (Lemma \ref{FittingLemma}).This eventually allows us to control loss of domain.  We then implement a fairly lengthly inductive argument to prove the theorem, getting better approximations to more functions in the Schlicht class with each stage in the induction, while ensuring that the region upon which the approximation holds does not shrink to nothing.

Theorem \ref{thetheorem} can be generalized somewhat to suitable normal families on arbitrary open sets. 

\begin{theorem}
\label{abitmore}
Let $\Omega \subset \C$ be open and let ${\mathcal N}$ be a locally bounded normal family of univalent functions on $\Omega$ all 
of whose limit functions are non-constant. Let $z_0 \in \Omega$. Then there exists a bounded sequence $\Pmt$ of quadratic polynomials and a bounded Fatou component $W$ for this sequence containing $z_0$ such that for all $f \in {\mathcal N}$ there exists a subsequence of iterates $\{\tilde Q_{m_k}\}_{k=1}^{\infty}$ which converges locally uniformly to $f$ on $W$.\end{theorem}

\section{Related Results}

In our proof we will make extensive use of the hyperbolic metric. This has two main advantages - conformal invariance and the fact that hyperbolic Riemann surfaces are infinitely large when measured using their hyperbolic metrics which allows one to neatly characterize relatively compact subsets using the external hyperbolic radius (see Definition \ref{intexthypraddef} in Section 2.2 below on the hyperbolic metric). An alternative approach is to try to do everything using the Euclidean metric. This requires among other things that in the analogue of the `up' portion of the proof of our Phase II (Lemma \ref{PhaseII}), we must ensure that 
the image of the Siegel disc under a dilation about the fixed point by a factor which is just larger than $1$ will cover the Siegel disc 
- in other words we need a Siegel disc which is star-shaped (about the fixed point). Fortunately, there is a result in the paper of Avila, Buff and Ch\'{e}ritat (Main Theorem in \cite{ABC}) which guarantees the existence of such Siegel discs. This led the authors to extensively investigate using this approach to prove a version of Theorem \ref{thetheorem} but, in practice, although this can probably be made to work, they found this to be at least as complicated as the proof outlined in the current manuscript.

Results on approximating a large class of analytic germs of diffeomorphisms were proved in the paper of Loray \cite{Loray}, particularly Th\'{e}or\`{e}me 3.2.3 in this work where he uses a pseudo-group induced by a non-solvable subgroup of diffeomorphisms to approximate all germs of conformal maps which send one prescribed point to another with only very mild restrictions.  Although we cannot rule out the possibility that these results could be used to obtain a version of our Theorems \ref{thetheorem} and \ref{abitmore}, this would be far from immediate. For example, pseudo-groups are closed under taking inverses (see D\'{e}finition 3.4.1 in \cite{Loray}). In our context, we can at best only approximate inverses, e.g. the suitable inverse branch of $P_\lambda$ on $U_\lambda$ which fixes $0$. Moreover, one would need to be able to compose many such approximations while still ensuring that the resulting composition would be close to the ideal version, as well as being defined on a set which was not too small. Thus, one would unavoidably require a complex bookkeeping scheme for tracking the sizes of errors and domains, which is a large part of what we need to concern ourselves with below.

Finally, in \cite{GT}, Gelfriech and Turaev show that an area-preserving two dimensional map with an elliptic periodic point can be renormalized so that the renormalized iterates are dense in the set of all real-analytic symplectic maps of a two-dimensional disk. However, this is clearly not as close to what we do as the two other cases mentioned above.

\section{Acknowledgements}
We wish to express our gratitude to Xavier Buff, Arnaud Ch\'eritat, and Pascale Roesch for their helpful comments and suggestions when the first author spent some time at the Universit\'{e} Paul Sabatier in Toulouse in 2016. We also wish to express our gratitude to Hiroki Sumi at Kyoto University for directing us to the work of Gelfriech and Turaev. Finally, we wish to thank Lo\"{\i}c Teyssier of the Universit\'{e} de Strasbourg for informing us about the work of Loray and helping us to determine how close Loray's results were to our own.

%% file: Background.tex
\chapter{Background}

We will now discuss some background which will be instrumental in proving Theorem \ref{thetheorem}.

\section{Classical Results on $\mathcal{S}$}

We now state some common results regarding the class $\mathcal{S}$.  These can be found in many texts, in particular \cite{CG}.  Before we state the first result, let us establish some notation.  Throughout, let $\D$ be the unit disk and let $\mathrm{D}(z,R)$ be the (open) Euclidean disk centered at $z$ of radius $R$.  The following is Theorem I.1.3 in \cite{CG}.  

\begin{theorem}
\label{Koebe}
(The Koebe one-quarter theorem) If $f \in \mathcal{S}$, then $f(\D)\supset \mathrm{D}(0,\frac{1}{4})$.  
\end{theorem}

Also of great importance are the well-known distortion theorems (Theorem I.1.6 in \cite{CG}).  

\begin{theorem}
\label{distortion}
(The distortion theorems):  If $f\in \calS$, then 
\begin{align*}
\frac{1-|z|}{(1+|z|)^3}\leq |f'(z)| \leq \frac{1+|z|}{(1-|z|)^3}, \\
\frac{|z|}{(1+|z|)^2}\leq |f(z)| \leq \frac{|z|}{(1-|z|)^2}. 
\end{align*}  
\end{theorem}

The above implies immediately that $\mathcal{S}$ is a normal family in view of Montel's theorem. More precisely, we have the following:

\begin{corollary}
\label{SisNormal}
The family $\mathcal{S}$ is normal, and the limit of any sequence in $\mathcal{S}$ belongs to $\mathcal{S}$.  
\end{corollary}

\section{The Hyperbolic Metric}

We will be using the hyperbolic metric to measure both the accuracy of our approximations and the loss of domain that occurs in each Phase II composition. We first establish some notation for hyperbolic discs. Let $R$ be a hyperbolic Riemann surface and let $\Delta_R(z,r)$ be the (open) hyperbolic disc in $R$ centered at $z$ of radius $r$.  If the domain is obvious in context, we may simply denote this disc $\Delta(z,r)$. Lastly, let $\mathrm{d}\rho_R$ represent the hyperbolic length element for $R$.

One of the key tools we will be using is the following relationship between the hyperbolic and Euclidean metrics (see \cite{CG} Theorem I.4.3).   If $D$ is a domain in $\mathbb{C}$ and $z\in D$, let $\delta_D(z)$ denote the Euclidean distance to $\partial D$.  

\begin{lemma}
\label{lemma4.3}
Let $D\subsetneq \mathbb{C}$ be a simply connected domain and let $z \in D$.  Then 
\begin{align*}
\frac{1}{2}\frac{|\mathrm{d}z|}{\delta_D(z)}\leq \mathrm{d}\rho_D(z) \leq 2 \frac{|\mathrm{d}z|}{\delta_D(z)}.
\end{align*}
\end{lemma}

We remark that there is also a more general version of this theorem for hyperbolic domains in $\mathbb{C}$ which are not necessarily simply connected (again see \cite{CG} Theorem I.4.3). However, for the purposes of this paper, we will consider only simply connected domains which are proper subsets of $\C$. The advantage of this is that there is always a unique geodesic segment joining any two distinct points, and we can use the length of this segment to measure hyperbolic distance. One immediate application of this lemma is the following which will be useful to us later in the proof of the induction (Lemma \ref{mediuminductionlemma}) leading up to the main result (Theorem \ref{thetheorem}).

\begin{lemma}
\label{net}
Let $K \subset \D$ be compact and let $\epsilon > 0$. We can then find a finite set $\{f_{i} \}_{i=1}^N  \subset \mathcal{S}$ such that, given $f \in \mathcal{S}$, there exists (at least one) $1 \le k \le N$ such that 
\[\sup_{z \in K} \rho_{\D}(f(z),f_k(z))<\epsilon.\] 
\end{lemma}

\begin{proof}
This follows immediately from the normality of $\mathcal S$ (Corollary \ref{SisNormal}), combined with Proposition VII.1.16 in \cite{Con}. 
\end{proof}

A set $\{f_{i} \}_{i=1}^N \in \mathcal{S}$ as above will be called an \emph{$\epsilon$-net} for $\mathcal S$ on $K$ or simply an  \emph{$\epsilon$-net} if the set $K$ is clear from the context.

Next, we will need a notion of internal and external hyperbolic radii, which is one of the crucial bookkeeping tools we will be using, especially for controlling loss of domain in Phase II.  

\begin{definition}
\label{intexthypraddef}	
Suppose $V$ is a hyperbolic Riemann surface, $v \in V$, and $X$ is a non-empty subset of $V$.  Define the {\rm external hyperbolic radius of $X$ in $V$ about $v$}, denoted $R^{ext}_{(V,v)}X$, by
$$R^{ext}_{(V,v)}X=\sup_{z \in  X} \rho_{V}(v,z).$$ 
If $v \in X$, we further define the {\rm internal hyperbolic radius of $X$ in $V$ about $v$}, denoted $R^{int}_{(V,v)}X$, by
$$R^{int}_{(V,v)}X = \inf_{z \in V \setminus X} \rho_{V}(v,z).$$  
If $v \in X$ and 
it happens that $R^{int}_{(V,v)}X=R^{ext}_{(V,v)}X$, we will call their common value the {\rm hyperbolic radius of $X$ in $V$ about $v$}, and denote it by $R_{(V,v)}X$.
\end{definition}

We remark that, for any $v \in V$,  if $X=V$, then $R^{int}_{(V,v)}X=R^{ext}_{(V,v)}X=\infty$.  Also, if $v \in X$ and $X \subsetneq V$, then $R^{int}_{(V,v)}X<\infty$. Indeed, let $w \in V \setminus X$.  Then 
\begin{align*}
R^{int}_{(V,v)}X&=\inf_{z \in V \setminus X} \rho_{V}(v,z) \\
&\leq\rho_{V}(v,w) \\
&<\infty.
\end{align*}

We also remark that the internal and external hyperbolic radii are increasing with respect to set-theoretic inclusion in the obvious way. Namely, if $\emptyset \ne X \subset Y$ are subsets of $V$, then $R^{ext}_{(V,v)}X \le R^{ext}_{(V,v)}Y$ while, if $v \in X$, we also have $R^{int}_{(V,v)}X \le R^{int}_{(V,v)}Y$.

The names `internal hyperbolic radius' and `external hyperbolic radius' are justified in view of the following lemma which is how they are often used in practice. 

\begin{lemma}
\label{anotherstupidfuckinglemma}
Let $V \subsetneq \C$ be a simply connected domain, $v \in V$, and $X$ be a non-empty subset of $V$. We then have the following:

\begin{enumerate}
\item If $0 < R^{ext}_{(V,v)}X < \infty$, then $X \subset \overline \Delta_V(v, R^{ext}_{(V,v)}X)$,

\item If $v \in X$ and $0 < R^{int}_{(V,v)}X < \infty$, then $R^{int}_{(V,v)}X = \sup\{r: \Delta_V(v, r) \subset X\}$ so that in particular
$\Delta_V(v, R^{int}_{(V,v)}X) \subset X$, 

\item If $v \in X$, then $R^{int}_{(V,v)}X \le R^{ext}_{(V,v)}X$. 

\end{enumerate}
\end{lemma}

\begin{proof}
\emph{(1)} follows immediately from the above definition for external hyperbolic radius. For \emph{(2)}, if we temporarily let $R:= \sup\{r: \Delta_V(v, r) \subset X\}$, then from the definition of internal hyperbolic radius, it follows easily that $V \setminus X \subset V \setminus \Delta_V(v, R^{int}_{(V,v)}X)$ from which we have that $R^{int}_{(V,v)}X \le R$. Note that since $R^{int}_{(V,v)}X > 0$, this means that $R > 0$ and the set of which we take the supremum to find $R$ must be non-empty. On the other hand, if we let $z \in V \setminus X$ (note that the requirement that $ R^{int}_{(V,v)}X < \infty$ ensures that we can always find such a point), then we must have that $\rho_V(v,x) \ge R$, and on taking an infimum over all such $x$, we have $R^{int}_{(V,v)}X \ge R > 0$ from which we obtain \emph{(2)}.\emph{(3)} then follows from \emph{(1)} and \emph{(2)} (the result being trivial in the cases where the external hyperbolic radius is infinite or the internal hyperbolic radius is zero) which completes the proof. 
\end{proof}

We remark that \emph{(2)} above illustrates how the internal hyperbolic radius $R^{int}_{(V,v)}X$ is effectively the radius of the largest disc about $v$ which lies inside $X$. The reason that we took $R^{int}_{(V,v)}X = \inf_{z \in V \setminus X} \rho_{V}(v,z)$ as our definition and not the alternative $\sup\{r: \Delta_V(v, r) \subset X\}$ is that this version still works, even if $\inf_{z \in V \setminus X} \rho_{V}(v,z)$ is zero or infinite. This lemma leads to the following handy corollary. 

\begin{corollary}
\label{HyperbolicAvoidance}
Suppose $V \subsetneq \C$ is a simply connected domain, $v \in V$, and that $X$, $Y$ are subsets of $V$, with $v \in Y$.  
\begin{enumerate}
\item If $R^{ext}_{(V,v)}X \le R^{int}_{(V,v)}Y$, then $\overline X \subset \overline Y$,  

\item If $R^{ext}_{(V,v)}X < R^{int}_{(V,v)}Y$, then $\overline X \subsetneq Y$.  
\end{enumerate}
\end{corollary}

We also have the following equivalent formulation for the internal and external hyperbolic radii which is often very useful in practice.  

\begin{lemma}
\label{equivalenthypradformulation}
Let $V \subsetneq \C$ be a simply connected domain, let $v \in V$, and let $X$ be a non-empty subset of $V$. We then have the following:

\begin{enumerate}
\item If $v \in X$, then $R^{int}_{(V,v)}X = \inf_{z \in (\partial X) \cap V} \rho_{V}(v,z),$

\item $R^{ext}_{(V,v)}X \ge \sup_{z \in (\partial X) \cap V} \rho_{V}(v,z).$
\end{enumerate}

If, in addition, $R^{ext}_{(V,v)}X < \infty$ or $X \subsetneq V$ and $\chat \setminus X$ is connected, we also have

\begin{enumerate}

\item[(3)] $R^{ext}_{(V,v)}X = \sup_{z \in (\partial X) \cap V} \rho_{V}(v,z).$
\end{enumerate}

In particular, the above holds if $X = U \subset V$ is a simply connected domain. 
\end{lemma}

Note that we can get strict inequality in \emph{(2)} above. For example, let $V = \D$, $v=0$, and let $X = \{z:  |z| \le \tfrac{1}{3}\} \cup \{z: \tfrac{2}{3} \le |z| < 1\}$. We leave the elementary details to the interested reader.

\begin{proof}
To prove  \emph{(1)}, we first observe that the result is trivial if  $R^{int}_{(V,v)}X = \infty$ which happens if and only if $X = V$. So suppose now that $X \subsetneq V$. Note that, in this case, $\partial X \cap V \ne \emptyset$, since otherwise $\interior X$ and $V \setminus \overline X$ would give a separation of the connected set $V$. 

Now let $z \in \partial X \cap V$ and pick $\epsilon > 0$. Since $z \in \partial X$ there exists $w \in V \setminus X$ with $\rho_V(z, w) < \epsilon$. By the triangle inequality $\rho_V(v, w) < \rho_V(v, z) + \epsilon$ and, on taking the infimum on the left hand side, 
$$ R^{int}_{(V,v)}X \le \rho_V(v, z) + \epsilon.$$
If we then take the infimum over all $z \in  \partial X \cap V$ on the right hand side and let $\epsilon$ tend to $0$, we then obtain that 
$$R^{int}_{(V,v)}X \le \inf_{z \in (\partial X) \cap V} \rho_{V}(v,z).$$
Now we show $R^{int}_{(V,v)}X \geq \inf_{z \in (\partial X) \cap V} \rho_{V}(v,z)$.  Take a point $w \in V \setminus X$ and connect $v$ to $w$ with a geodesic segment $\gamma$ in $V$. $\gamma$ must then meet $\partial X$ since otherwise $\interior X$ and $V \setminus \overline X$ would give a separation of the connected set $[\gamma]$ (the track of $\gamma$). So let $z_0 \in \partial X \cap [\gamma]$. Clearly
\begin{align*}
\rho_V(v,z_0)\leq \rho_V(v,w),
\end{align*}
so 
\begin{align*}
\inf_{z \in (\partial X) \cap V}\rho_{V}(v,z) \leq \rho_V(v,w),
\end{align*}
and thus 
\begin{align*}
\inf_{z \in (\partial X) \cap V}\rho_{V}(v,z) \leq R^{int}_{(V,v)}X.
\end{align*}
This completes the proof of \emph{(1)}. 

To prove  \emph{(2)}, we first consider the case when $\sup_{z \in (\partial X) \cap V} \rho_{V}(v,z)=\infty$. Note that, since the supremum of the empty set is minus infinity, this in particular implies that $\partial X \cap V \neq \emptyset$. Thus we can find a sequence $\{z_n \}\in (\partial X)\cap V$ such that $\rho_{V}(v,z_n)\rightarrow \infty$.  For each $z_n$, choose $x_n\in X$ such that $\rho_{V}(z_n,x_n)\leq 1$.  Then $\rho_{V}(v,x_n)\rightarrow \infty$ by the reverse triangle inequality, which shows $R^{ext}_{(V,v)}X=\infty$ so that we have equality. 

Now consider the case when $\sup_{z \in (\partial X) \cap V} \rho_{V}(v,z)<\infty$. The result is trivial if this supremum is minus infinity, so again we can assume that $\partial X \cap V \neq \emptyset$. Similarly to above, we can take a sequence $\{z_n \}\in (\partial X)\cap V$ for which $\rho_V(v,z_n)\rightarrow \sup_{z \in (\partial X) \cap V} \rho_{V}(v,z)$.  Then take a sequence $\{x_n \}\in X$ such that $\rho_V(x_n,z_n)<\frac{1}{n}$.  By definition of the external hyperbolic radius, we must have 
\begin{align*}
\rho_V(v,x_n)\leq R^{ext}_{(V,u)}X, 
\end{align*}  
and since $\rho_V(x_n,z_n)<\frac{1}{n}$, by the reverse triangle inequality, on letting $n$ tend to infinity,
\begin{align*}
\sup_{z \in (\partial X) \cap V} \rho_{V}(v,z) \leq R^{ext}_{(V,v)}X
\end{align*}
which proves \emph{(2)} as desired.

Now we show that, under the additional assumption that $R^{ext}_{(V,v)}X < \infty$ or $X \subsetneq V$ and $\chat \setminus X$ is connected, 
$\sup_{z \in (\partial X) \cap V} \rho_{V}(v,z) \geq R^{ext}_{(V,v)}X$ from which \emph{(3)} follows.  Assume first that $R^{ext}_{(V,v)}X < \infty$ and let $\{x_n \}\in X$ be a sequence in $X$ such that $\rho_V(v, x_n) \to R^{ext}_{(V,v)}X$ as $n$ tends to infinity  (recall that we have assumed $X \neq \emptyset$ so that the set over which we are taking our supremum to obtain the external hyperbolic radius is non-empty). Note also that $R^{ext}_{(V,v)}X = 0$ if and only if $X = \partial X =  \{v\}$ in which case the result is trivial, so we can assume that $\rho_V(v, x_n)> 0$ for each $n$. For each $n$ let $\gamma_n$ be the unique hyperbolic geodesic in $V$ which passes through $v$ and $x_n$. Then there must be a point $z_n$ (which may possibly be $x_n$ itself) on $\gamma_n \cap \partial X$ which does not lie on the same side of $x_n$ as $v$ since otherwise the portion of $\gamma_n$ on the same side of $v$ as $x_n$ and which runs from $x_n$ to $\partial V$ would be separated by the open sets $\interior X$ and $V \setminus \overline X$. However, this is impossible since $x_n \in X$ while $R^{ext}_{(V,v)}X < \infty$ which forces $\gamma_n$ to eventually leave $X$ (in both directions). It then follows that for each $n$ we have that 
$$\rho_V(v, z_n) \ge \rho_V(v, x_n)$$ so that 
$$\sup_{z \in (\partial X) \cap V} \rho_V(v, z) \ge \rho_V(v, x_n)$$
and the desired conclusion then follows on letting $n$ tend to infinity. 

Now suppose that $X \subsetneq V$ and $\chat \setminus X$ is connected. We observe that, since $V$ is connected we must have $\partial X \cap V \ne \emptyset$, while if $X = U$ is a simply connected domain, then $\chat \setminus X$ is automatically connected (e.g. \cite{New} VI.4.1 or \cite{Con} Theorem VIII.2.2).

In view of \emph{(2)} above, \emph{(3)} holds if 
$\sup_{z \in (\partial X) \cap V} \rho_{V}(v,z) = \infty$, so assume from now on that $\sup_{z \in (\partial X) \cap V} \rho_{V}(v,z) < \infty$ and note that $\partial X \cap V \ne \emptyset$ implies that this supremum will be non-negative so that we can set $\rho:=\sup_{z \in (\partial X) \cap V} \rho_{V}(v,z)$. Note that, if $\rho = 0$, then $V \setminus \{v\} \cap \partial X = \emptyset$ and, since $V \setminus \{v\}$ is connected, then either $V \setminus \{v\} \subset \interior X \subset X$ or $V \setminus \{v\} \subset V \setminus \overline X \subset V \setminus X$. In the first case, $X = V \setminus \{v\}$ in which case one checks easily that \emph{(3)} fails. However, we can rule out this case since $\chat \setminus X = \{v\} \cup \chat \setminus V$ is disconnected which violates our hypothesis that $\chat \setminus X$ be connected. In the second case, we have $X = \{v\}$ in which case one easily checks that \emph{(3)} holds. Thus we can assume from now on that $\rho > 0$.

\begin{claim}
$X \subset {\overline \Delta_V}(v,\rho)$. 
\end{claim}
\begin{claimproof}Suppose not.  Then there exists $x \in X$ such that $\rho_V(v, x)>\rho$.  Set ${\tilde \rho}:=\rho_V(v,x) > \rho$ and define $C$ to be the hyperbolic circle of radius ${\tilde \rho}$ about $v$ with respect to the hyperbolic metric of $V$.  Then we have $C \cap \partial X = \emptyset$ since, for all $z \in \partial X \cap V$, by definition we have $\rho_V(v,z) \le \rho < \tilde \rho$. Thus $x \in \interior X$.

Now we have $x \in C$.  We next argue that each point of $C$ must lie in $X$.  Suppose $z$ is another point on $C$ such that $z \notin X$.  Then $z$ would be in $V\setminus X$.  As $C \cap \partial X = \emptyset$, we have that $z \in V \setminus {\overline X}= \interior (V \setminus X)$.  But this is impossible as $\interior X$ and $\interior (V \setminus X)$ would then form a separation of the connected set $C$.  Thus $C \subset X$ and $C$ induces a separation of $\chat \setminus X$.  Indeed, since $\rho < \tilde \rho$, $\partial X$ is inside the Jordan curve $C$ and hence there are points of $\chat \setminus X$ inside $C$. On the other hand, 
 $\infty \in \chat \setminus V \subset \chat \setminus X$ lies outside $C$. This contradicts our assumption that $\chat \setminus X$ is connected. \end{claimproof}

Immediately from the above claim, we see that $\rho = \sup_{z \in (\partial X) \cap V} \rho_{V}(v,z) \geq R^{ext}_{(V,v)}X$, and thus, in the case where $X \subsetneq V$ and $\chat \setminus X$ is connected,
\begin{align*}
\sup_{z \in (\partial X) \cap V} \rho_{V}(v,z) = R^{ext}_{(V,v)}X
\end{align*}
which proves \emph{(3)} as desired.  \end{proof}

We will require the following elementary definition from metric spaces:

\begin{definition}
\label{HypDist}
Suppose $R$ is a hyperbolic Riemann surface and that $A$ and $B$ are non-empty subsets of $R$.  For $z \in R$, we define 
\begin{align*}
\rho_{R}(z, B)=\inf_{w\in B}\rho_R(z,w)
\end{align*}
and 
\begin{align*}
\rho_{R}(A,B)=\inf_{z\in A}\rho_R(z,B).
\end{align*}
\end{definition}

We say that a subset $X$ of a simply connected domain $V \subsetneq \C$ is \emph{hyperbolically convex} if, for every $z, w \in X$, the geodesic segment $\gamma_{z,w}$ from $z$ to $w$ lies inside $X$ (this is the same as the definition given in Section 2. of \cite{MM1}). We then have the following elementary but useful lemma. 

\begin{lemma}
\label{stupidfuckinglemma}
(The hyperbolic convexity lemma) Let $V \subsetneq \C$ be a simply connected domain. Then any hyperbolic disc $\Delta_V(z, R)$ is hyperbolically convex with respect to the hyperbolic metric of $V$. 
\end{lemma}

\begin{proof}
Let $a$, $b$ be two points in $\Delta_V(z, R)$. Using conformal invariance we can apply a suitably chosen Riemann map from $V$ to the unit disc $\D$ so that, without loss of generality, we can assume that $a=0$ while $b$ is on the positive real axis whence the shortest geodesic segment from $a$ to $b$ is the line segment $[0,b]$ on the positive real axis. On the other hand, the disc 
$\Delta_\D(z, R)$ is a round disc ${\mathrm D}(w, r)$ for some $w \in D$ and $r \in (0,1)$ which is therefore convex (with respect to the Euclidean metric) and the result follows. 
\end{proof}

Ordinary derivatives are useful for estimating how points get moved apart by applying functions when using the Euclidean metric.  In our case we will need a notion of a derivative taken with respect to the hyperbolic metric.  

Let $R,S$ be hyperbolic Riemann surfaces with metrics 
\begin{align*}
\mathrm{d}\rho_R &= \sigma_R(z)|\mathrm{d}z|, \\
\mathrm{d}\rho_S &= \sigma_S(z)|\mathrm{d}z|, 
\end{align*}
respectively and let $\ell_R$, $\ell_S$ denote the hyperbolic length in $R$, $S$ respectively. Let $X \subset R$ and let 
$f$ be defined and analytic on an open set containing $X$ with $f(X) \subset S$. For $z \in X$, define the \emph{hyperbolic derivative:}
\begin{align}
\label{hypderiv}
f_{R,S}^{\natural}(z):= f'(z)\frac{\sigma_S(f(z))}{\sigma_R(z)}
\end{align}

see e.g. the differential operation $D_{h1}$ defined in Section 2. of \cite{MM1} and Section 3 of \cite{MM2}.  Note that the hyperbolic derivative satisfies the chain rule, i.e. if $R$, $S$, $T$, are hyperbolic Riemann surfaces with $g$ defined and analytic on an open set containing $X \subset R$ and $f$ defined and analytic on an open set containing $Y \subset S$ with $f(X) \subset Y$, then, on the set $X$, 
\begin{align}
\label{chainrule}
(f\circ g)_{R,T}^{\natural}=(f^{\natural}_{S,T}\circ g)\cdot g_{R,S}^{\natural}.
\end{align}

We also have a version of conformal invariance which is essentially Theorem 7.1.1 in \cite{KL} or which the interested reader can simply deduce from the formula for the hyperbolic metric using a universal covering map from the disc (see, e.g. on page 12 in \cite{CG}), namely:
\begin{align}
\label{isometry}
\mbox{If} \: f: R \mapsto S\:  \mbox{is a covering map, then} \: |f^\natural| = 1 \: \mbox{on} \: R. 
\end{align}
We observe that the above is bascially another way of rephrasing part of the Schwarz lemma for the hyperbolic metric (e.g. \cite{CG} Theorem I.4.1) where we have an isometry of hyperbolic metrics if and only if the mapping from one Riemann surface to the other lifts to an automorphism of the unit disc. The main utility of the hyperbolic derivative for us will be via the hyperbolic metric version of the standard M-L estimates for line integrals (see Lemma \ref{hyperbolicML} below). First, however, we make one more definition.

Let $X$ be a non-empty subset of $R$ and let $f$ be defined and analytic on an open set containing $X$ with $f(X) \subset S$. Define the \emph{hyperbolic Lipschitz bound of $f$ on $X$} as 
\begin{align*}
\|f_{R,S}^{\natural} \|_X :=\sup_{z \in X}|f_{R,S}^{\natural}(z)|.
\end{align*} 
We recall that for any two points $z$, $w$ in $R$, the hyperbolic distance $\rho_R(z,w)$ is the same as the length of a shortest geodesic segment in $R$ joining $z$ to $w$ (see e.g. Theorems 7.1.2 and 7.2.3 in \cite{KL}).

\begin{lemma}
\label{hyperbolicML}
(Hyperbolic M-L estimates) Suppose $R,S$ are hyperbolic Riemann Surfaces. Let  $\gamma$ be a piecewise smooth curve in $R$ and let $f$ be holomorphic on an open neighbourhood of $[ \gamma ] $ and map this neighbourhood inside $S$
with $|f_{R,S}^{\natural}|\leq M$ on $[ \gamma ]$.  Then 
$$\ell_S(f(\gamma)) \leq M \ell_R (\gamma).$$

In particular, if $z, w \in R$ and $\gamma$ is a shortest hyperbolic geodesic segment connecting $z$ and $w$ and $|f_{R,S}^{\natural}|\leq M$ on $[ \gamma ] $, then
$$\rho_S(f(z), f(w))  \leq M\rho_R(z,w).$$
\end{lemma}

\begin{proof} For the first part, if $\gamma: [a,b] \to R$, we calculate 
\begin{align*}
\ell_S(f(\gamma)) &=\int_{f(\gamma)}\mathrm{d}\rho_S \\
&=\int_{a}^{b} \sigma_S(f(\gamma(t))) \cdot |f'(\gamma(t))| \cdot |\gamma'(t)|\, \mathrm{d}t \\
&=\int_{a}^{b}|f^{\natural}(\gamma(t))|\cdot \sigma_R(\gamma(t)) \cdot |\gamma'(t)|\, \mathrm{d}t \\
&=\int_{\gamma}|f^{\natural}|\, \mathrm{d}\rho_R \\
&\leq M \int_{\gamma} \mathrm{d}\rho_R \\
&= M \ell_R(\gamma).
\end{align*}
The second part then follows immediately from this and the facts that by Theorems 7.1.2 and 7.2.3 in \cite{KL}, $\rho_R(z,w)$ is equal to the length of the shortest geodesic segment in $R$ joining $z$ and $w$ while $f(\gamma)$ is at least as long in $S$ as the distance between $f(z)$ and $f(w)$.
\end{proof}

In this paper, we will be working with hyperbolic derivatives only for mappings which map a subset of $U$ to $U$ where $U$ is a suitably scaled version of the Siegel disc $U_\lambda$ introduced in the last chapter and 
where we are obviously using the hyperbolic density of $U$ in the definition \eqref{hypderiv} above. For the sake of readability, from now on, we will suppress the subscripts and simply write $f^\natural$ instead of $f_{U,U}^\natural$ for derivatives taken with respect to the hyperbolic metric of $U$.

\section{Star-Shaped Domains}

Recall that a domain $D \subset \C$ is said to be \emph{star-shaped} with respect to some point $z_0 \in D$ if, for every point $z \in D$, $[z_0, z] \subset D$ where $[z_0, z]$ denotes the Euclidean line segment  from $z_0$ to $z$. We have the following classical result which will be of use to us later in the `up' section of the proof of Phase II (Lemma \ref{PhaseII}).   

\begin{lemma}[\cite{Dur} Corollary to Theorem 3.6] \label{starshapedDuren}
For every radius $r \le \rho := \tanh(\tfrac{\pi}{4})= .655\ldots\,$, each function $f \in {\mathcal S}$ maps the Euclidean disc ${\mathrm D}(0,r)$ to a domain which is starlike with respect to the origin. This is false for every $r > \rho$.
\end{lemma}

Since this value of $r$ corresponds via the formula $\rho_\D(0,z) = \log \tfrac{1 + |z|}{1-|z|}$ to a hyperbolic radius about $0$ of exactly $\tfrac{\pi}{2}$, we have the following easy consequence. 

\begin{lemma} \label{starshaped}
If $f$ is univalent on $\D$, $z_0 \in \D$, $r  \le  \tfrac{\pi}{2}$, and $\Delta_\D(z_0, r)$ denotes the hyperbolic disc in $\D$ of radius $r$ about $z_0$, then the image
$f(\Delta_\D(z_0, r))$ is star-shaped with respect to $f(z_0)$.
\end{lemma}

The important property of star-shaped domains for us is that, if we dilate such a domain about its centre point by an amount greater than $1$, then the enlarged domain will cover the original. More precisely, if $X$ is star-shaped with respect to $z_0$, $r > 1$ and we let $rX$ be the domain $rX := \{z: (z - z_0)/r + z_0 \in X\}$, then $X \subset rX$. Again, this is something we will make use of in the `up' portion of the proof of Phase II (Lemma \ref{PhaseII}).

\section{The \Car Topology}

The \Car topology is a topology on pointed domains,  which consist of a domain and a marked point of the domain which is referred to as the base point.  In \cite{Car} Constantin \Car defined a suitable topology for simply connected pointed domains for which convergence in this topology is equivalent to the convergence of suitably normalized inverse Riemann maps.  The work was then extended in an appropriate sense to hyperbolic domains by Adam Epstein in his Ph.D thesis \cite{Ep}.  This work was expanded upon further still by the first author \cite{Com10,Com11}.  This is a supremely useful tool in non-autonomous iteration where the domains on which certain functions are defined may vary.  We follow \cite{Com10} for the following discussion. Recall that a \emph{pointed domain} 
is an ordered pair $(U,u)$ consisting of an open  
connected subset $U$ of $\chat$, (possibly equal to $\chat$ itself) and a point $u$ in 
$U$. 

\begin{definition}
\label{CarConv}
We say that $\Umum \to \Uu$ in the \Car topology if:

\begin{enumerate}

\item $u_m \to u$ in the spherical topology,

\item for all compact sets $K \subset U$, $K \subset U_m$ for all but finitely many $m$,

\item for any \emph{connected} (spherically) open set $N$ containing $u$, if $N \subset U_m$ for
infinitely many $m$, then $N \subset U$.

\end{enumerate}
\end{definition}

We also wish to consider the degenerate case where $U = \{u\}$. In this 
case, condition \emph{(2)} is omitted ($U$ has no interior of which we can take
compact subsets) while condition \emph{(3)} becomes
\emph{
\begin{enumerate}
\setcounter{enumi}{2}
\item for any \emph{connected} (spherically) open set $N$ containing $u$, 
$N$ is contained in at most finitely many of the sets $U_m$.
\end{enumerate}}

Convergence in the \Car topology can also be described using the \textit{\Car kernel}.  Originally defined by \Car himself in \cite{Car}, one first requires that $u_m \to u$ in the spherical topology.  If there is no open set containing $u$ which is contained in the intersection of all but finitely many of the sets $U_m$, then one defines the kernel of the sequence $\{(U_m,u_m) \}_{m=1}^{\infty}$ to be $\{u \}$.  Otherwise one defines the \Car kernel as the largest domain $U$ containing $u$ with the property \emph{(2)} above.  It is easy to check that a largest domain does indeed exist.  \Car convergence can also be described in terms of the Hausdorff topology.  We have the following theorem in \cite{Com10}.

\begin{theorem}
\label{EquivCarConv}
Let $\{(U_m,u_m) \}_{m=1}^{\infty}$ be a sequence of pointed domains and $(U,u)$ be another pointed domain where we allow the possibility that $(U,u)=(\{ u\},u)$.  Then the following are equivalent:

\begin{enumerate}
\item $\Umum \to \Uu$;

\item $u_m \to u$ in the spherical topology and $\{(U_m,u_m) \}_{m=1}^{\infty}$ has \Car kernel $U$ as does every subsequence;

\item $u_m \to u$ in the spherical topology and, for any subsequence where the complements of the sets $U_m$ converge in the Hausdorff topology (with respect to the spherical metric), $U$ corresponds with the connected component of the complement of the Hausdorff limit which contains $u$ (this component being empty in the degenerate case $U=\{ u\}$).
\end{enumerate}
\end{theorem}

Of particular use to us will be the following theorem in \cite{Com10} regarding the equivalence of \Car convergence and the local uniform convergence of suitably normalized covering maps, most of which was proved by Adam Epstein in his PhD thesis \cite{Ep}:

\begin{theorem}
\label{theorem1.2}
Let $\{(U_m,u_m) \}_{m\geq 1}$ be a sequence of pointed hyperbolic domains and for each $m$ let $\pi_m$ be the unique normalized covering map from $\D$ to $U_m$ satisfying $\pi_m(0)=0$, $\pi_m'(0)>0$. 

Then $(U_m,u_m)$ converges in the \Car topology to another pointed hyperbolic domain $(U,u)$ if and only if the mappings $\pi_m$ converge with respect to the spherical metric uniformly on compact subsets of $\D$ to the covering map $\pi$ from $\D$ to $U$ satisfying $\pi(0)=u$, $\pi'(0)>0$.  

In addition, in the case of convergence, if $D$ is a simply connected subset of $U$ and $v \in D$, then locally defined branches $\omega_m$ of $\pi_m^{-1}$ on $D$ for which $\omega_m(v)$ converges to a point in $\D$ will converge locally uniformly with respect to the spherical metric on $D$ to a uniquely defined branch $\omega$ of $\pi^{-1}$. 

Finally, if $\pi_m$ converges with respect to the spherical topology locally uniformly on $\D$ to the constant function $u$, then $(U_m,u_m)$ converges to $(\{u \},u)$.  
\end{theorem}

%% file: PIL.tex
\chapter{The Polynomial Implementation Lemma}

\section{Setup}

\label{IntrotoPIL}
Let $\Omega,\Omega'\subset \C$ be bounded Jordan domains with analytic boundary curves $\gamma$ and $\Gamma$, respectively, such that ${\overline \Omega} \subset \Omega'$. By making a translation if necessary, we can assume without loss of generality that $0 \in \Omega$ so that $\gamma$ then separates $0$ from $\infty$. Suppose $f$ is analytic and injective on a neighborhood of ${\overline \Omega}$ such that $f(\gamma)$ is still inside $\Gamma$. Let $A=\Omega' \setminus {\overline \Omega}$ be the conformal annulus bounded by $\gamma$ and $\Gamma$ and let $\tilde A$ be the conformal annulus bounded by $f(\gamma)$ and $\Gamma$.  Define 
\[  F(z)= \left\{
\begin{array}{ll}
     f(z)  & z\in {\overline \Omega} \\
     z & z \in \chat \setminus \Omega' \\
\end{array} .
\right. \]
 
We wish to extend $F$ to a quasiconformal homeomorphism of $\chat$. To do this, the main tool we use will be a lemma of Lehto \cite{Leh} which allows us to define $F$ in the `missing' region between $\Omega$ and $\chat \setminus \Omega'$. First, however, we need to gather some terminology.

Recall that in \cite{New}, a Jordan curve $C$ in the plane (parametrized on the unit circle $\T$) is said to be \emph{positively oriented} if the algebraic number of times a ray from the bounded complementary domain to the unbounded complementary domain crosses the curve is $1$ or, equivalently, the winding number of the curve about points in its bounded complementary region is also $1$ (the reader is referred to the discussion on Page 188-194 of \cite{New}). 

Following the proof of Theorem VII.11.1, Newman goes on to define a homeomorphism $g$ defined on $\C$ to be \emph{orientation-preserving} or \emph{sense-preserving} if it preserves the orientation of all simple closed curves. Lehto and Virtanen adopt Newman's definitions in their text on quasiconformal mappings \cite{LV} and they have a related and more general definition of orientation-preserving maps defined on an arbitrary plane domain $G$ where $g$ is said to be \emph{orientation-preserving} on $G$ if the orientation of the boundary of every Jordan domain $D$ with $\overline D \subset G$ is preserved (\cite{LV} page 9). 

Lehto and Virtanen also introduce the concept of the \emph{orientation of a Jordan curve $C$ with respect to one of its complementary domains} $G$ (\cite{LV} page 8). Let $C(z): \T \mapsto C$ be a parametrization of $C$ which defines its orientation and let $\Phi$ be a 
M\"obius transformation which maps $G$ to the bounded component of the complement of $\Phi(C)$ such that $0 \in \Phi(G)$. $C$ is then said to be \emph{positively oriented with respect to} $G$ if the argument of $\Phi \circ C(t)$ increases by $2\pi$ as one traverses $\T$ anticlockwise. Using this, if $G$ is an $n$-connected domain whose boundary consists of $n$ disjoint Jordan curves (what Lehto and Virtanen on page 12 of \cite{LV} refer to as \emph{free boundary curves}), it is easy to apply the above definition to define the orientation of $G$ with respect to each curve which comprises $\partial G$. 

Recall that in Lehto's paper \cite{Leh}, he considers a conformal annulus (ring domain) $D \subset \chat$ bounded by two Jordan curves $C_1$ and $C_2$. If $\phi$ is a homeomorphism of $C_1 \cup C_2$ into the plane, then the curves $\phi(C_1)$,  $\phi(C_2)$ will bound another conformal annulus which we call $D'$. If, under the mapping $\phi$, the positive orientations of $C_1$ and $C_2$ with respect to $D$ correspond to the positive orientations of $\phi(C_1)$ and $\phi(C_2)$ with respect to $D'$, then $\phi$ is called an \emph{admissible boundary function} for $D$. 

\begin{lemma}[Lehto \cite{Leh}]
\label{LehtoLemma}
Let $D$ be a conformal annulus in $\chat$ bounded by the Jordan curves $C_1$ and $C_2$ and let $w_h: \chat \mapsto \chat$, $h = 1, 2$ be quasiconformal mappings such that the restrictions of $w_h$ to $C_h$, $h=1,2$ constitute an admissible boundary function for $D$. Then there exists a quasiconformal mapping $w$ of $D$ such that $w(z) = w_h(z)$ for $z \in [C_h]$, $h=1,2$ (where for each $h$ $[C_h]$ is the track of the curve $C_h$). 
\end{lemma}

Applying this result to our situation, we have the following.

\begin{lemma}
\label{qcextension}
For $\Omega$, $\Omega'$, $f$, $F$ as above, we can extend the mapping $F$ above to a quasiconformal homeomorphism of $\chat$. 
\end{lemma}

\begin{proof} In order to apply Lehto's lemma above, we need to verify two things: firstly that $f$ (and the identity) can be extended as quasiconformal mappings from $\chat$ to itself  and second that we have an admissible pair of mappings on $\partial A = \partial (\Omega' \setminus \overline \Omega)$ according to Lehto's definition given above. 

First note that, in view of the argument principle, $f$, being univalent, is an orientation-preserving mapping on a neighbourhood of $\overline \Omega$. Using Satz II.8.1 in \cite{LV}, $f$ (and trivially the identity) can be extended as a quasiconformal mapping of $\C$ to itself. Using either Theorem VII.11.1 in \cite{New} or the Orientierungssatz on page 9 of \cite{LV}, the above extension can be easily extended to an orientation-preserving homeomorphism of $\chat$ which is then readily seen to be a quasiconformal mapping of $\chat$ to itself as follows easily from Satz I.8.1 in \cite{LV}.

Both $f$ and the identity preserve the positive orientations of $\gamma$ and $\Gamma$, respectively. In addition, since $f(\gamma)$ lies inside $\Gamma$, it follows that the orientations of $\gamma$ and $\Gamma$ with respect to $A$ are the same as those of $\tilde A$. To be precise, let $\gamma$ be positively oriented with respect to $A$ and let
$\tfrac{1}{A}$, $\tfrac{1}{\tilde A - f(0)}$ denote the images of $A$, $\tilde A$ respectively under $\tfrac{1}{z}$, $\tfrac{1}{z - f(0)}$ respectively. Since $A$ lies in the unbounded complementary component of $\gamma$, it follows from the above definition of the orientation of a boundary curve for a domain that the winding number of $\frac{1}{\gamma}$ about points of $\tfrac{1}{A}$ is $1$ so that the winding number of $\gamma$ about $0$ (which lies inside $\gamma$) is $-1$. By the argument principle $f(0)$ lies inside $f(\gamma)$ and, since $f$ is orientation-preserving, the winding number of $f(\gamma)$ about $f(0)$ is also $-1$. 

A simple calculation then shows that the winding number of $\tfrac{1}{f(\gamma) - f(0)}$ about $0$ and thus also about points in $\tfrac{1}{\tilde A - f(0)}$ is also $1$. This shows that $f$ and thus $F$ preserve the positive orientations of $\gamma$, $f(\gamma)$ with respect to $A$ and $\tilde A$ respectively. Since $F$ is the identity on $|\Gamma|$, it trivially preserves the positive orientation of $\Gamma$ with respect to $A$ and $\tilde A$ (both of which lie inside $\Gamma$) and, with this, we have shown the hypotheses of Lemma \ref{LehtoLemma} above from \cite{Leh} are met.

Lemma \ref{LehtoLemma}  now allows us to extend $F$ to a quasiconformal mapping on the conformal annulus $A = \Omega' \setminus \overline \Omega$ such that this extension agrees with the original values of $F$ on the boundary and maps $A$ to $\tilde A$. We can then use Satz I.8.3 in \cite{LV} on the removeability of analytic arcs or, remembering that 
$f$ is defined on a neighbourhood of $\overline \Omega$ while the identity is defined on all of $\chat$, twice invoke Rickman's Lemma (e.g. \cite{DH2} Lemma 2)  to conclude that the resulting homeomorphism of $\chat$ is quasiconformal.  \end{proof}

We can summarize the above in the following useful definition.

\begin{definition}
\label{admissiblepair}
If $f$, $F$, $\gamma$, $\Gamma$, and $A$ are all as above, with $F$ an admissable boundary function for $A$, we will say that $(f,Id)$ is an {\rm admissible pair} on $(\gamma,\Gamma)$. 
\end{definition}

Recall we have $P_\lambda= \lambda z(1-z)$ where $\lambda= e^{\frac{2\pi i (\sqrt{5}-1)}{2}}$.  Let $\calK_\lambda$ be the filled Julia set for $P_\lambda$, and let $U_\lambda$ be the corresponding Siegel disc containing $0$.  Let $\kappa \ge 1$ and set $P=P_{\kappa}=\frac{1}{\kappa}P_\lambda(\kappa z) =\lambda z - \lambda \kappa z^2$.  Then, if $\calK$ is the filled Julia set for $P$, we have $\calK\subset D(0,\frac{2}{\kappa})$.  Let $U$ be the Siegel disk for $P$ and note that $U= \{z \in \C \text{ : } z=\frac{w}{\kappa}\text{ for some } w\in U_\lambda\}$. Now choose the Jordan domains $\Omega$, $\Omega'$ above such that $\calK \subset \Omega \subset {\overline \Omega} \subset \Omega' \subset \mathrm{D}(0,\frac{2}{\kappa})$, where from above $\frac{2}{\kappa}<2$ is an escape radius for $P$. \label{RestrictionforPIL}

Let $(f,Id)$ be an admissable pair where $f$, $F$, $\gamma$, $\Gamma$, and $A$ are all as above. In view of Lemma \ref{qcextension}, $F$ can be extended to a quasiconformal homeomorphishm of $\chat$ 
and we let $\mu_F$ denote the complex dilatation of $F$. Next let $N \in \N$ and, for each $0 \le m \le N$, set $\mu_m^N:= (P^{N-m})^*\mu_F$ i.e. $\mu_m^N(z) = \mu_{F\circ P^{N-m}}(z)$. Let $\phi_N^N:=F$ and, for $0\leq m \leq N-1$, let $\phi_m^N$ be the unique normalized solution of the Beltrami equation for $\mu_m^N$ which satisfies $\phi_m^N(z) = z + \mathcal{O}(\frac{1}{|z|})$ near $\infty$ (see e.g. Theorem I.7.4 in \cite{CG}).  For $0 \leq m \leq N$, let 
$${\tilde P}_m^N(z)= \phi_m^N \circ P \circ (\phi_{m-1}^N)^{-1}(z).$$
Then for each $m$, ${\tilde P}_m^N$ is an analytic degree 2 branched cover of $\C$ which has a double pole at $\infty$ and no other poles.  Thus ${\tilde P}_m^N$ is a quadratic polynomial and the fact that each $\phi_m^N$ is tangent to the identity at $\infty$ ensures that the the leading coefficient of ${\tilde P}_m^N$ is $-\lambda \kappa$ and thus has absolute value $\kappa$. Let $\alpha_m^N:=\phi_m^N(0)$.  Since the dilatation of $\phi_m^N$ is zero on $\chat \setminus {\overline {\mathrm D}}(0,\frac{2}{\kappa})$, we know $\phi_m^N$ is univalent on this set. Thus $\frac{1}{\phi_m^N(1/z)}$ is univalent on ${\mathrm D}(0,\frac{\kappa}{2})$ and is tangent to the identity at $0$. It follows from the Koebe one-quarter theorem (Theorem \ref{Koebe}) and the injectivity of $\phi_m^N$ that $|\alpha_m^N|\leq 4\frac{2}{\kappa}=\frac{8}{\kappa}$. 

Define $\psi_m^N(z):= \phi_m^N(z)-\alpha_m^N$.  Then for each $0 \leq m \leq N$, if we define
\begin{align}
\label{PmnDef}
P_m^N(z)= \psi_m^N \circ P \circ (\psi_{m-1}^N)^{-1}(z)
\end{align}
we have that $P_m^N$ is a quadratic polynomial whose leading coefficient is again $-\lambda \kappa$ and thus has absolute value $\kappa$. Moreover, $P_m^N$ fixes $0$ as it is ${\tilde P}_m^N$ composed with suitably chosen (uniformly bounded) translations.  We now turn to calculating bounds on the coefficients of each $P_m^N$.  

\begin{lemma}
\label{17plusk}
Any sequence formed from the polynomials $P_m^N(z)$ for $0\leq m \leq N$ as above is a $(17+\kappa)$-bounded sequence of polynomials.  
\end{lemma}

\begin{proof}
By the construction \eqref{PmnDef} above, the  leading coefficient has absolute value $\kappa$ while the constant term is zero.  Now, for $|z|$ sufficiently large
\begin{eqnarray*}
P_m^N(z) &=&\lambda \left ( z+\alpha_{m-1}^N+\mathcal{O}\left (\frac{1}{|z|}\right ) \!\right ) \!\left (1-\kappa z-\kappa \alpha_{m-1}^N+\mathcal{O}\left (\frac{1}{|z|}\right )\!\right )\\
& & - \alpha_{m}^N + \mathcal{O}\left (\frac{1}{|P\circ (\psi_{m-1}^N)^{-1}(z))|}\right ),
\end{eqnarray*}
and one sees easily that the $\mathcal{O}\left (\frac{1}{|P\circ (\psi_{m-1}^N)^{-1}(z))|}\right )$ term is actually $\mathcal{O}\left (\!\frac{1}{|z|^2}\!\right )$. Therefore the coefficient of the linear term is $\lambda-2 \lambda \kappa \alpha_{m-1}^N$, and thus is bounded in modulus by $1+2\cdot1\cdot \kappa \cdot \frac{8}{\kappa}=17$. Lastly, since $\kappa \ge 1$, $\tfrac{1}{17 + \kappa} < \kappa < 17 + \kappa$ and so we have indeed 
constructed a $17+ \kappa$-bounded sequence of polynomials (as defined near the start of Section 1.1), proving the lemma as desired. \end{proof}

\begin{lemma}
\label{psi0N-1est}
$\psi_0^N$ and $(\psi_0^N)^{-1}$ both converge uniformly to the identity on $\C$ (with respect to the Euclidean metric).
\end{lemma}

\begin{proof}
Recall that $\Gamma$ is the boundary of $\Omega'$ and that we chose $\calK \subset \Omega \subset {\overline \Omega} \subset \Omega' \subset \mathrm{D}(0,\frac{2}{\kappa})$. Let $G(z)$ be the Green's function for $P$ and set $h:= \sup_{z \in \Gamma}G(z)$.  Then $\mathrm{supp}\, \mu_m^N \subset \{z: 0 < G(z) \leq h \cdot 2^{m-N} \}$ and so $\mathrm{supp}\, \mu_0^N \subset \{z: 0 < G(z) \leq h \cdot 2^{-N} \}$. Thus $\mu_0^N \rightarrow 0$ everywhere as $N \rightarrow \infty$.  By Theorem I.7.5 on page 24 of \cite{CG} (see also Lemma 1 on page 93 of \cite{Ahl}), we have that $\phi_0^N$ and $(\phi_0^N)^{-1}$ both converge uniformly to the identity on $\C$ (recall that the unique solution for $\mu \equiv 0$ is the identity in view of the uniqueness part of the measurable Riemann mapping theorem for solving the Beltrami equation e.g. Theorem I.7.4 on page 22 of \cite{CG}). Finally, $\alpha_0^N=\phi_{0}^N(0) \rightarrow 0$ as $N\rightarrow \infty$, and since $\psi_0^N=\phi_0^N(z)-\alpha_0^N$, the result follows. \end{proof}

\begin{figure}[htb]
	\label{PILShrink}
	\begin{center}
		\scalebox{.807}{\includegraphics[frame]{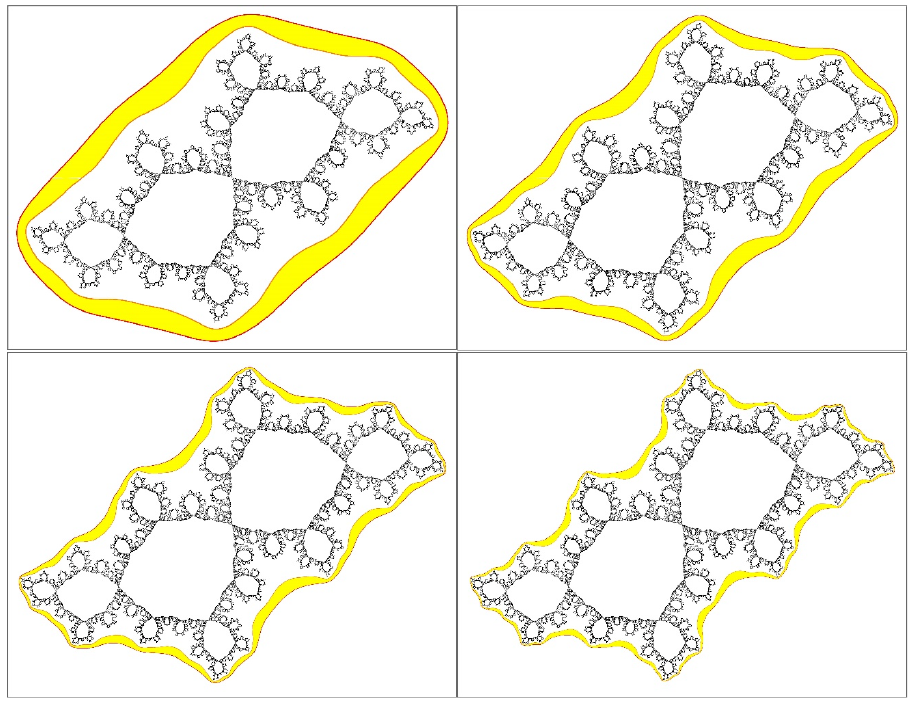}}
		\caption{Supports of Dilatations Converging to Zero Almost Everywhere}
	\end{center}
\end{figure}

The support of each $\mu_0^N$ is contained the basin of infinity for $P$, $A_{\infty}$.  Since we had $2^{-N}\inf_{z \in \gamma}G(z)>0$, $\psi_0^N$ is analytic on a neighborhood of $\overline{U}$.  Then if we define $U^N :=\psi_0^N(U)$, we have that $(\psi_0^N)^{-1}$ is analytic on a neighborhood of $\overline{U^N}$.  We now prove two fairly straightforward technical lemmas. 

\begin{lemma}
\label{uno}
$(U^N,0)\rightarrow (U,0)$ in the \Car topology. 
\end{lemma}

\begin{proof}
Define $\psi^{-1}: \D \rightarrow U$ to be the unique inverse Riemann map from $\D$ to $U$ satisfying $\psi^{-1}(0)=0$, $(\psi^{-1})'(0)>0$. By Lemma \ref{psi0N-1est} above, $\psi_{0}^N \circ \psi^{-1} $ converges locally uniformly to $Id \circ\psi^{-1}$ on $\D$.  The result then follows from Theorem \ref{theorem1.2} in view of the fact that by the above result $(\psi_{0}^N)'(0) \to 1$ as $N \to \infty$ (so that the argument of $(\psi_{0}^N)'(0)$ converges to $0$ as $N \to \infty$). \end{proof}

\begin{lemma}
\label{psi0N-1hd}
For any $\epsilon>0$ and any relatively compact subset $A$ of $U$, there exists an $N_0$ such that
\begin{align*}
|(\psi_0^N)^{\natural}(z)-1|&<\epsilon \\
|((\psi_0^N)^{-1})^{\natural}(z)-1|&<\epsilon
\end{align*}
for all $z$ in $A$, $N\geq N_0$.
\end{lemma}

\begin{proof}
Let $\mathrm{d}\rho_U=\sigma(z)|\mathrm{d}z|$, where the hyperbolic density $\sigma$ is continuous on $U$ (e.g. \cite{KL} Theorem 7.2.2) and bounded away from $0$ on any relatively compact subset of $U$. For each $N$, $\psi_0^N$ is analytic on a neighbourhood of $\overline U$ while 
by Lemma \ref{psi0N-1est}, $(\psi_0^N)^{-1}$ is analytic on any relatively compact subset of $U$ for $N$ sufficiently large, so that by the same result both $(\psi_0^N)'$ and $((\psi_0^N)^{-1})'$ converge uniformly to $1$ on $A$. Since $A$ is a relatively compact subset of $U$, there exists $\delta_0>0$ such that $U$ contains a Euclidean $2\delta_0$-neighborhood of $A$. Let $\tilde A$ denote a Euclidean $\delta_0$ neighbourhood of $A$, so that $\tilde A$ is still a relatively compact subset of $U$. By Lemma \ref{psi0N-1est} again, we can choose $N_0$ large enough such that $\psi_0^N(A)\subset \tilde{A}$ for all $N\geq N_0$.  Then, since $\sigma$ is continuous on the relatively compact subset $A$ of $U$, there exists $\eta>0$ such that $|\sigma|>\eta$ on $A$.  Then for $z \in A$, using the uniform continuity of $\sigma$ on the relatively compact subset $\tilde{A}$ of $U$, 
\begin{align*}
(\psi_0^N)^{\natural}(z)&= \frac{(\psi_0^N)'(z)\sigma(\psi_0^N(z))}{\sigma(z)} 
\end{align*}
converges uniformly to 1 on $A$, as desired.  The proof for $((\psi_0^N)^{-1})^{\natural}$ is similar. \end{proof}

\section{Statement and Proof of the Polynomial Implementation Lemma}

Recall that we had defined $P_m^N(z)= \psi_m^N \circ P \circ (\psi_{m-1}^N)^{-1}(z)$ so that we have defined $P_m^N$ for $0 \leq m \leq N$.  Recall also that we have a sequence $\{n_k\}_{k=1}^\infty$ for which the subsequence $\{P^{\circ n_k}\}_{k=1}^\infty$ converges uniformly to the identity on compact subsets of $U$ (in fact, we can choose $\{n_k \}_{k=1}^{\infty}$ to be the Fibonacci sequence e.g. Problem C-3. on page 244 of \cite{Mil}).    

Define $Q_{n_k}^{n_k}(z)= P_{n_k}^{n_k}\circ P_{n_k -1}^{n_k}\circ \cdots \cdots \circ P_{2}^{n_k}\circ P_{1}^{n_k}(z)$ and note that this simplifies so that $Q_{n_k}^{n_k}(z)=\psi_{n_k}^{n_k}\circ P^{\circ n_k} \circ (\psi_{0}^{n_k})^{-1}(z)$. Essentially the same argument as in the proof of Lemma \ref{psi0N-1hd} allows us to prove the following:

\begin{lemma}
\label{Pnkhd}
For any $\epsilon>0$ and any relatively compact subset $A$ of $U$, there exists $k_0$ such that
\begin{align*}
|(P^{\circ n_k})^{\natural}(z)-1|&<\epsilon
\end{align*}
for all $z$ in $A$, $k\geq k_0$.
\end{lemma}

\vspace{.2cm}
We now state the Polynomial Implementation Lemma.  It is by means of this lemma that we create all polynomials constructed in the proofs of Phases I and II. First we make the definition that, for a relatively compact set $A$ of $U$ and $\delta >0$, the set $\{z \in U: \rho_U(z, A) < \delta\}$ (where $ \rho_U(z, A)$ is the hyperbolic distance in $U$ from $z$ to $A$ as specified in Definition \ref{HypDist}) is called the \emph{$\delta$-neighbourhood} of $A$. Observe that such a neighbourhood is again a relatively compact subset of $U$.

\begin{lemma}
\label{PIL}
(The Polynomial Implementation Lemma) Let $P_\lambda$, $U_\lambda$, $\kappa$, $P$, $U$, $\{n_k\}_{k=1}^\infty$, $\Omega$, $\Omega'$, $\gamma$, $\Gamma$, and $f$ be as above where in addition we also require $f(0) = 0$. Suppose $A\subset U$ is relatively compact and $\delta$, $M$ are positive such that, if $\hat{A}$ is the $\delta$-neighborhood of $A$ with respect to $\rho_U$ as above, then we have $f(\hat A) \subset U$ and $\|f^{\natural} \|_{\hat{A}}\leq M$. Then, for all $\epsilon$ positive, there exists $k_0 \ge 1$ (determined by $\kappa$, the curves $\gamma$, $\Gamma$, the function $f$, as well as $A$, $\delta$, $M$, and $\epsilon$) such that for each $k_1 \ge k_0$ there exists a (17+$\kappa$)-bounded finite sequence of quadratic polynomials $\{P_m^{n_{k_1}}\}_{m=1}^{n_{k_1}}$ (which also depends on $\kappa$, $\gamma$, $\Gamma$, $f$, $A$, $\delta$, $M$, $\epsilon$ as well as $k_1$) such that $Q_{n_{k_1}}^{n_{k_1}}$ is univalent on $A$ and

\begin{enumerate}
\item $\rho_U(Q_{n_{k_1}}^{n_{k_1}}(z), f(z))< \epsilon$ for all $z \in A$,

\item $\|(Q_{n_{k_1}}^{n_{k_1}})^{\natural} \|_{A}\leq M(1+\epsilon)$,

\item $Q_{n_{k_1}}^{n_{k_1}}(0)=0$.
\end{enumerate}

\end{lemma}

Before embarking on the proof, a couple of remarks: first this result is set up so that the subsequence of iterates $\{n_k\}_{k=1}^\infty$ used is always the same. Although we do not require this, it is convenient as it allows us to apply the theorem to approximate many functions simultaneously (which may be of use in some future application) and use the same number of polynomials in each of the compositions we obtain. Second, one can view this result as a weak form of our main theorem (Theorem \ref{thetheorem}), in that it allows to to approximate a single element of $\calS$ with arbitrary accuracy using a finite composition of quadratic polynomials.

\begin{proof}
Let $\epsilon,\delta$ be as above and without loss of generality take $\epsilon<\min \{\delta,1\}$. By Lemma \ref{lemma4.3} the Euclidean and hyperbolic metrics are equivalent on compact subsets of $U$ and we can then we can use Lemma \ref{psi0N-1est} to pick $k_0$ sufficiently large so that for all $k_1 \ge k_0$
\begin{align}
\label{psi-1app}
\rho_U((\psi_0^{n_{k_1}})^{-1}(z),z)< \frac{\epsilon}{3(M+1)}, \hspace{1cm} z\in A.
\end{align}
This also implies that if we let $\check A$ be the $\tfrac{\delta}{2}$-neighbourhood of $A$ in $U$, then, since $\epsilon < \delta$,  
\begin{align}
\label{psi-1loc}
(\psi_0^{n_{k_1}})^{-1}(A) \subset \check{A}.  
\end{align}

Next, by Lemma \ref{psi0N-1hd}, we can make $k_0$ larger if needed such that for all $k_1 \ge k_0$
\begin{align}
\label{psi-1hd}
|((\psi_0^{n_{k_1}})^{-1})^{\natural}(z)-1|< \frac{\epsilon}{3}, \hspace{1cm} z \in  A.
\end{align}

From above, since $\{P^{\circ n_k}\}_{k=1}^\infty$ converges locally uniformly to the identity on $U$ (with respect to the Euclidean metric), using Lemma \ref{lemma4.3}, we can again make $k_0$ larger if necessary to ensure for all $k_1 \ge k_0$ that
\begin{align}
\label{Papp}
\rho_U(P^{\circ n_{k_1}}(z),z)<\frac{\epsilon}{3(M+1)}. \hspace{1cm} z \in \check{A}.
\end{align}
This also implies 
\begin{align}
\label{Ploc}
P^{\circ n_{k_1}}(\check{A}) \subset \hat{A}.
\end{align}

By Lemma \ref{Pnkhd}, we can again make $k_0$ larger if needed such that for all $k_1 \ge k_0$
\begin{align}
\label{Phd}
|(P^{\circ n_{k_1}})^{\natural}(z)-1|< \frac{\epsilon}{3},\hspace{1cm} z \in \check{A}.
\end{align}
We remark that this is the last of our requirements on $k_0$ and we are now in a position to establish the dependencies of $k_0$ on $\kappa$, $\gamma$, $\Gamma$, $f$, $A$, $\delta$, $M$, $\epsilon$ in the statement. To be precise, the requirements on $k_0$ in \eqref{psi-1app} depend on  $\kappa$, $\gamma$, $\Gamma$, $f$, $A$, $M$, and $\epsilon$ (but not $\delta$) while those in \eqref{psi-1hd} depend on $\kappa$, $\gamma$, $\Gamma$, $f$, $A$, and $\epsilon$ (but not $\delta$ or $M$). 
Note that the dependency of these two estimates on the curves $\gamma$, $\Gamma$, or equivalently on the domains $\Omega$, $\Omega'$, (which in turn depend on the scaling factor $\kappa$) as well as the function $f$ arises from the quasiconformal interpolation performed with the aid of Lemma \ref{qcextension} which is clearly dependent on these curves and this function. 
Further, the requirements on $k_0$ in \eqref{Papp} depend on  $\kappa$, $A$, $\delta$, $M$, and $\epsilon$ (but not $\gamma$, $\Gamma$, or $f$) while those in \eqref{Phd} depend on $\kappa$, $A$, $\delta$, and $\epsilon$ (but not $\gamma$, $\Gamma$, $f$, or $M$). Finally, for the remaining estimates, \eqref{psi-1loc} is a direct consequence of \eqref{psi-1app} while \eqref{Ploc} follow immediately from  \eqref{Papp} so that none of these three introduces any further dependencies. 

Now fix $k_1 \ge k_0$ arbitrary and let the finite sequence $\{P_m^{n_{k_1}}\}_{m=1}^{k_1}$ be constructed according to the sequence $\{n_k \}_{k=1}^{\infty}$ specified at the start of this section and the prescription given in \eqref{PmnDef}. Note that this sequence is then $(17+\kappa)$-bounded in view of Lemma \ref{17plusk}. By construction $Q_{n_k}^{n_k}(0)=0$ for every $k$ so that \emph{(3)} in the statement above will be automatically satisfied. 

Now \eqref{PmnDef}, \eqref{psi-1loc}, \eqref{Ploc}, the univalence of $P$ on $U$, and the univalence of $f$ on a neighbourhood of $\calK$ imply that $Q_{n_{k_1}}^{n_{k_1}}$ is univalent on $A$. Now let $z \in A$ and, using \eqref{psi-1app}, consider a geodesic segment $\gamma$ connecting $z$ to ${(\psi_{0}^{n_{k_1}})}^{-1}(z)$ which, since $\epsilon < \delta$, has length smaller than $\tfrac{\epsilon}{3}$. Since $\epsilon < \min\{\delta, 1\}$, $\tfrac{\epsilon}{3}$ is in turn smaller than $\tfrac{\delta}{2}$ and so, by the definition of $\check A$, we have $[\gamma] \subset \check A$. This allows us to apply \eqref{psi-1app},  \eqref{Phd} and the hyperbolic M-L estimates (Lemma \ref{hyperbolicML}) for $P^{\circ n_{k_1}}$ to conclude that the length of $P^{\circ n_{k_1}}(\gamma)$ is at most $(1 + \tfrac{\epsilon}{3})\tfrac{\epsilon}{3(M +1)}$ which is smaller than $\tfrac{\delta}{2}$ since $\epsilon < \min\{\delta, 1\}$. As $[\gamma] \subset \check A$, by \eqref{Ploc}, $[P^{\circ n_{k_1}}(\gamma)] \subset \hat A$ and we are than able to apply the hyperbolic M-L estimates for $f$ since 
 by hypothesis we have $|f^{\natural}(z)|\leq M$ on $\hat{A}$.

 In a similar manner, if instead we consider a geodesic segment connecting $z$ to $P^{\circ n_{k_1}}(z)$, then, since $\epsilon < \delta$, by \eqref{Papp} this segment again has length less than $\tfrac{\delta}{2}$ and starts at $z \in A$, whence it lies inside $\check A \subset \hat A$ and we are again able to apply the hyperbolic M-L estimates for $f$ to this segment. Using the triangle inequality and applying the estimates \eqref{psi-1app} - \eqref{Phd} (except \eqref{psi-1hd}) as well as $|f^{\natural}(z)|\leq M$ on $\hat{A}$ from the statement, for each $z \in A$, since $P^{\circ n_{k_1}} \circ (\psi_{0}^{n_{k_1}})^{-1}(z) \in \hat A$ and $f(\hat A) \subset U$ by hypothesis, we have
\begin{align*}
\rho_U(Q_{n_{k_1}}^{n_{k_1}}(z),f(z)) &= \rho_U(\psi_{n_{k_1}}^{n_{k_1}}\circ P^{\circ n_{k_1}} \circ (\psi_{0}^{n_{k_1}})^{-1}(z),f(z)) \\
&\leq \rho_U(f\circ P^{\circ n_{k_1}} \circ (\psi_{0}^{n_{k_1}})^{-1}(z),f\circ P^{\circ n_{k_1}}(z))\\
&+ \rho_U(f\circ P^{\circ n_{k_1}}(z),f(z)) \\
&<M\! \left (1+\frac{\epsilon}{3}\right )\!\left (\frac{\epsilon}{3(M+1)}\right )+M \!\left (\frac{\epsilon}{3(M+1)}\right ) \\
&<\epsilon
\end{align*}
(recall that we assumed $\epsilon <1$) which proves \emph{(1)}.  Also, using the chain rule \eqref{chainrule} for the hyperbolic derivative, the estimate $|f^{\natural}(z)|\leq M$ on $\hat{A}$, and \eqref{psi-1loc}, \eqref{psi-1hd}, \eqref{Ploc}, \eqref{Phd}, for each $z \in A$,
\begin{align*}
|((Q_{n_{k_1}}^{n_{k_1}})^{\natural})(z)| &= |f^{\natural}(P^{\circ n_{k_1}}\circ (\psi_0^{n_{k_1}})^{-1}(z))\cdot (P^{\circ n_{k_1}})^{\natural}((\psi_0^{n_{k_1}})^{-1}(z))\cdot ((\psi_{0}^{n_{k_1}})^{-1})^{\natural}(z)| \\
&\leq M\! \left (1+\frac{\epsilon}{3}\right )\! \left (1+\frac{\epsilon}{3}\right ) \\
&<M(1+\epsilon)
\end{align*}
again using $\epsilon < 1$ at the end which proves \emph{(2)} as desired.  \end{proof}

%% file: PhaseI.tex
\chapter{Phase I}

\section{Setup}

We begin by finding a suitable disk on which $f \circ g^{-1}$ is defined for arbitrary $f,g \in \calS$.

\begin{lemma}
\label{fcircgdefined}
If $f,g \in \calS$, then $f\circ g^{-1}$ is defined on $\mathrm{D}(0,\tfrac{1}{12})$ and 
$$(f\circ g^{-1})(\mathrm{D}(0,\tfrac{1}{12}))\subset \mathrm{D}(0,\tfrac{1}{3}).$$  
\end{lemma}

\begin{proof}
Let $f,g \in \calS$.  By the Koebe one-quarter theorem (Theorem \ref{Koebe}) we have $\mathrm{D}(0,\frac{1}{4}) \subset g(\D)$ so $g^{-1}$ is defined on $\mathrm{D}(0,\frac{1}{4})$.   Then if $h(w):=4g^{-1}(\frac{w}{4})$ for $w \in \D$ we have that $h\in \calS$ and $g^{-1}(z) = \frac{1}{4}h(4z)$ for $z \in \mathrm{D}(0,\frac{1}{4})$, where $z=\tfrac{w}{4}$.  Thus if $|z|\leq \frac{1}{12}$, we have $|w| \leq \frac{1}{3}$ and by the distortion theorems (Theorem \ref{distortion}) we have that $|h(w)|\leq \frac{3}{4}$ and $|g^{-1}(z)|\leq \frac{3}{16} < 1$ so that in particular $f \circ g^{-1}(z)$ exists. Then, using the distortion theorems again, if $z \in \mathrm{D}(0,\frac{1}{12})$ we have that $(f \circ g^{-1})(z) \leq \frac{48}{169}<\frac{1}{3}$.  Thus $f \circ g^{-1}$ is defined on $\mathrm{D}(0,\frac{1}{12})$ for all $f,g \in \calS$ and maps $\mathrm{D}(0,\frac{1}{12})$ into $\mathrm{D}(0,\frac{1}{3})$ as required.  \end{proof}

In the proof of Phase I we will scale the filled Julia set for the polynomial \label{overviewpi1}$P_\lambda(z)=\lambda z (1-z)$ where $\lambda= e^{\frac{2\pi i (\sqrt{5}-1)}{2}}$ so that the filled Julia set is a subset of $\mathrm{D}(0,\frac{1}{12})$.  We are then able to apply $f\circ g^{-1}$ for $f,g \in \calS$ which are then defined on this filled Julia set.  We wish to find a suitable subdomain of this scaled filled Julia set so that we may control the size of the hyperbolic derivative $(f\circ g^{-1})^{\natural}$ on that subdomain.  There are two strategies for doing this: one can either consider a small hyperbolic disk in the Siegel disc, or one can scale $P_\lambda$ so that the scaled filled Julia set lies inside a small Euclidean disc about $0$.  We found the second option more convenient, as it allows us to consider an arbitrarily large hyperbolic disk inside the scaled Siegel disc on which $|(f\circ g^{-1})'|$ is tame and $|(f\circ g^{-1})^{\natural}|$ thus easier to control.  Lemmas \ref{modz2} through \ref{fcircghdlemma} deal with finding a suitable scaling which allows us to obtain good estimates for $|(f\circ g)^{\natural}|$.  

\begin{lemma}
\label{modz2}
There exists $K_1>0$ such that for all $f,g \in \calS$, if $|z|\leq \frac{1}{24}$, then 
\begin{align*}
|(f\circ g^{-1})(z)-z|\leq K_1|z|^2.
\end{align*}
\end{lemma}

\begin{proof}
Let $f,g \in \calS$.  By Lemma \ref{fcircgdefined} the function $f\circ g^{-1}$ is defined on $\mathrm{D}(0,\frac{1}{12})$.  Let $w\in \D$, $z=\frac{1}{12}w$, so that $z \in \mathrm{D}(0,\frac{1}{12})$, and define $h(w)=12(f\circ g^{-1})(\frac{w}{12})$ so that $h \in \calS$.  
Then, letting $w + \sum_{n=2}^\infty {a_n w^n}$ denote the Taylor series about $0$ for $h$ and setting $K_0 = e \sum_{n=2}^{\infty}n^3 \frac{1}{2^{n-2}}$, if $|w|\leq \frac{1}{2}$ we have 
\begin{align*}
|h'(w)-1|&=\left|w\sum_{n=2}^{\infty}na_nw^{n-2} \right| \\
&\leq |w|\sum_{n=2}^{\infty}n|a_n||w|^{n-2} \\
&\leq |w| e\sum_{n=2}^{\infty}n^3\frac{1}{2^{n-2}} \\
&=K_0|w|
\end{align*}
where we used that $|a_n|\leq en^2$ as $h\in \calS$ (see e.g. Theorem I.1.8 in \cite{CG}).  Let $\gamma=[0,w]$ be the radial line segment from $0$ to $w$.  Then, if $|w|\leq \frac{1}{2}$,
\begin{align*}
|h(w)-w|&=\left|\int_{\gamma}\left[ h'(\zeta)-1 \right] \mathrm{d}\zeta\right| \\
&\leq K_0|w|\int_{\gamma}|\mathrm{d}\zeta| \\
&=K_0|w|^2.
\end{align*}
Then if $|z|\leq \frac{1}{24}$ (so that $|w|\leq \frac{1}{2}$), a straightforward calculation shows
\begin{align*}
|(f\circ g^{-1})(z)-z|&\leq 12K_0|z|^2,
\end{align*}
from which the lemma follows on setting $K_1=12K_0$.
\end{proof}

Recall \label{overviewpi2}$P_\lambda = \lambda z(1-z)$ and the corresponding Siegel disc $U_\lambda$.  Now fix $R>0$ arbitrary and let ${\tilde U}_R$ denote $\Delta_{U_\lambda}(0,R)$, the hyperbolic disc of radius $R$ about $0$ in $U_\lambda$. Let $\psi_\lambda:{U_\lambda}\rightarrow \D$ be the unique Riemann map satisfying $\psi_\lambda(0)=0$, $\psi_\lambda'(0)>0$.  \label{tilder0} Let $\tilde{r}_0=\tilde{r}_0(R):=\mathrm{d}(\partial {\tilde U}_R,\partial {U_\lambda})$, the Euclidean distance from $\partial {\tilde U}_R$ to $\partial U_\lambda$. Similarly to in the last chapter, for $\kappa>0$ arbitrary, set $P:=\frac{1}{\kappa}{P_\lambda}(\kappa z)$ and note that $P$ obviously depends on $\kappa$.  Then, if $\calK=\calK(\kappa)$ is the filled Julia set for $P$, we have $\calK\subset \mathrm{D}(0,\frac{2}{\kappa})$. Let $U= \{z: \kappa z \in U_\lambda\}$ be the corresponding  Siegel disc for $P$ and set $U_R=\Delta_{U}(0,R)$.  Define $\psi(z) := \psi_\lambda(\kappa z)$ and observe that $\psi$ is the unique Riemann map from $U$ to $\D$ satisfying $\psi(0)=0$, $\psi'(0)>0$. \label{r0} Lastly, define $r_0=r_0(\kappa,R):=\mathrm{d}(\partial U_R, \partial U)$ and note $r_0=\frac{\tilde {r}_0}{\kappa}$. Observe that $\tilde{r}_0$ and $r_0$ are decreasing in $R$ while we must have $\tilde{r}_0 \le 2$. In what follows, let $P_\lambda$, $U_\lambda$, $\psi_\lambda$, $P$, $U$, $\psi$, $\tilde{r}_0$ and $r_0$ be fixed. For the moment, we let $\kappa >0$ be arbitrary. We will, however, be fixing a lower bound on $\kappa$ in the lemmas which follow.

\begin{lemma}
\label{localdistortion}
(Local Distortion) For all $\kappa,\, R_0>0$, there exists $C_0=C_0(R_0)$ depending on $R_0$ (in particular, $C_0$ is independent of $\kappa$) which is increasing, real-valued, and (thus) bounded on any bounded subset of $[0,\infty)$ such that, if $U_{R_0}$ and $r_0 = r_0(\kappa, R_0)= \mathrm{d}(\partial U_{R_0}, \partial U)$ are as above and $z_0\in \overline U_{R_0}$, $z\in U$ with $|z-z_0|\leq s<r_0$, we have 
\begin{enumerate}
\item $|\psi(z)-\psi(z_0)|\leq \frac{C_0\frac{s}{r_0}}{(1-\frac{s}{r_0})^2},$ 

\item $\frac{1-\frac{s}{r_0}}{(1+\frac{s}{r_0})^3}\leq|\frac{\psi'(z)}{\psi'(z_0)}|\leq\frac{1+\frac{s}{r_0}}{(1-\frac{s}{r_0})^3}.$
\end{enumerate}
\end{lemma}

\begin{proof}
Set $C_0=C_0(R_0)=2\max_{z\in{\overline{\tilde U}_{R_0}}}|\psi_\lambda '(z)| = \tfrac{2}{\kappa}\max_{z\in{\overline{U}_{R_0}}}| \psi'(z)|$.  Then $C_0(R_0)$ does not depend on $\kappa$ and is clearly increasing in $R_0$ and therefore bounded on any bounded subinterval of $[0,\infty)$. For $z \in {\mathrm D}(z, r_0)$, set $\zeta=\frac{z-z_0}{r_0}$ and note that, if we define $\phi(\zeta):=\frac{\psi(r_0\zeta+z_0)-\psi(z_0)}{r_0\psi'(z_0)}$, we have that $\phi \in \calS$.  Applying the distortion theorems (Theorem \ref{distortion}) to $\phi$ we see
\begin{align*}
|\phi(\zeta)|&\leq\frac{|\zeta|}{(1-|\zeta|)^2} \\
&\leq\frac{\frac{s}{r_0}}{(1-\frac{s}{r_0})^2}
\end{align*}
from which we can conclude (using $r_0 = \tfrac{\tilde r_0}{\kappa}$ and $\tilde{r}_0 \le 2$)
$$
|\psi(z)-\psi(z_0)| \le \frac{\frac{s}{r_0}}{(1-\frac{s}{r_0})^2}\cdot C_0
$$
which proves \emph{(1)}.  For \emph{(2)} we again apply the distortion theorems to $\phi$ and observe
\begin{align*}
\frac{1-\frac{s}{r_0}}{(1+\frac{s}{r_0})^3}\leq\frac{1-|\zeta|}{(1+|\zeta|)^3}\leq |\phi'(\zeta)|\leq\frac{1+|\zeta|}{(1-|\zeta|)^3} \leq\frac{1+\frac{s}{r_0}}{(1-\frac{s}{r_0})^3}, 
\end{align*}
from which 2. follows as $\phi'(\zeta)=\frac{\psi'(z)}{\psi'(z_0)}$.
\end{proof}

\begin{lemma}
\label{fcircgid}
For any  $R_0 > 0$ and $\eta>0$ there exists $\kappa_0=\kappa_0(R_0, \eta) \ge 48$ such that, for all $\kappa\geq \kappa_0$, $f,g\in \calS$ and $z\in U$, 
\begin{align*}
|(f\circ g^{-1})(z)-z|\leq \eta r_0.
\end{align*}
where $r_0 = r_0(\kappa, R_0) = \mathrm{d}(\partial U_{R_0}, \partial U)$ is as above. In particular, this holds for $z \in \overline U_{R_0}$.
\end{lemma}

\begin{proof}
Fix $\kappa_0\geq 48$.  By Lemma \ref{modz2} we have, on $U\subset \mathrm{D}(0,\frac{2}{\kappa})\subset \mathrm{D}(0,\frac{1}{24})$, that $|(f\circ g^{-1})(z)-z|<K_1|z|^2$ for some $K_1>0$ (note that $f\circ g^{-1}$ is defined on $U$ by Lemma \ref{fcircgdefined}).  So $|(f\circ g^{-1})(z)-z|<\tfrac{4K_1}{\kappa^2}$ since $|z|<\frac{2}{\kappa}$.  Then make $\kappa_0$ larger if necessary to ensure that $\tfrac{4K_1}{\kappa^2} \le \eta r_0  = \eta \tfrac{\tilde r_0}{\kappa}$ for all $\kappa\geq \kappa_0$ (where we recall that $\tilde r_0 = \tilde r_0(R_0) = \mathrm{d}(\partial {\tilde U}_{R_0},\partial U_\lambda) = \kappa r_0$).  In fact, $\kappa_0=\max\{48, \frac{4K_1}{\eta \tilde r_0}\}$ will suffice and since $\tilde r_0$ depends only on $R_0$, we have the correct dependencies for $\kappa_0$ and the proof is complete. \end{proof}

Lemmas \ref{modz2} through \ref{fcircgid} are technical lemmas that assist in proving the following result which will be essential for controlling the hyperbolic derivative of $\psi$:

\begin{lemma}
\label{fcircghdestimates}
Given $R_0  > 0$, there exists $\kappa_0= \kappa_0(R_0) \ge 48$ such that for all $\kappa\geq \kappa_0$, $f,g\in \calS$, and $z\in \overline U_{R_0}$, $(f\circ g^{-1})(z) \in U$ and
\begin{enumerate}
\item $\frac{1-|\psi(z)|^2}{1-|\psi((f\circ g^{-1})(z))|^2}\leq \frac{10}{9},$ 

\item $\frac{|\psi'((f\circ g^{-1})(z))|}{|\psi'(z)|}\leq \frac{9}{8}.$
\end{enumerate}
\end{lemma}

\begin{proof} For $R > 0$, set $C_R:= \frac{e^R-1}{e^R+1}$. Then, if  we fix $z_0 \in \overline U_{R_0}$ we have that $|\psi(z_0)|\leq c_{R_0}$ (recall that $\rho_{\D}(0,z)=\log \frac{1+|z|}{1-|z|}$ for $z \in \D$).  Thus $c_{R_0}<1$ and 
\begin{align}
\label{MVTEstimate1}
1-|\psi(z_0)|^2  \ge 1-c_{R_0}^2>0.
\end{align} As in the proof of Lemma \ref{localdistortion}, set $C_0=C_0(R_0)=2\max_{z\in{\overline{\tilde U}_{R_0}}}|{\tilde \psi_{\lambda}}'(z)|$ Let $0<\eta_1=\eta_1(R_0)<\frac{1}{2}$ be such that 
\begin{align}
\label{Eta1Estimate}
\frac{C_0 \eta_1}{(1-\eta_1)^2}\leq \frac{1}{2} (\log 10 - \log 9)(1-c_{R_0 + \log 3}^2)
\end{align}
and note that $\eta_1$ depends only on $R_0$.  Using Lemma \ref{fcircgid}, we can pick $\kappa_1=\kappa_1(R_0, \eta_1) = \kappa_1(R_0) >0$ such that, if $\kappa\geq \kappa_1$, then $|(f\circ g^{-1})(z)-z|<\eta_1 r_0$ on $U \supset \overline U_{R_0}$ (recall the definitions of $\tilde{r_0} = \tilde{r_0}(R_0)$ and $r_0 = r_0(\kappa, R_0)$ given before Lemma \ref{localdistortion}).

Now set $s:=|(f\circ g^{-1})(z_0)-z_0|$. We have $|(f\circ g^{-1})(z_0)-z_0|=s<\eta_1 r_0<\frac{r_0}{2}$ as $\eta_1<\frac{1}{2}$. Then, recalling the definition of $r_0 = \mathrm{d}(\partial U_{R_0}, \partial U)$, we have $(f\circ g^{-1})(z_0) \in {\mathrm D}(z_0, \tfrac{r_0}{2}) \subset {\mathrm D}(z_0, r_0) \subset U$ as in the statement so that in particular $\psi(f\circ g^{-1})(z_0)$ is well-defined. Again using $\rho_{\D}(0,z)=\log \frac{1+|z|}{1-|z|}$ for $z \in \D$ combined with the Schwarz lemma for the hyperbolic metric, we must have that $(f\circ g^{-1})(z_0) \in \overline \Delta_U(z_0, \log 3)$. By the triangle inequality for the hyperbolic metric, 
$(f\circ g^{-1})(z_0) \in \overline \Delta_U(0, R_0 + \log 3) = \overline U_{R_0 + \log 3}$ so that
$|\psi(f\circ g^{-1})(z_0)| \le C_{R_0 + \log 3}$. Then, similarly to \eqref{MVTEstimate1} above, 
\begin{align}
\label{MVTEstimate2}
1-|\psi((f\circ g^{-1})(z_0))|^2 \ge 1-c_{R_0 + \log 3}^2 > 0. 
\end{align}
We may then apply \emph{(1)} of Lemma \ref{localdistortion} and \eqref{Eta1Estimate} above to see that
\begin{align*}
|\psi(z_0)-\psi((f\circ g^{-1})(z_0))|&\leq \frac{C_0\frac{s}{r_0}}{(1-\frac{s}{r_0})^2} \\
&\leq \frac{C_0\eta_1}{(1-\eta_1)^2} \\
&\leq \frac{1}{2} (\log 10 - \log 9)(1-c_{R_0 + \log 3}^2).
\end{align*}
Thus, using the reverse triangle inequality (and the fact that $\psi$ is a Riemann mapping to the unit disc) we see that 
\begin{align*}
|(1-|\psi(z_0)|^2)-(1-|\psi((f\circ g^{-1})(z_0))|^2)|< (\log 10 - \log 9)(1-c_{R_0 + \log 3}^2).
\end{align*}

Making use of \eqref{MVTEstimate1}, \eqref{MVTEstimate2}, noting that $C_R$ is an increasing function of $R$, and applying the mean value theorem to the logarithm function on the interval $[1-c_{R_0 + \log3}^2,\infty)$ we have 
\begin{align*}
|\log(1-|\psi(z_0)|^2)-\log(1-|\psi((f\circ g^{-1})(z_0))|^2)|<\log 10-\log 9
\end{align*}
from which \emph{(1)} follows easily.  For \emph{(2)}, let $0<\eta_2<1$ (e.g. $\eta_2 = \tfrac{1}{35}$) be such that 
$$\frac{1+\eta_2}{(1-\eta_2)^3}<\frac{9}{8}.$$

By Lemma \ref{fcircgid}, using the same $R_0 > 0$ as above, we can pick $\kappa_2=\kappa_2(R_0, \eta_2) = \kappa_2(R_0)>48$ such that for all $\kappa\geq \kappa_2$, if $z\in U \supset \overline U_{R_0}$ 
\begin{align*}
|(f\circ g^{-1})(z)-z|<\eta_2r_0.
\end{align*}

Using the same $z_0\in \overline U_{R_0}$ as above, in a similar way to how we used \emph{(1)} of Lemma \ref{localdistortion} above, we can apply \emph{(2)} of the same result to see that 
\begin{align*}
\frac{|\psi'((f\circ g^{-1})(z_0))|}{|\psi'(z_0)|}\leq \frac{9}{8}
\end{align*}
as desired.  The result follows if we set $\kappa_0=\kappa_0(R_0)= \max\{\kappa_1(R_0),\kappa_2(R_0)\}$. \end{proof}

\begin{lemma}
\label{fcircgd}
For all $\kappa\geq \kappa_0:=576$, for any $f,g \in \calS$ and $z \in {\overline U}$, 
\begin{align*}
|(f\circ g^{-1})'(z)|\leq \frac{6}{5}.
\end{align*}
\end{lemma}

\begin{proof}
As in the proof of Lemma \ref{modz2}, define $h(w)=12(f\circ g^{-1})(\frac{w}{12})$. Note that $h$ is defined on all of $\D$ by Lemma \ref{fcircgdefined} and that $h\in \calS$. Let $z=\frac{w}{12}$.  Using the distortion theorems (Theorem \ref{distortion}), we have that, for $z \in {\mathrm D}(0, \tfrac{1}{12})$,
\begin{align}
\label{fcircgderw}
|(f\circ g^{-1})'(z)|\leq \frac{1+|12z|}{(1-|12z|)^3}.
\end{align}

If $\kappa\geq \kappa_0$ we have that $\mathrm{D}(0,\frac{2}{\kappa})\subset\mathrm{D}(0,\frac{2}{\kappa_0})=\mathrm{D}(0,\frac{1}{288})$.  Let $z \in \overline{U}$ and, since ${\overline U}\subset \calK \subset \mathrm{D}(0,\frac{2}{\kappa})\subset \mathrm{D}(0,\frac{1}{288})$, we have $|z|<\frac{1}{288}$ for $\kappa\geq \kappa_0$.  Thus the right hand side of \eqref{fcircgderw} is less than $\frac{25\cdot 24^2}{23^3}$, which in turn is less than $\frac{6}{5}$ for all $\kappa\geq \kappa_0$ as desired.  \end{proof}

As all the previous lemmas hold for all $\kappa$ sufficiently large, applying them in tandem in the next result is valid.  In general, each lemma may require a different choice of $\kappa_0$, but we may choose the maximum so that all results hold simultaneously.  The purpose of Lemmas \ref{fcircghdestimates} and \ref{fcircgd} is to prove the following:

\begin{lemma}
\label{fcircghdlemma}
Given $R_0 > 0$, there exists $\kappa_0= \kappa_0(R_0) \ge 576$  such that, for all $\kappa \geq \kappa_0$, for any  $f,g \in \calS$, and  $z \in \overline U_{R_0}$, $(f\circ g^{-1})(z) \in U$ and
\begin{align*}
|(f\circ g^{-1})^{\natural}(z)|\leq \frac{3}{2}.
\end{align*}
 \end{lemma}

\begin{proof}
Applying Lemmas \ref{fcircghdestimates} and \ref{fcircgd} to the definition \eqref{hypderiv} of the hyperbolic derivative taken with respect to the hyperbolic metric of $U$, and letting $\kappa_0$ be the maximum of the two lower bounds on $\kappa$ in these lemmas, \label{kappa0} we have that there exists a $\kappa_0 \ge 576$ depending on $R_0$ such that, for all $\kappa\geq \kappa_0$, and $z \in \overline U_{R_0}$, $(f\circ g^{-1})(z) \in U$ and
\begin{align*}
|(f\circ g^{-1})^{\natural}(z)|&=\frac{1-|\psi(z)|^2}{1-|\psi((f\circ g^{-1})(z))|^2}\cdot\frac{2\cdot|\psi'((f\circ g^{-1})(z))|}{2\cdot|\psi'(z)|}\cdot |(f\circ g^{-1})'(z)| \\
&\leq \frac{10}{9}\cdot\frac{9}{8}\cdot\frac{6}{5} \\
&= \frac{3}{2}
\end{align*}
as desired.  
\end{proof}

\section{Statement and Proof of Phase I}

\begin{lemma}
\label{PhaseI}
(Phase I) Let $P_\lambda$, $U_\lambda$, $\kappa$, $P$, and $U$ be as above. Let $R_0 > 0$ be given and let $\tilde{U}_{R_0}$,  and $U_{R_0}$ also be as above. Then, there exists $\kappa_0=\kappa_0(R_0) \ge 576$ such that, for all $\kappa\geq \kappa_0$, $\epsilon>0$, and $N \in \N$, if $\{f_i\}_{i=0}^{N+1}$ is a collection of mappings with $f_i \in \calS$ for $i=0,1,2,\cdots \cdots, N+1$ with $f_0=f_{N+1}= Id$, there exists an integer $M_N$ and a $(17+\kappa)$-bounded finite sequence $\{P_m\}_{m=1}^{(N+1)M_N}$ of quadratic polynomials both of which depend on  $R_0$, $\kappa$, $N$, the functions  $\{f_i\}_{i=0}^{N+1}$, and $\epsilon$ such that, for each $1\leq i \leq N+1$,

\begin{enumerate}

\item $Q_{i M_N}(0)=0$,  

\item $Q_{iM_N}$ is univalent on $U_{2R_0}$,

\item $\rho_{U}(f_i(z), Q_{iM_N}(z))<\epsilon$ on $U_{2R_0}$,

\item $\|Q_{iM_N}^{\natural} \|_{U_{R_0}}\leq 7$.

\end{enumerate}

\end{lemma}

Before proving this result, we remark first  that the initial function $f_0 = Id$ in the sequence $\{f_i\}_{i=0}^{N+1}$ does not actually get approximated. The reason we included this function was purely for convenience as this allowed us to describe all the functions being approximated in the proof using the Polynomial Implementation Lemma (Lemma \ref{PIL}) as $f_{i+1} \circ f_i^{-1}$, $0 \le i \le N$. 

Second, we can view this result as a weak form of our main theorem in that it allows to to approximate finitely many elements of $\calS$ with arbitrary accuracy using a finite composition of quadratic polynomials. Phase I is thus intermediate in strength between the Polynomial Implementation Lemma (Lemma \ref{PIL}) and our main result (Theorem \ref{thetheorem}).

\begin{proof} {\bf Step 1:} Setup 

Without loss of generality, make $\epsilon$ smaller if necessary to ensure $\epsilon<R_0$. \label{epsilonbound1} Let $\kappa_0= \kappa_0(R_0) \ge 576$ be as in the statement of Lemma \ref{fcircghdlemma} so that the conclusions of this lemma as well as those of Lemmas \ref{fcircghdestimates} and \ref{fcircgd} also hold.  Then for all $\kappa\geq \kappa_0$ we have $U \subset \calK \subset \mathrm{D}(0,\frac{2}{\kappa})\subset \mathrm{D}(0,\frac{1}{288}) \subset \mathrm{D}(0,\frac{1}{12})$.  Note that the last inclusion implies that, if $f,g \in \calS$, then $f\circ g^{-1}$ is defined on $U$ in view of Lemma \ref{fcircgdefined}.  

{\bf Step 2:} Application of the Polynomial Implementation Lemma.  

First apply Lemma \ref{fcircghdlemma} with $5R_0 + 1$ replacing $R_0$  so that, for all $\kappa\geq \kappa_0$, if $f,g \in \calS$, we have $(f\circ g^{-1})(U_{5R_0 +1}) \subset U$ and
\begin{align}
\label{fcircghd}
\| (f \circ g^{-1})^{\natural}  \|_{U_{5R_0}}\leq \| (f \circ g^{-1})^{\natural}  \|_{U_{5R_0+1}} \le\frac{3}{2}.
\end{align} 
Note that, by Lemmas \ref{stupidfuckinglemma} and \ref{hyperbolicML}, since $U_{2R_0}$ is then hyperbolically convex, this implies 
\begin{align}
\label{fcircgloc}
(f \circ g^{-1})(U_{2R_0}) &\subset U_{3R_0}.
\end{align}
We observe that, since $Id \in \calS$, in particular we have $f(U_{2R_0})\subset U_{3R_0}$ for all $f \in \calS$. 

Fix $\kappa \ge \kappa_0$ and for each $0 \leq i \leq N$, using \eqref{fcircghd}, apply the Polynomial Implementation Lemma (Lemma \ref{PIL}), with $\Omega=\mathrm{D}(0,\frac{1}{24})$, $\Omega'=\mathrm{D}(0,\frac{1}{2})$, $\gamma=\mathrm{C}(0,\frac{1}{24})$, $\Gamma=\mathrm{C}(0,\frac{1}{2})$ (where both of these circles are positively oriented with respect to the round annulus of which they form the boundary), $f=f_{i+1}\circ f_{i}^{-1}$, $A=U_{5R_0}$, $\delta = 1$ (and hence ${\hat{A}}= U_{5R_0+1}$), $M = \tfrac{3}{2}$, and $\epsilon$ replaced with $\frac{\epsilon}{3^{N}}$. Note that $f(0)=0$ and that, in view of Lemma \ref{fcircgdefined}, $f$ is analytic and injective on a neighbourhood of $\overline \Omega$ and maps $\gamma$ inside $\mathrm{D}(0,\frac{1}{3})$ which lies inside $\Gamma$ so that $(f, Id)$ is indeed admissable pair on $(\gamma, \Gamma)$ in the sense given in Definition \ref{admissiblepair} in Chapter 3 on the Polynomial Implementation Lemma which then allows us to obtain a quasiconformal homeomorphism of $\chat$ using Lemma \ref{qcextension}.  

Let $M_N$ be the maximum of the integers $n_{k_0}$ in the statement of Lemma \ref{PIL} for each of the $N+1$ applications of this lemma above. Note that each $k_0$ depends on $\kappa$, the curves $\gamma=\mathrm{C}(0,\frac{1}{24})$, $\Gamma=\mathrm{C}(0,\frac{1}{2})$, and the individual function $f = f_{i+1}\circ f_{i}^{-1}$ being approximated, as well as 
$A$, $\delta$, the upper bound $M$ on the hyperbolic derivative (which in our case by \eqref{fcircghd} is $\tfrac{3}{2}$ for every function we are approximating) and finally $\epsilon$. Thus, $M_N$,
in addition to $N$, then also depends on $R_0$, $\kappa$, the finite sequence of functions $\{f_i\}_{i=0}^{N+1}$, and, finally, (recalling that here we have $\gamma=\mathrm{C}(0,\frac{1}{24})$, $\Gamma=\mathrm{C}(0,\frac{1}{2})$ and $A=U_{5R_0}$, $\delta = 1$, $M = \tfrac{3}{2}$), $\epsilon$. From these $N+1$ applications, we also then obtain 
(after a suitable and obvious labelling) a finite $(17 + \kappa)$-bounded sequence $\{P_m\}_{m=1}^{(N+1)M_N}$ such that each $Q_{iM_N,(i+1)M_N}$ is univalent on $U_{5R_0}$ and we have, for each $0 \le i \le N$ and each $z \in U_{5R_0}$,
\begin{align}
\label{pilapp}
\rho_U(Q_{iM_N,(i+1)M_N}(z),f_{i+1}\circ f_{i}^{-1}(z) )<\frac{\epsilon}{3^N}.
\end{align}
It also follows from Lemma \ref{PIL} that each $Q_{iM_N,(i+1)M_N}$ depends on $N$,  $R_0$, $\kappa$, the functions $f_i$, $f_{i+1}$, and $\epsilon$ so that we obtain the correct dependencies for $M_N$ and 
$\{P_m\}_{m=1}^{(N+1)M_N}$ in the statement. In addition, by \emph{(3)} of Lemma \ref{PIL},
$Q_{iM_N,(i+1)M_N}(0)=0$, for each $i$, proving \emph{(1)} in the statement above.

{\bf Step 3:} Estimates on the compositions $\{Q_{iM_N}\}_{i=1}^{N+1}$  

We use the following claim to prove \emph{(2)} and \emph{(3)} in the statement (note that we do not require 2. of the claim below for this, but we will need it in proving \emph{(4)} later).
\begin{claim}
\label{PhIClaim}For each $1\leq j \leq N+1$, we have that $Q_{jM_N}$ is univalent on $U_{2R_0}$ and, for each $z \in U_{2R_0}$,
\begin{align*}
1.& \; \rho_U(Q_{jM_N}(z),f_{j}(z))<\frac{\epsilon}{3^{N+1-j}}, \\
2.& \; \rho_U(Q_{jM_N}(z),0)<4 R_0.
\end{align*}
\end{claim}
Note that the error in this polynomial approximation for $j=1$ is the smallest as this error needs to pass through the greatest number of subsequent mappings.  

\begin{claimproof} We prove the claim by induction on $j$.  Let $z \in U_{2R_0}$.  For the base case, we have that univalence and 1. in the claim follow immediately from our applications of the Polynomial Implementation Lemma and in particular from \eqref{pilapp} (with $j = i+1 = 1$ so that $i=0$) since  $f_0= Id$.  For 2.,  using 1. (or \eqref{pilapp} above) and \eqref{fcircgloc}, compute
\begin{align*}
\rho_U(Q_{M_N}(z),0) &\leq \rho_U(Q_{M_N}(z),f_{1}(z))+\rho_U(f_{1}(z),0) \\
&<\frac{\epsilon}{3^{N}}+3R_0 \\
&<4R_0,
\end{align*}
which completes the proof of the base case since we had assumed $\epsilon < R_0$. Now suppose the claim holds for some $1\leq j <N+1$.  Then 
\begin{align*}
\rho_U(Q_{(j+1)M_N}(z),f_{j+1}(z) ) \leq 
\rho_U(Q_{jM_N,(j+1)M_N}\circ Q_{jM_N}(z),(f_{j+1}\circ f_{j}^{-1})\circ Q_{jM_N}(z)) \\
+\rho_U((f_{j+1}\circ f_{j}^{-1})\circ Q_{jM_N}(z),(f_{j+1}\circ f_{j}^{-1})\circ f_j(z) ).
\end{align*}

Now $Q_{jM_N}(z) \in U_{4R_0} \subset U_{5R_0}$ by the induction hypothesis, so \eqref{pilapp} implies that the first term on the right hand side in the inequality above is less than $\frac{\epsilon}{3^N}$. Again by the induction hypothesis, $Q_{jM_N}(z) \in U_{4R_0}\subset U_{5R_0}$ while we also have $f_j(z)\in U_{3R_0}\subset U_{5R_0}$ by \eqref{fcircgloc}.  Thus \eqref{fcircghd}, the hyperbolic convexity of $U_{5R_0}$ from Lemma \ref{stupidfuckinglemma}, Lemma \ref{hyperbolicML}, and the induction hypothesis imply that the second term in the inequality is less than $\frac{3}{2}\cdot\frac{\epsilon}{3^{N+1-j}}$.  Thus we have $\rho_U(Q_{(j+1)M_N}(z),f_{j+1}(z) ) < \frac{\epsilon}{3^{N+1-(j+1)}}$, proving the first part of the claim.  

Also, using what we just proved, \eqref{fcircgloc}, and our assumption that $\epsilon < R_0$,
\begin{align*}
\rho_U(Q_{(j+1)M_N}(z),0) &\leq \rho_U(Q_{(j+1)M_N}(z),f_{j+1}(z))+\rho_U(f_{j+1}(z),0) \\
&<\frac{\epsilon}{3^{N+1-(j+1)}}+3R_0 \\
&<4R_0
\end{align*}
which proves 2. in the claim.  Univalence of $Q_{(j+1)M_N}$ follows by hypothesis as $Q_{jM_N}(U_{2R_0})\subset U_{4R_0}$ while $Q_{(j+1)M_n,jM_N}$ is univalent on $A=U_{5R_0}\supset U_{4R_0}$ by the Polynomial Implementation Lemma as stated immediately before \eqref{pilapp}. This completes the proof of the claim, from which \emph{(2)} and \emph{(3)} in the statement of Phase I follow easily. \end{claimproof}

{\bf Step 4:} Proof of \emph{(4)} in the statement.

To finish the proof, we need to give a bound on the size of the hyperbolic derivatives of the compositions $Q_{iM_N}$, $1 \le i \le N+1$. It will be of essential importance to us later that this bound not depend on the number of functions being approximated, the reason being that, in the inductive construction in Lemma \ref{mediuminductionlemma}, the error from the prior application of Phase II (Lemma \ref{PhaseII}) needs to pass through all these  compositions while remaining small. This means that the estimate on the size of the hyperbolic derivative in part \emph{(2)} of the statement of Lemma \ref{PIL} is too crude for our purposes and so we have to proceed with greater care.

Let $\mathrm{d}\rho_U(z)$ be the hyperbolic length element in $U$ and write $\mathrm{d}\rho_U(z)=\sigma_U(z)|\mathrm{d}z|$, where the hyperbolic density $\sigma_U$ is continuous and positive on $U$ (e.g. \cite{KL} Theorem 7.2.2) and therefore uniformly continuous on $U_{4R_0}$, as $U_{4R_0}$ is relatively compact in $U$.  Let $\sigma = \sigma(R_0) > 0$ be the infimum of $\sigma_U$ on $U_{4R_0}$ so that 
\begin{align}
\label{sigmabound}
\sigma_U(z) \ge \sigma, \quad z \in U_{4R_0}.
\end{align}

Let $z \in U_{2R_0}$ and observe that, since $\kappa \ge \kappa_0 \ge 576$, $U \subset {\mathrm D}(0, \tfrac{1}{288}) \subset \D$. Then \emph{(3)} in the statement together with the Schwarz lemma for the hyperbolic metric (e.g. \cite{CG} Theorem 4.1 or 4.2) give, for $1 \leq i \leq N+1$, $\rho_{\D}(Q_{iM_N}(z),f_i(z))\leq \rho_{U}(Q_{iM_N}(z),f_i(z))<\epsilon$.  If $\gamma$ is a geodesic segment in $\D$ from $Q_{iM_N}(z)$ to $f_i(z)$, we see that 
\begin{align*}
\epsilon&>\rho_{U}(Q_{iM_N}(z),f_i(z)) \\
&\geq \rho_{\D}(Q_{iM_N}(z),f_i(z)) \\
&= \int_{\gamma}\mathrm{d}\rho_{\D} \\
&= \int_{\gamma}\frac{2|\mathrm{d}w|}{1-|w|^2} \\
&\geq \int_{\gamma}2|\mathrm{d}w| \\
&=2l(\gamma) \\
&\geq 2|Q_{iM_N}(z)-f_i(z)|
\end{align*}
and so, in particular, 
\begin{align}
\label{qifiepsilon}
|Q_{iM_N}(z)-f_i(z)|<\epsilon.
\end{align}

Now suppose further that $z \in U_{R_0}$ and set 
\begin{align}
\label{distboundary}
\delta_0 = \delta_0(R_0) = \min_{w\in \partial U_{R_0}}\mathrm{d}(w,\partial U_{\frac{3}{2}R_0}),
\end{align}
where $\mathrm{d}(\cdot \,, \cdot)$ denotes Euclidean distance. By Theorem VII.9.1 in \cite{New}, the  winding number of $\partial U_{\frac{3}{2}R_0}$ (suitably oriented) around $z$ is $1$. Then, using Corollary IV.5.9 in \cite{Con} together with the standard distortion estimates in Theorem \ref{distortion} and \eqref{qifiepsilon} above, we obtain
\begin{align*}
|Q_{iM_N}'(z)|&\leq |f_i'(z)|+|Q_{iM_N}'(z)-f_i'(z)| \\
&= |f_i'(z)|+\left| \frac{1}{2\pi i} \int_{\partial U_{\frac{3}{2}R_0}}\frac{Q_{iM_N}(w)-f_i(w)}{(w-z)^2}\mathrm{d}w\right| \\
&\leq \frac{1+\frac{1}{288}}{(1-\frac{1}{288})^3}+\frac{\epsilon}{2\pi\delta_0^2}l(\partial U_{\frac{3}{2}R_0})
\end{align*}
where $l(\partial U_{\frac{3}{2}R_0})$ is the Euclidean length of $\partial U_{\frac{3}{2}R_0}$.  By making $\epsilon$ smaller if needed, we can thus ensure, for $z \in U_{R_0}$, that 
\begin{align}
\label{Qder}
|Q_{iM_N}'(z)|\leq \frac{3}{2}.
\end{align}  

We can make $\epsilon$ smaller still if needed to guarantee that, if $z,w \in U_{4R_0}$, and $|z-w|<\epsilon$, then, by uniform continuity of $\sigma_U$ on $U_{4R_0}$,
\begin{align}
\label{hypdensigma}
|\sigma_U(z)-\sigma_U(w)|<\sigma.
\end{align}

Note that both \eqref{Qder}, \eqref{hypdensigma} above required us to make $\epsilon$ smaller, but these  requirements depended only on $R_0$ and in particular not on the sequence of polynomials we have constructed. Although this means we may possibly need to run the earlier part of the argument again to find a new integer $M_N$ and then construct a new polynomial sequence $\{P_m \; : \; 1 \leq m \leq (N+1)M_N, \; 0 \leq i \leq N \}$, our requirements on $\epsilon$ above will then automatically be met. Alternatively, these requirements on $\epsilon$ could be made before the sequence is constructed. However, we decided to make them here for the sake of convenience.  \label{epsilonbound2}

If $z \in U_{R_0}$ we then have 
\begin{align*}
|Q_{iM_N}^{\natural}(z)|&\leq |f_i^{\natural}(z)| + |Q_{iM_N}^{\natural}(z)-f_i^{\natural}(z)| \\
&=|f_i^{\natural}(z)| + \left|\frac{\sigma_U(Q_{iM_n}(z))}{\sigma_U(z)}Q_{iM_N}'(z)-\frac{\sigma_U(f_{i}(z))}{\sigma_U(z)}f_{i}'(z)\right| \\
&\leq |f_i^{\natural}(z)| + \left|\frac{\sigma_U(Q_{iM_n}(z))}{\sigma_U(z)}Q_{iM_N}'(z)-\frac{\sigma_U(f_{i}(z))}{\sigma_U(z)}Q_{iM_N}'(z)\right| + \\ & \hspace{2.15cm} \left|\frac{\sigma_U(f_i(z))}{\sigma_U(z)}Q_{iM_N}'(z)-\frac{\sigma_U(f_{i}(z))}{\sigma_U(z)}f_{i}'(z)\right|. \\
\end{align*}

We need to bound each of the three terms on the right hand side of the above inequality.  Recall that, as $g=Id \in \calS$, we have that $|f_i^{\natural}(z)|\leq \frac{3}{2}$ by \eqref{fcircghd}.  For the second term, by \eqref{fcircgloc}, \eqref{sigmabound}, \eqref{qifiepsilon}, \eqref{Qder}, \eqref{hypdensigma}, and 2. in Claim \ref{PhIClaim}, we have 
\begin{align*}
\left|\frac{\sigma_U(Q_{iM_n}(z))}{\sigma_U(z)}Q_{iM_N}'(z)-\frac{\sigma_U(f_{i}(z))}{\sigma_U(z)}Q_{iM_N}'(z)\right| =& \\
& \hspace{-2.5cm}\frac{1}{|\sigma_U(z)|}\cdot |Q_{iM_N}'(z)|\cdot |\sigma_U(Q_{iM_N}(z))-\sigma_U(f_{i}(z))| \\
& \hspace{-2.8cm}\leq \frac{1}{\sigma}\cdot \frac{3}{2}\cdot \sigma = \frac{3}{2}.\\
\end{align*}

For the third and final term, recall that we chose $\kappa_0 = \kappa_0(R_0)$ sufficiently large to ensure that the conclusions of Lemmas \ref{fcircghdestimates} and  \ref{fcircgd} hold. We can then apply Lemmas \ref{fcircghdestimates} and \ref{fcircgd}, together with \eqref{Qder} to obtain that
\begin{align*}
\left|\frac{\sigma_U(f_i(z))}{\sigma_U(z)}Q_{iM_N}'(z)-\frac{\sigma_U(f_{i}(z))}{\sigma_U(z)}f_{i}'(z)\right|&\leq \left|\frac{\sigma_U(f_{i}(z))}{\sigma_U(z)}\right|\cdot (|Q_{iM_N}'(z)|+|f_{i}'(z)|) \\
&\leq \frac{10}{9}\cdot \frac{9}{8}\left (\frac{3}{2}+\frac{6}{5} \right) \\
&<  4.
\end{align*}

Thus
\begin{align*}
|Q_{iM_N}^{\natural}(z)|&\leq \frac{3}{2}+\frac{3}{2}+4 =7
\end{align*}
as desired.  \end{proof}

%% file: PhaseII.tex
\chapter{PhaseII}

The approximations in Phase I inevitably involve errors and the correction of these errors is the purpose of Phase II. However, this correction comes at a price in that it is only valid on a domain which is smaller than that on which the error itself is originally defined; in other words there is an unavoidable loss of domain. There are two things here which work in our favour and stop this getting out of control: the first is the Fitting Lemma (Lemma \ref{FittingLemma}) which shows us that loss of domain can be controlled and in fact diminishes to zero as the size of the error to be corrected tends to zero, while the second is that the accuracy of the correction can be made arbitrarily small, which allows us to control the errors in subsequent approximations.

\label{overviewpii1}We will be interpolating functions between Green's lines of a scaled version of the polynomial $P_\lambda= \lambda z(1-z)$ where $\lambda= e^{\frac{2\pi i (\sqrt{5}-1)}{2}}$.  If we denote the corresponding Green's function by $G$, we will want to be able to choose $h$ small enough so that the regions between the Green's lines $\{z: G(z) =h \}$ and $\{z: G(z) =2h \}$ are small in a sense to be made precise later. This will eventually allow us to control the loss of domain.  On the other hand, we will want $h$ to be large enough so that, if we distort the inner Green's line $\{z: G(z) =h \}$ slightly (with a suitably conjugated version of that same error function), the distorted region between them will still be a conformal annulus which will then allow us to invoke the Polynomial Implementation Lemma (Lemma \ref{PIL}).  However, first we must prove several technical lemmas.

\section{Setup and the Target and Fitting Lemmas}

\vspace{.2cm}
We begin this section with continuous versions of Definition \ref{CarConv} of \Car convergence and of local uniform convergence and continuity on varying domains (Definition 3.1 in \cite{Com10}).

\begin{definition}
\label{ContCon}
Let $\calW=\{(W_{h},w_{h}) \}_{h \in I}$ be a sequence of pointed domains indexed by a non-empty set $I\subset \R$.  We say that $\calW$ {\rm varies continuously in the \Car topology} at $h_0 \in I$ or is {\rm continuous at} $h_0$ if, for any sequence $\{h_n\}_{n=1}^\infty$ in $I$ tending to $h_0$, $(W_{h_n},w_{h_n})\rightarrow (W_{h_0},w_{h_0})$ as $n \to \infty$.  If this property holds for all $h \in I$, we say $\calW$ {\rm varies continuously in the \Car topology over} $I$.

For each $h \in I$, let $g_h$ be an analytic function defined on $W_h$. If $h_0 \in I$ and $\calW$ is continuous at $h_0$ as above, we say $g_h$ {\rm converges locally uniformly to $g_{h_0}$ on $W_{h_0}$} if, for every compact subset $K$ of $W_{h_0}$ and every  
sequence $\{h_n\}_{n=1}^\infty$ in $I$ tending to $h_0$, $g_{h_n}$ converges uniformly to $g_{h_0}$ uniformly on $K$ as $n \to \infty$. 

Finally, if we let $\mathcal G = \{g_h\}_{h \in I}$ be the corresponding family of functions, we say that $\mathcal G$ is continuous at $h_0 \in I$ if $g_h$ converges locally uniformly to $g_{h_0}$ on $W_{h_0}$ as above. If this property holds for all $h \in I$, we say $\mathcal G$ {\rm is continuous over} $I$.
\end{definition}

\begin{definition}
Let $I \subset \R$ be non-empty and let $\{\gamma_h \}_{h\in I}$ be a family of Jordan curves indexed over $I$.  We say that $\{\gamma_h \}_{h\in I}$ is a {\rm continuously varying family of Jordan curves over $I$} if we can find a continuous function $F: \mathbb{T}\times I \rightarrow \C$ which is injective in the first coordinate such that, for each $h \in I$ fixed, $F(z,h)$ is a parameterization of $\gamma_h$.  
\end{definition}

Recall that a Jordan curve $\gamma$ divides the plane into exactly two complementary components whose common boundary is $[\gamma]$ (e.g. \cite{Mun} Theorem 8.13.4 or \cite{New} Theorem V.10.2). It is well known that we can use winding numbers to distinguish between the two complementary components of $[\gamma]$. More precisely, we can 
parametrize (i.e. orient) $\gamma$, such that $n(\gamma, z) = 1$ for those points in the bounded complementary component of $\C \setminus [\gamma]$ while $n(\gamma, z) = 0$ for those points in the unbounded complementary component (e.g. Corollary 2 to Theorem VII.8.7 combined with Theorem VII.9.1 in \cite{New}).

\begin{lemma}
\label{jordancurvestodomainscts}
Let $I \subset \R$ be non-empty and $\{\gamma_h \}_{h\in I}$ be a continuously varying family of Jordan curves indexed over $I$. For each $h \in I$ let $W_h$ be the Jordan domain which is the bounded component of $\chat \setminus [ \gamma_h ]$, and let $w:I\rightarrow \C$ be continuous with $w(h) \in W_h$ for all $h$. Then the family $\{(W_h,w(h)) \}_{h \in I}$ varies continuously in the \Car topology over $I$.  
\end{lemma}

\begin{proof}
The continuity of $w$ implies \emph{(1)} of \Car convergence in the sense of Definitions \ref{CarConv}, \ref{ContCon} above.  For \emph{(2)}, fix $h_0 \in I$, let $K \subset W_{h_0}$ be compact, and let $z \in K$.  Set $\delta:={\mathrm d}(K,\partial W_{h_0})$. By the uniform continuity of $F$ on compact subsets of $\mathbb{T}\times I$,  we can find $\eta >0$ such that, for each $h \in I$ with $|h - h_0| < \eta$, 
$$|\gamma_h(t) - \gamma_{h_0}(t)| < \frac{\delta}{2}, \qquad \mbox{for all} \: t\in \T$$
and $\gamma_h$ is thus homotopic to $\gamma_{h_0}$ in $\C \setminus K$. We observe that we have not assumed that $I \cap (h_0-\eta, h_0 + \eta)$ is an interval, so we may not be able to use the parametrization induced by $\gamma_h(z)$ to make the homotopy. However, using the above, it is a routine matter to construct the desired homotopy using convex linear combinations. By the above remark on winding numbers and Cauchy's theorem, one then obtains
$$n(\gamma_h,w)= n(\gamma_{h_0},w) = 1, \qquad \mbox{for all} \: w \in K.$$
Thus, if $|h - h_0| < \eta$, then $K \subset W_h$ and \emph{(2)} of \Car convergence follows readily from this. 

To show \emph{(3)} of \Car convergence, let $\{h_n\}$ be any sequence in $I$ which converges to $h_0$ and suppose $N$ is an open connected set containing $w(h_0)$ such that $N\subset W_{h_n}$ for infinitely many $n$.  Without loss of generality we may pass to a subsequence to assume that $N\subset W_{h_n}$ for all $n$.  Let $z \in N$ and connect $z$ to $w(h_0)$ by a curve $\eta$ in $N$.  As $[\eta]$ is compact, there exists $\delta>0$ such that a Euclidean $\delta$-neighborhood of $[\eta]$ is contained in $N$ and thus avoids $\gamma_{h_n}$ for all $n$.  By the continuity of $F$, this neighborhood also avoids $\gamma_{h_0}$.  Since $w(h_0)$ and $z$ are connected by $\eta$ which avoids $\gamma_{h_0}$, they are in the same region determined by $\gamma_{h_0}$ so that $n(\gamma_{h_0},z)=n(\gamma_{h_0},w(h_0))$. On the other hand, since by hypothesis $w(h_0) \in W_{h_0}$, by Corollary 2 to Theorem VII.8.7 combined with Theorem VII.9.1 in \cite{New},
$n(\gamma_{h_0},w(h_0))=1$ whence $z \in W_{h_0}$. As $z$ is arbitrary, we have $N \subset W_{h_0}$ and \emph{(3)} of \Car convergence and the result then follow.   
\end{proof}

Recall that a Riemann surface is said to be \emph{hyperbolic} if its universal cover is the unit disc $\D$. For a simply connected domain $U \subset \C$, this is equivalent to $U$ being a proper subset of $\C$. The next lemma makes use of the following definition, originally given in \cite{Com11} for families of pointed domains of finite connectivity. Recall that, for a domain $U \subset \C$, we use the notation $\delta_U(z)$ for the Euclidean distance from a point $z$ in $U$ to the boundary of $U$.

\begin{definition}(\cite{Com11}Definition 6.1)
\label{BoundedContainment}
Let $\mathcal V = \{(V_\alpha, v_\alpha)\}_{\alpha \in A}$ be a family of hyperbolic simply connected domains and let $\mathcal U = \{(U_\alpha, u_\alpha)\}_{\alpha \in A}$ be another family of hyperbolic simply connected domains indexed over the same set $A$ where $U_\alpha \subset V_\alpha$ for each $\alpha$. We say that $\mathcal U$ is  {\rm bounded above and below} or just {\rm bounded} in $\mathcal V$ with constant $K \ge 1$ if 

\vspace{-.1cm}

\begin{enumerate}
\item $U_\alpha$ is a subset of $V_\alpha$ which lies within hyperbolic distance at most $K$ about $v_\alpha$ in $V_\alpha$;

\vspace{.2cm}
\item \[ \delta_{U_\alpha} (u_\alpha) \ge \frac{1}{K}  \delta_{V_\alpha} (u_\alpha).\]

\end{enumerate}
In this case we write $\pt \sqsubset \mathcal U \sqsubset \mathcal V$.
\end{definition}

The essential point of this definition is that the domains of the family $\mathcal U$ are neither too large nor too small in those of the family $\mathcal V$. For families of pointed domains of higher connectivity, two extra conditions are required relating to certain hyperbolic geodesics of the family $\mathcal U$. See \cite{Com11} for details.

\begin{lemma}
\label{domainstoexthypradcts}
Let $I \subset \R$ be non-empty, $\mathcal U=\{(U_h,v_h) \}_{h\in I}$ be a sequence of pointed Jordan domains, and $\mathcal V=\{(V_h,v_h) \}_{h\in I}$ be a sequence of pointed hyperbolic simply connected domains with the same base points, both indexed over $I$.  If $\pt \sqsubset \mathcal U \sqsubset \mathcal V$, $\mathcal V$ varies continuously in the \Car topology over $I$, and $\partial U_h$ is a continuously varying family of Jordan curves on $I$, then $R^{ext}_{(V_h,v_h)}U_h$ is continuous on $I$.
\end{lemma}

Before embarking on the proof, we observe that, since both families $\mathcal U$ and $\mathcal V$ have the same basepoints, it follows from Lemma \ref{jordancurvestodomainscts} and the fact that $\mathcal V$ varies continuously in the \Car topology over $I$, that $\mathcal U$ also  varies continuously in the \Car topology over $I$. However, we do not need to make use of this in the proof below.

\begin{proof}
As $\partial U_h$ is a continuously varying family of Jordan curves, let $F:\mathbb{T} \times I\rightarrow \mathbb{C}$ be a continuous mapping, injective in the first coordinate where, for each $h$ fixed, $F(t,h)$ is a parametrization of $\partial U_h$. We first need to uniformize the domains $V_h$ by mapping to the unit disc $\D$ where we can compare hyperbolic distances directly. So let
$\phi_h$ be the unique normalized Riemann map from $V_h$ to $\D$ satisfying $\phi_h(v_h)=0$, $\phi_h'(v_h)>0$.

Since $\pt \sqsubset \calU \sqsubset \calV$ there exists $K \ge 1$ such that $R^{ext}_{(V_h,v_h)}U_h\leq K$ and thus $\phi_h(U_h)\subset \Delta_\D(0,K) = \mathrm{D}\left (0, \tfrac{e^K-1}{e^K+1}\right )$. Also, for any $h_0 \in I$, we know from Theorem  \ref{theorem1.2} that $\phi_h$ converges to $\phi_{h_0}$ locally uniformly on $V_{h_0}$ as $h \to h_0$ since $(V_h,v_h)\rightarrow (V_{h_0},v_{h_0})$ in the sense of Definition \ref{ContCon} above. Now, set ${\tilde \phi}(z,h)=\phi_h(z)$.

\begin{claim}
For all $h_0 \in I$ and $z_0 \in V_{h_0}$, ${\tilde \phi}(z,h)$ is jointly continuous in $z,h$ on a suitable neighborhood of $(z_0,h_0)$. 
\end{claim} 

\begin{claimproof} Let $\epsilon>0$.  Let $\{h_n \}$ be a sequence in $I$ which converges to $h_0$ and $\{z_n \}$ be a sequence in $V_{h_0}$ which converges to $z_0$. Using \emph{(1)} of \Car convergence (Definition \ref{CarConv}) and the fact that $V_{h_0}$ is open, we have that $z_n \in V_{h_n}$ for all sufficiently large $n$.  Then, for $n$ sufficiently large so that $z_n$ and $h_n$ are sufficiently close to $z_0$ and $h_0$, respectively, since  $\phi_h$ converges to $\phi_{h_0}$ locally uniformly on $V_{h_0}$ and $\phi_{h_0}$ is continuous, we have 
\begin{align*}
|{\tilde \phi}(z_n,h_n)-{\tilde \phi}(z_0,h_0)|&=|\phi_{h_n}(z_n)-\phi_{h_0}(z_0)| \\
&\leq |\phi_{h_n}(z_n)-\phi_{h_0}(z_n)|+|\phi_{h_0}(z_n)-\phi_{h_0}(z_0)| \\
&<\frac{\epsilon}{2}+\frac{\epsilon}{2} \\
&=\epsilon
\end{align*}
which proves the claim. \end{claimproof}

Using this claim, if we now define $\psi(t,h):={\tilde \phi}(F(t,h),h)$, we have that $\psi(t,h)$ is jointly continuous in $t$ and $h$ on $\T \times I$. 

Now let $h_0 \in I$ be arbitrary and let $\{h_n\}$ be any sequence in $I$ which converges to $h_0$. If we write $R_n=R^{ext}_{(V_{h_n},v_{h_n})}U_{h_n}$ and $R_0=R^{ext}_{(V_{h_0},v_{h_0})}U_{h_0}$, we then wish to show that $R_n \rightarrow R_0$ as $n \to \infty$. As $\pt \sqsubset \calU \sqsubset \calV$, we may choose a subsequence $\{R_{n_k} \}$ which converges using Definition \ref{BoundedContainment} to some finite limit in $[0,K]$.  If we can show that the limit is $R_0$, we will have completed the proof. In view of Lemma \ref{equivalenthypradformulation}, for each $k$, we have that $R_{n_k}$ is attained at some $z_{n_k}\in \partial U_{n_k}$, so we may write $R_{n_k}=\rho_{V_{h_{n_k}}}(v_{h_{n_k}},z_{n_k})=\rho_\D(0,{\tilde \phi}(z_{n_k},h_{n_k}))$.  Now $z_{n_k}=F(t_{n_k},h_{n_k})$ for some $t_{n_k}\in \mathbb{T}$, so $R_{n_k}=\rho_\D(0,\psi(t_{n_k},h_{n_k}))$. As $h_{n_k}\rightarrow h_0$, applying the compactness of $\T$ and passing to a further subsequence if necessary, we have that $(t_{n_k},h_{n_k})\rightarrow (t_0,h_0)$ for some $t_0 \in \mathbb{T}$.  Observe that there is no loss of generality in passing to such a further subsequence.  

\begin{claim}
$R_0 = \rho_{\D}(0,\psi(t_{0},h_{0})) = \rho_{V_{h_0}}(v_{h_0}, F(t_0, h_0))$ 
\end{claim}

\begin{claimproof}Suppose not. Since $\partial U_h$ is a continuously varying family of Jordan curves on $I$, $F(t, h_0) \in \partial U_{h_0}$ for any $t \in \T$. In view of Lemma \ref{equivalenthypradformulation}, this means that the external hyperbolic radius for $U_{h_0}$ is not attained at  $F(t_0, h_0)$ and so 
there must exist $\tilde t_0 \in \mathbb{T}$ such that $\rho_{V_{h_0}}(v_{h_0}, F(t_0, h_0)) < \rho_{V_{h_0}}(v_{h_0}, F(\tilde t_0, h_0))$ i.e. $\rho_\D(0, \psi(t_0, h_0)) < \rho_\D(0, \psi(\tilde t_0, h_0))$ whence
$|\psi(t_0,h_0)|<|\psi({\tilde t_0}, h_0)|$. Choose a sequence $\{({\tilde t_{n_k}},h_{n_k}) \}$ in $\mathbb{T}\times I$ which converges to $({\tilde t_0}, h_0)$.  Then by joint continuity of $\psi$ there exists a $k_0 \in \N$ such that for all $k\geq k_0$ we have that $|\psi({\tilde t_{n_k}},h_{n_k})|>|\psi(t_{n_k},h_{n_k})|$, which contradicts the fact that $R_{n_k}=\rho_{\D}(0,\psi(t_{n_k},h_{n_k}))$, again using Lemma \ref{equivalenthypradformulation}. This completes the proof of both the claim and the lemma. \hspace{3.2cm}\end{claimproof}  \end{proof}

\begin{figure}[htb]
\label{siegeldiscjuliaset1}
\scalebox{.85}{\includegraphics{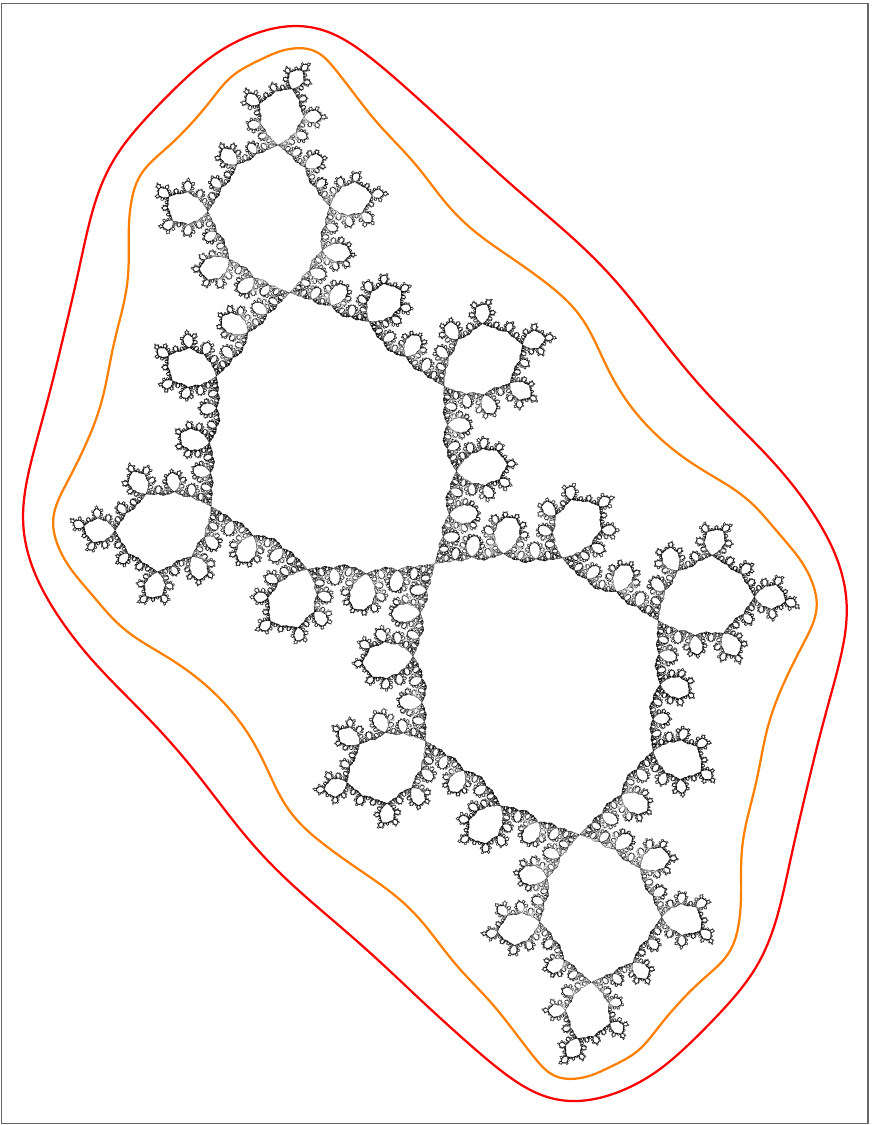}}
\begin{picture}(0.01,0.01)
     \put(-50,110){\color{orange}$\partial V_h$}
     \put(-80,105){\color{red}$\partial V_{2h}$}
 \end{picture}
\caption{The filled Julia Set $\calK$ for $P$ with the Green's Lines $\partial V_h = \{z: G(z)=h \}$ and $\partial V_{2h} = \{z: G(z)=2h \}$}
\end{figure}

\label{PhaseIIKappa}
Recall that we had 
$P_\lambda = \lambda z(1-z)$ where $\lambda= e^{\frac{2\pi i (\sqrt{5}-1)}{2}}$. For $\kappa \ge 1$ we then defined $P=\frac{1}{\kappa}{P_\lambda }(\kappa z)$ and let $G$ be the Green's function for this polynomial.  For each $h > 0$ set $V_h:=\{z \in \C \; : \; G(z)<h \}$ - see Figure 2 below for an illustration showing two of these domains.   

\begin{lemma}
\label{bdryvhctsjordancurves}
The family $\{\partial V_h \}_{h>0}$ gives a continuously varying family of Jordan curves.
\end{lemma}

\begin{proof}
Let $P$ be as above, let $\calK$ be the filled Julia set for $P$, and let $\phi:\chat \setminus \calK \rightarrow \chat \setminus {\overline \D}$ be the associated B\"ottcher map.  Then the map $F:\mathbb{T}\times (0,\infty)\rightarrow \C$, $F(e^{i\theta},h)\mapsto \phi^{-1}(e^{h+i\theta})$ is the desired mapping which yields a continuously varying family of Jordan curves.  
\end{proof}

\begin{lemma}
\label{v2htoucar}
$(V_h,0)\rightarrow (U,0)$ as $h\rightarrow 0_+$. 
\end{lemma}

\begin{proof}
By appealing to Definitions \ref{CarConv}, \ref{ContCon} and Theorem \ref{EquivCarConv}, we can make use of the \Car kernel version of \Car convergence to prove this. So let $h_n$ be any sequence of positive numbers such that $h_n\rightarrow 0$ as $n \to \infty$. From above, we will then be done if we can show that  the \Car Kernel of $\{(V_{h_n},0) \}_{n=1}^\infty$ as well as that of every subsequence of this sequence of pointed domains is $U$.

Let $\{(V_{h_{n_k}},0) \}_{k=1}^{\infty}$ be an arbitrary subsequence of $\{(V_{h_n},0) \}_{n=1}^{\infty}$ (which could possibly be all of  $\{(V_{h_n},0) \}_{n=1}^\infty$) and let 
$W$ be the \Car Kernel of this subsequence $\{(V_{h_{n_k}},0) \}_{k=1}^{\infty}$. Since $U \subset V_{h}$ for every $h >0$, clearly $U\subseteq W$. To show containment in the other direction, let $z\in W$ be arbitrary and construct a path $\gamma$ from $0$ to $z$ in $W$. By definition of $W$ as the \Car kernel 
of the domains $\{(V_{h_{n_k}},0) \}_{k=1}^{\infty}$, the track $[\gamma]$ is contained in $V_{h_{n_k}}$ for all $k$ sufficiently large.  From this it follows that the iterates of $P$ are bounded on $W$ which immediately implies that $W \subset \calK$.  Since $W$ is open, $W \subset \text{int } \calK$. Moreover, since $W$ is connected, $W$ is then contained in a Fatou component for $P$ and, since $0 \in W$, $W \subseteq U$.  Since we have already shown $U \subseteq W$, we have $W=U$ as desired. 
\end{proof} 

\label{overviewpii2}As in the discussion in the proof of Phase I in Chapter 4 just before Lemma \ref{localdistortion}, let $\calK$ be the filled Julia set for $P$ and let $U$ be the Siegel disc about $0$ for $P$.  Again, for $R>0$, define $U_R:= \Delta_U(0,R)$.  

For the remainder of this chapter we will be working extensively with these hyperbolic discs $U_R$ of radius $R$ about $0$ in $U$. At this point we choose $0 < r_0$ and restrict ourselves to $R \ge r_0$ (we will also impose an upper bound on $R$ just before stating the Target Lemma (Lemma \ref{TargetLemma})).\label{r0Def}

\label{overviewpii3}Again let $\psi:U\rightarrow \D$ be the unique normalized Riemann map from $U$ to $\D$ satisfying $\psi(0)=0$, $\psi'(0)>0$.  For $h > 0$ let $\psi_{2h}:V_{2h}\rightarrow \D$ be the unique normalized Riemann map from $V_{2h}$ to $\D$ satisfying $\psi_{2h}(0)=0$, $\psi_{2h}'(0)>0$.  Set
${\tilde R}=R_{(V_{2h},0)}^{int}U_R$ and define ${\tilde V_{2h}}=\Delta_{V_{2h}}(0,{\tilde R})$. Let \label{phi2h}$\phi_{2h}:{\tilde V_{2h}}\rightarrow V_{2h}$ be the unique conformal map from ${\tilde V_{2h}}$ to $V_{2h}$ normalized so that $\phi_{2h}(0)=0$ and $\phi_{2h}'(0)>0$. An important fact to note is that ${\tilde V_{2h}}$ is round in the conformal coordinates of $V_{2h}$, i.e. $\psi_{2h}(\tilde V_{2h})$ is a disc (about $0$). This is an essential point we will be making use of later in the `up' portion of Phase II. 
We now prove a small lemma concerning this disc ${\tilde V_{2h}}$.

\begin{lemma} For $R \ge r_0$, we have the following:
\label{hypradv2hbddbelow}
\begin{enumerate}
\item There exists $d_0$, \label{dnaught} determined by $\kappa$ and $r_0$ such that 
\begin{align*}
\mathrm{d}(0,\partial U_R)\geq d_0
\end{align*}
(where where $\mathrm{d}(0,\partial U_R)$ denotes the Euclidean distance from $0$ to $\partial U_R$),

\item Given any finite upper bound $h_0 \in (0,\infty)$, there exists $\rho_0$, \label{rhonaught} determined by $r_0$ and $h_0$, such that, for all $h\in (0,h_0]$, we have that the hyperbolic radius $R_{(V_{2h},0)}{\tilde V_{2h}}$  of ${\tilde V_{2h}}$ in $V_{2h}$ about $0$ satisfies 
\begin{align*}
R_{(V_{2h},0)}{\tilde V_{2h}}\geq \rho_0.
\end{align*}
\end{enumerate}
\end{lemma}

\begin{proof}
Since $R \ge r_0$, we have that $\partial U_R$ is the image under $\psi^{-1}$ of the circle ${\mathrm C}(0, s)$ in $\D$ where $s \ge s_0 := \tfrac{e^{r_0} - 1}{e^{r_0} + 1}$. \emph{(1)} then follows on applying the Koebe one-quarter theorem (Theorem \ref{Koebe}).

For \emph{(2)}, using Lemma \ref{equivalenthypradformulation}, since $\partial{\tilde V_{2h}}$ is the hyperbolic `incircle' about 0 of $\partial U_R$ in the hyperbolic metric of $V_{2h}$, we have that for all $h\in (0,h_0]$, there exists $z_h \in \partial{\tilde V_{2h}} \cap \partial U_R$. By \emph{(1)} we have $|z_h|\geq d_0$.  On the other hand, as the domains $\{V_{2h} \}_{h\in (0,h_0] }$ are increasing in $h$, there exists $D_0$ depending only on $\kappa$ and $h_0$ such that for all $z\in U$, and for all $h \in (0,h_0]$, we have $\delta_{V_{2h}}(z)\leq D_0$ (where $\delta_{V_{2h}}(z)$ is the Euclidean distance from $z$ to $\partial V_{2h}$).  Letting $\rho_h$ be the hyperbolic radius about 0 of ${\tilde V_{2h}}$ in $V_{2h}$, we have 
\begin{align*}
\rho_h = \int_{\gamma}\mathrm{d}\rho_{V_{2h}(z)},
\end{align*}
where $\gamma$ is a geodesic segment in $V_{2h}$ from $0$ to $z_h$.  Then, using Lemma \ref{lemma4.3}, we have 
\begin{align*}
\rho_h &= \int_{\gamma}\mathrm{d}\rho_{V_{2h}(z)} \\
&\geq \frac{1}{2}\int_{\gamma}\frac{1}{\delta_{V_{2h}}(z)}|\mathrm{d}z| \\
&\geq \frac{1}{2D_0}l(\gamma) \\
&\geq \frac{1}{2D_0}|z_h| \\
&\geq \frac{d_0}{2D_0},
\end{align*}
from which the desired lower bound follows by setting $\rho_0=\frac{d_0}{2D_0}$ (note that in the above we use $l$ to denote Euclidean arc length). Finally, the fact that $\rho_0$ does not depend on the scaling factor $\kappa$ follows immediately by the conformal invariance of the hyperbolic metric of $V_{2h}$ with respect to (Euclidean) scaling. \end{proof}

Now define ${\tilde V_h}:= \phi_{2h}^{-1}(V_h)$ and recall that ${\tilde V_{2h}}=\phi_{2h}^{-1}(V_{2h})$.  Further, define \label{rcheck}${\check R}(h):=R_{(V_{2h},0)}^{ext}V_h$ and note that the function ${\check R}(h)$ does not depend on the scaling factor $\kappa$, while by conformal invariance we have ${\check R}(h) = R_{(\tilde V_{2h},0)}^{ext}\tilde V_h$. 

\begin{lemma}
\label{rcheckcts}	
${\check R}(h)$ is continuous on $(0,\infty)$. 
\end{lemma}	

\begin{proof}
This follows easily from Lemmas \ref{jordancurvestodomainscts}, \ref{domainstoexthypradcts}, and \ref{bdryvhctsjordancurves}. Note that it follows easily from Lemmas \ref{lemma4.3} and \ref{bdryvhctsjordancurves} that the family $(V_h,0)$ is bounded above and below in the family $(V_{2h},0)$ where $h$ is allowed to range over over any closed bounded subset $I$ of $(0, \infty)$.
\end{proof}

Further, we have

\begin{lemma}
\label{exthypradgoestoinfty}
${\check R}(h) \rightarrow \infty$ as $h \rightarrow 0_+$.
\end{lemma}

\begin{proof} By Lemma \ref{v2htoucar} and Theorem \ref{theorem1.2}, $\psi_{2h}$ converges locally uniformly on $U$ to $\psi$ as $h \to 0_+$ (in the sense given in Definition \ref{ContCon}) where we recall that $\psi_{2h}$ and $\psi$ are the suitably normalized Riemann maps from $V_{2h}$ and $U$ respectively to the unit disc (these were introduced in the discussion before Lemma \ref{hypradv2hbddbelow}).

Now let $R > 0$ be large and let $z \in \partial U_R$ where $U_R$ is the hyperbolic disc of radius $R$ about $0$ in $U$ introduced above. From the above, we then have that $\rho_{V_{2h}}(0, z) \ge R-1$ for all $h$ sufficiently small so that by the definition of external hyperbolic radius (Definition \ref{intexthypraddef}) we must have $R_{(V_{2h},0)}^{ext}U_R \ge R-1$. Since $U_R \subset U \subset V_h$, we must have ${\check R}(h):=R_{(V_{2h},0)}^{ext}V_h \ge R_{(V_{2h},0)}^{ext}U \ge R_{(V_{2h},0)}^{ext}U_R \ge R-1$. The result then follows on letting $R$ tend to infinity. \end{proof}

At this point we choose $0 < r_0 < R_0 \le \tfrac{\pi}{2}$ and restrict ourselves to $R \in [r_0, R_0]$. The upper bound $\tfrac{\pi}{2}$ is chosen so that the disc $U_R$ as well as its image under any conformal mapping whose domain of definition contains $U$ is star-shaped (about the image of $0$ - see Lemma \ref{starshaped}).\label{R0Def}

Given $\epsilon_1>0$, using the hyperbolic metric of $U$, construct a $2\epsilon_1$ open neighbourhood of $\partial{\tilde V_{2h}}$ which we will denote by ${\hat N}$. We now fix our upper bound $h_0$ on the value of the Green's function $G(z)$. \label{h0set}
. Recall the lower bound $\rho_0$ on the hyperbolic radius about $0$ of $\tilde V_{2h}$ in $V_{2h}$ as in Part \emph{(2)} of the statement of Lemma \ref{hypradv2hbddbelow}. 
Recall also the scaling factor $\kappa$ and that $U \subset \mathrm{D}(0,\frac{2}{\kappa})$. We now state and prove one of the most important lemmas we need to prove Phase II (Lemma \ref{PhaseII}).

\begin{lemma}
	\label{TargetLemma}
	(Target Lemma) There exist an upper bound ${\tilde \epsilon_1}\in(0,\frac{\rho_0}{2})$ and a continuous function $T:(0,{\tilde \epsilon_1}]\rightarrow (0, \infty)$, both of which are determined by $h_0$ and $r_0$ such that, for all $h\in(0,h_0]$ and $R \in [r_0,  R_0]$, we have
	 
	\begin{enumerate}
		\item $R^{int}_{({\tilde V_{2h}},0)}({\tilde V_{2h}} \setminus {\hat N})\geq T(\epsilon_1)$ for all $\epsilon_1 \in (0,{\tilde \epsilon_1}]$, 
		
		\item $T(\epsilon_1) = \tfrac{1}{2}\log \frac{1}{\epsilon_1} + C_0$ on $(0,{\tilde \epsilon_1}]$ where $C_0 = C_0(h_0, r_0)$, so that in\\ particular we have
		
		\item $T(\epsilon_1)\rightarrow \infty$ as $\epsilon_1\rightarrow 0_+$.
	\end{enumerate}
	\end{lemma}

Before embarking on the proof, we remark that part \emph{(1)} in the statement of Lemma \ref{TargetLemma} above will help us to interpolate in the `during' portion of Phase II. Part \emph{(3)} will be vital for the Fitting Lemma (Lemma \ref{FittingLemma}); it allows us to conclude that $h\rightarrow 0$ as $\epsilon_1 \rightarrow 0_+$ (see the statement of the Fitting Lemma) which is key to controlling the inevitable loss of domain incurred in correcting the errors in our approximations from Phase I (Lemma \ref{PhaseI}). We observe that, even though we require $R \in [r_0,  R_0]$, the upper bound $R_0$ does not appear in the dependencies for $\tilde \epsilon_1$ and the function $T$ above. The reason for this is that we apply the upper bound $R_0 \le \tfrac{\pi}{2}$ in the proof which eliminates the dependence on $R_0$. Lastly, we observe that, although the domain $\tilde V_{2h}$ by definition will depend on $R$ (as will the mapping $\phi_{2h}: \tilde V_{2h} \mapsto V_{2h}$), $\tilde \epsilon_1$ and $T(\epsilon_1)$ do not depend on $R$ since we are obtaining estimates which work simultaneously for all $R \in [r_0,  R_0]$. 

\begin{proof}
We first deduce the existence of ${\tilde \epsilon_1}$.  Regarding the upper bound $\frac{\rho_0}{2}$ on $\tilde \epsilon_1$ in the statement: we note that, if $\epsilon_1$ is too large, then we would actually have ${\tilde V_{2h}} \subset {\hat N}$ so that ${\tilde V_{2h}} \setminus {\hat N}=\emptyset$.  Recall that, by Part \emph{(2)} of Lemma \ref{hypradv2hbddbelow}, we have that $\rho_0$ is such that for all $R \in [r_0, R_0]$ and $h \in (0, h_0]$, we have $R_{(V_{2h},0)}{\tilde V_{2h}}\geq \rho_0$.  Using the Schwarz Lemma for the hyperbolic metric (e.g. \cite{CG} Theorem I.4.1 or I.4.2), we see that $R^{int}_{(U,0)}{\tilde V_{2h}} \geq R_{(V_{2h},0)}{\tilde V_{2h}}\geq \rho_0$, so setting ${\tilde \epsilon_1}:=\frac{\rho_0}{4} < \frac{\rho_0}{2}$ implies that, if $\epsilon_1 \le \tilde \epsilon_1$, then ${\tilde V_{2h}} \setminus {\hat N}\neq\emptyset$. Note that, in view of Lemma \ref{hypradv2hbddbelow}, since $\rho_0$ depends on $r_0$ and $h_0$, the quantity $\tilde \epsilon_1$ inherits these dependencies. \label{FirstEpsilonOne}

Recall the lower bound $d_0  = d_0(\kappa, r_0)$ from Part \emph{(1)} of the statement of Lemma \ref{hypradv2hbddbelow} for which we have $\textit{d}(0,\partial U_R)\geq d_0$ so that, if $\xi \in \partial U_R$, then $|\xi|\geq d_0$.  With the distortion theorems in mind, applied to $\psi_{2h}^{\circ-1}$, we define 
\begin{align*}
r_1&:=\frac{e^{\pi/2}-1}{e^{\pi/2}+1}, \\
D_1&:=\left (\frac{1+r_1}{1-r_1}\right )^2=e^\pi.
\end{align*}
Note that $r_1$ is chosen so that $\mathrm{D}(0,r_1)$ has hyperbolic radius $\tfrac{\pi}{2}$ in $\D$, i.e. $\mathrm{D}(0,r_1)=\Delta_{\D}(0,\tfrac{\pi}{2})$.  By the Schwarz Lemma for the hyperbolic metric, since $\tilde V_{2h} \subset U_R$, $U \subset V_{2h}$, and $R\leq R_0 \le \tfrac{\pi}{2}$, we have 
\begin{align*}
R_{(V_{2h},0)}{\tilde V_{2h}} = R^{ext}_{(V_{2h},0)}{\tilde V_{2h}} \le R^{ext}_{(U,0)}{\tilde V_{2h}} \le R^{ext}_{(U,0)}{U_R} = R_{(U,0)}{U_R} = R \le \frac{\pi}{2}
\end{align*}
(recall that ${\tilde V_{2h}}$ and $U_R$ are round in the conformal coordinates of $V_{2h}$, $U$ respectively so that the internal and external hyperbolic radii coincide). By Lemma \ref{equivalenthypradformulation} and the definition of $\tilde V_{2h}$ given before Lemma \ref{hypradv2hbddbelow}, $\partial U_R$ and $\partial \tilde V_{2h}$ meet, and it then follows by comparing the maximum and minimum values of $|\psi_{2h}^{-1}|$ given by the distortion theorems (Theorem \ref{distortion}) that 
\begin{align}
\label{tildev2hbddbelow}
|z|\geq \tfrac{d_0}{D_1} = d_0 e^{-\pi}, \quad \mbox{if} \:\; z \in \partial {\tilde V_{2h}}.
\end{align}

Now suppose $\zeta_0 \in \partial {\tilde V_{2h}}$ and let $\epsilon_1 \in (0, \tilde \epsilon_1]$.  If $\zeta \in \overline \Delta_U(\zeta_0,2\epsilon_1)$, we wish to find an upper bound on the Euclidean distance from $\zeta$ to $\zeta_0$. Let $\gamma_0$ be a geodesic segment in $U$ from $\zeta_0$ to $\zeta$.  Then, using Lemma \ref{lemma4.3} and the fact that $U \subset \mathrm{D}(0,\frac{2}{\kappa})$, we calculate
\begin{align*}
2\epsilon_1 & \ge \int_{\gamma_0}\mathrm{d}\rho_U \\
&\geq\frac{1}{2}\int_{\gamma_0} \frac{|\mathrm{d}w|}{\delta_U(w)} \\
&\geq \frac{\kappa}{4}\int_{\gamma_0}|\mathrm{d}w| \\
&=\frac{\kappa}{4} \,l(\gamma_0) \\
&\geq \frac{\kappa}{4}|\zeta-\zeta_0|
\end{align*}
where $l(\gamma_0)$ is (as usual) the Euclidean arc length of $\gamma_0$.  Thus $|\zeta-\zeta_0| \le \frac{8}{\kappa}\epsilon_1$.  As $\zeta_0$, $\zeta$ were arbitrary, this implies that 
\begin{align}
\label{8kineq}
\overline \Delta_U(\zeta_0,2\epsilon_1)\subset  \overline{\mathrm{D}}\left (\zeta_0,\frac{8}{\kappa}\epsilon_1 \right ) \quad \mbox{for any} \:\; \zeta_0 \in \partial {\tilde V_{2h}}.
\end{align}

\begin{figure}[htb]
\label{TargetLemmaPicture}
\begin{center}
\scalebox{.405}{\includegraphics[frame]{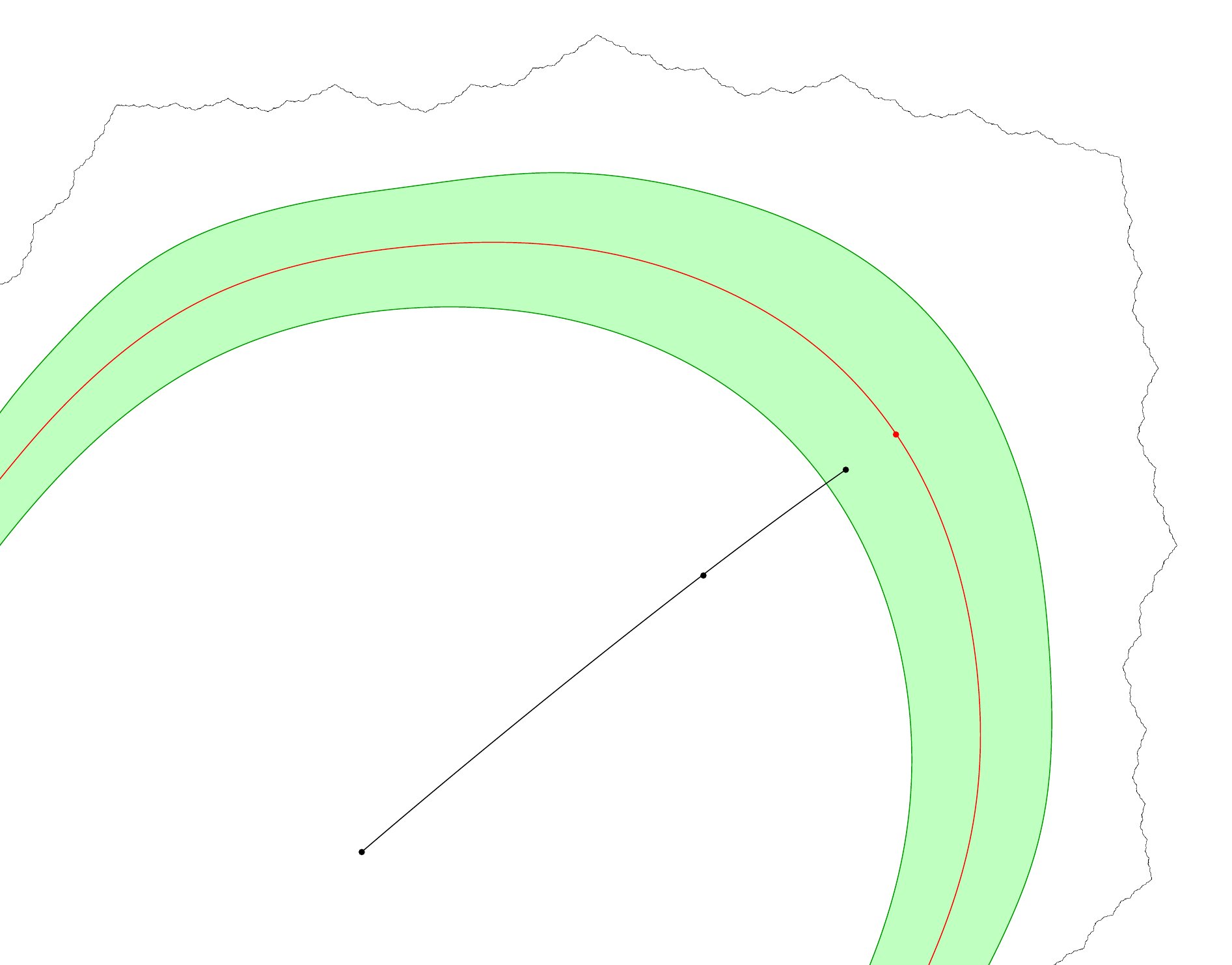}}
\begin{picture}(0.01,0.01)
       \put(-80,42){\small $0$}
       \put(25,132){\small $w$}
       \put(70,165){\small $z_0$}
       \put(91,176){\small \color{red}$z$}
       \put(-18,80){$\gamma$}
       \put(114,120){\color{red}$\partial \tilde V_{2h}$}
       \put(30,225){\large \color{Green}$\hat N$}
       \put(122,220){\LARGE$U$}
      \end{picture}
\caption{Finding a lower bound for $\rho_{{\tilde V_{2h}}}(0,z_0)$}
\end{center}
\end{figure}

Now we aim to specify the value of the function $T(\epsilon_1)$.  Choose a point $z_0 \in \tilde V_{2h} \cap \hat N = \tilde V_{2h} \setminus (\tilde V_{2h} \setminus \hat N)$.  Pick $z \in \partial {\tilde V_{2h}}$ which is closest to $z_0$ in the hyperbolic metric of $U$ (see Figure 4 for an illustration). Then $\rho_U(z_0,z)\leq 2\epsilon_1$ which by  \eqref{8kineq} which we just showed above, implies $|z_0-z|\leq \tfrac{8}{\kappa}\epsilon_1$.  Note that as $|z|\geq \frac{d_0}{D_1}$ by \eqref{tildev2hbddbelow} above, using the reverse triangle inequality, we have that 
\begin{align} 
\label{nhatineq}
|z_0| &\geq \frac{d_0}{D_1}-\frac{8}{\kappa}\epsilon_1.
\end{align}

\label{FirstRestriction}
Note also that, in order to make sure that 
$T(\epsilon_1)$ is defined and positive on $(0, \tilde \epsilon_1]$, it will be essential (because we will be taking the difference of the logs in the two terms in this quantity) that $\frac{d_0}{D_1}-\frac{8}{\kappa}\epsilon_1 > 0$, so 
we may need to make $\tilde \epsilon_1$ smaller if needed so that $\tilde \epsilon_1<\frac{\kappa d_0}{8D_1}$. Since the constant $d_0$ is determined by $\kappa$ and the lower bound $r_0$ for $R$, $\tilde \epsilon_1$ will then be determined by these same constants as well as $h_0$ in view of our earlier discussion on $\tilde \epsilon_1$ above (we will argue later that the dependence on the scaling factor $\kappa$ can be removed). Now let $\gamma$ be a geodesic segment in ${\tilde V_{2h}}$ from $z_0$ to $0$.  If $w \in [\gamma]$, since $|z_0-z|\leq \frac{8}{\kappa}\epsilon_1$ from above, we have 
\begin{align*}
\delta_{{\tilde V_{2h}}}(w)&\leq |w-z| 
\leq |w-z_0|+|z_0-z| 
\leq |w-z_0|+\frac{8}{\kappa}\epsilon_1.
\end{align*}
So, once more using Lemma \ref{lemma4.3}, 
\begin{align*}
\rho_{{\tilde V_{2h}}}(0,z_0) &= \int_{\gamma}\mathrm{d}\rho_{\tilde V_{2h}(w)}\\ 
&\geq \frac{1}{2}\int_{\gamma}\frac{|\mathrm{d}w|}{\delta_{{\tilde V_{2h}}}(w)} \\
&\geq \frac{1}{2}\int_{\gamma}\frac{|\mathrm{d}w|}{|w-z_0| + \frac{8}{\kappa}\epsilon_1}.
\end{align*}

Now parametrize $\gamma$ by 
$w=\gamma(t)=z_0+r(t)e^{i\theta(t)}$, for $t \in [0,1]$ and note that, as $\gamma$ is a geodesic segment in ${\tilde V_{2h}}$ from $z_0$ to $0$, $r(1)e^{i\theta(1)}=-z_0$.  Since $\gamma$ is not self-intersecting, we have $r(t)>0$ for all $t\in (0,1]$. Then, using \eqref{nhatineq},

\begin{align*}
\frac{1}{2}\int_{\gamma}\frac{|\mathrm{d}w|}{|w-z_0|+\frac{8}{\kappa}\epsilon_1} &= \frac{1}{2}\int_0^1\frac{|r'(t)e^{i\theta(t)}+i\theta'(t)r(t)e^{i\theta(t)}|}{r(t)+\frac{8}{\kappa}\epsilon_1} \mathrm{d}t \\
&\geq \frac{1}{2}\int_0^1\frac{|r'(t)|}{r(t)+\frac{8}{\kappa}\epsilon_1} \mathrm{d}t, \\
&\geq \frac{1}{2} \left| \int_0^1\frac{r'(t)}{r(t)+\frac{8}{\kappa}\epsilon_1} \mathrm{d}t \right|\\
&=\frac{1}{2}\int_0^{|z_0|}\frac{1}{u+\frac{8}{\kappa}\epsilon_1}\mathrm{d}u \\
&\geq \frac{1}{2}\int_0^{\frac{d_0}{D_1}-\frac{8}{\kappa}\epsilon_1}\frac{1}{u+\frac{8}{\kappa}\epsilon_1} \mathrm{d}u \\
&= \frac{1}{2}\int_{\frac{8}{\kappa}\epsilon_1}^{\frac{d_0}{D_1}}\frac{1}{x}\mathrm{d}x \\
&= \frac{1}{2}\left (\log\left (\frac{d_0}{D_1}\right )-\log \left (\frac{8}{\kappa}\epsilon_1  \right )\right )\\
&= \frac{\log d_0 - \pi - \log \epsilon_1 - \log 8 + \log \kappa}{2}\\
&= \frac{1}{2}\log \frac{1}{\epsilon_1} + \frac{\log d_0 - \pi - \log 8 + \log \kappa}{2}.\\
\end{align*}

Taking an infimum over all $z_0 \in \tilde V_{2h} \setminus \hat N$ and applying the definition (Definition \ref{intexthypraddef}) of internal hyperbolic radius (which in particular does not require that ${\tilde V_{2h}} \setminus {\hat N}$ be connected) and setting $T(\epsilon_1)=\frac{1}{2}\log \tfrac{1}{\epsilon_1} + \frac{1}{2}(\log d_0 - \pi - \log 8 + \log \kappa) $ (which is strictly positive in view of the definition of $\tilde \epsilon_1$ given in the discussion after \eqref{nhatineq}) gives the desired lower bound on $R^{int}_{({\tilde V_{2h}},0)}({\tilde V_{2h}} \setminus {\hat N})$. \label{DefofT} Explicitly, the function $T(\epsilon_1)$ above is determined by $\kappa$, $r_0$, and $h_0$ (this last being due to the requirement that $\epsilon_1 \le \tilde \epsilon_1$). However, similarly to the end of the proof of Lemma \ref{hypradv2hbddbelow}, we may eliminate the dependence on $\kappa$ (for both $\tilde \epsilon_1$ and $T$) given the conformal invariance of the hyperbolic metric of $\tilde V_{2h}$ with respect to the scaling factor $\kappa$. \end{proof}

Before turning to the Fitting Lemma (Lemma \ref{FittingLemma}), we prove a small lemma from real analysis. 

\begin{lemma}
\label{realanalysislemma}
Let $b >0$ and let $\phi:(0,b] \rightarrow [0,\infty)$ be a continuous function such that $\phi(x)\rightarrow \infty$ as $x \rightarrow 0_+$.  Then, for all $y\ge \min \{\phi(x):x \in (0,b] \}$, if we set 
\begin{align*}
x(y):=\min \{x: \phi(x) = y\},
\end{align*}
we have $x(y)\rightarrow 0$ as $y\rightarrow \infty$.  
\end{lemma}

\begin{proof} We note first that, since $\phi$ is continuous while $\phi(x)\rightarrow \infty$ as $x \rightarrow 0_+$, $\phi$ attains its minimum on $(0,b]$. Also, in view of the intermediate value theorem, for each $y\ge\min\{\phi(x):x \in (0,b] \}$, the set $\{x: \phi(x) = y\}$ is non-empty.  Because $\phi(x)\rightarrow \infty$ as $x \rightarrow 0_+$, the infimum of this set is strictly positive, and, as $\phi$ is continuous, we must have that $\phi(x(y)) = y$ so that this infimum is attained and is in fact a minimum. Suppose now the conclusion is false and that there exists a sequence $\{y_n\}_{n=1}^{\infty}$ such that $y_n \rightarrow \infty$ but $x(y_n)\not\rightarrow 0$ as $n \rightarrow \infty$.  Set $x_n:=x(y_n)$.  Since $x_n\not\rightarrow 0$, we can take a convergent subsequence $\{x_{n_k}\}_{k=1}^{\infty}$ which converges to a limit $x_0 > 0$.  This then leads to a contradiction to the continuity of $\phi$ at $x_0$. 
\end{proof}

Recall the quantity ${\check R}(h):=R_{(V_{2h},0)}^{ext}V_h = R_{(\tilde V_{2h},0)}^{ext}\tilde V_h$ which was introduced before Lemma \ref{rcheckcts} and the $2\epsilon_1$ open neighbourhood $\hat N$ of $\partial \tilde V_{2h}$ which was introduced before the statement of Lemma \ref{TargetLemma}. We now state and prove

\begin{lemma}
\label{FittingLemma}
(The Fitting Lemma) There exists $\tilde \epsilon_1 > 0$  and a function $h: (0 ,\tilde \epsilon_1] \mapsto (0, \infty)$ both of which are determined by $h_0$, $r_0$ for which the following hold:

\begin{enumerate}
\item $\tilde V_{h(\epsilon_1)} \subset \tilde V_{2h(\epsilon_1)} \setminus \hat N$ for each $0 < \epsilon_1 \le \tilde \epsilon_1$,
\item $h(\epsilon_1)\rightarrow 0$ as $\epsilon_1 \rightarrow 0_+$.  
\end{enumerate}
\end{lemma}

\begin{proof}
We first apply the Target Lemma (Lemma \ref{TargetLemma}) to find $\tilde \epsilon_1 > 0$ and a function $T: (0, \tilde \epsilon_1] \mapsto (0, \infty)$ as above, both of which are determined by $h_0$, $r_0$. We now show how to use the function $T$  to define an appropriate value $h$ of the Green's function for which \emph{(1)} above holds which will then allow us to do the interpolation in the `during' part of the proof of Phase II (Lemma \ref{PhaseII}). Our first step is to fix a (possibly) smaller value of $\tilde \epsilon_1$ which still has the same dependencies as in the statement of Lemma  \ref{TargetLemma}. Since by \emph{(3)} in the statement of Lemma  \ref{TargetLemma} $T(\epsilon_1) \to \infty$ as $\epsilon_1 \to 0_+$, we can make $\tilde \epsilon_1$ smaller if needed so as to ensure that 
\begin{align}
\label{ExistenceCriterion}
\min_{0 < \epsilon_1 \le \tilde \epsilon_1} T(\epsilon_1) \:\ge \min_{0 < h \le h_0}{\check R}(h).
\end{align} 
Note that $T(\epsilon_1)$ and ${\check R}(h)$ attain their minimum values above in view of the fact that $T(\epsilon_1)$ is continuous on $(0, \tilde \epsilon_1]$ by Lemma \ref{TargetLemma} and ${\check R}(h)$ is continuous on $(0, h_0]$ by Lemma \ref{rcheckcts} while $T(\epsilon_1) \to \infty$ as $\epsilon_1 \to 0_+$ and ${\check R}(h) \to \infty$ as $h \to 0_+$ by \emph{(3)} in the statement of Lemma  \ref{TargetLemma} and Lemma \ref{exthypradgoestoinfty} respectively.
We now define the function $h$ of the variable $\epsilon_1$ on the interval $(0, \tilde \epsilon_1]$ by setting, for each $0 < \epsilon_1 \le \tilde \epsilon_1$,  
\begin{equation}
\label{defofh1}
h(\epsilon_1) := \min\{h\in (0, h_0] \: : \:{\check R}(h)=T(\epsilon_1) \}.
\end{equation}
Note that in view of \eqref{ExistenceCriterion} above, again since ${\check R}$ is continuous and ${\check R}(h) \to \infty$ as $h \to 0_+$, using the intermediate value theorem, the set of which we are taking the minimum above will be non-empty, and so this function is well-defined. It also follows that the set $\{h\in (0, h_0] \: : \:{\check R}(h)=T(\epsilon_1) \}$ has a positive infimum which, by the continuity of 
${\check R}(h)$, is attained and is thus in fact a minimum and moreover
\begin{align}
\label{defofh2}
{\check R}(h(\epsilon_1)) = T(\epsilon_1).
\end{align}

The right hand side of \eqref{ExistenceCriterion} above depends only on $h_0$. Hence, the upper bound $\tilde \epsilon_1$ and the function $T$ from the statement of Lemma \ref{TargetLemma} will still only depend on $h_0$ and $r_0$ and the dependencies in the statement of that lemma thus remain unaltered. Lastly, we observe that this function $h$ above is then determined by $h_0$, and $r_0$ in view of \eqref{defofh1}.

By \emph{(1)} of the Target Lemma (Lemma \ref{TargetLemma}), for each $0 < \epsilon_1 \le \tilde \epsilon_1$, 
\begin{align}
\label{IntHypRadIneq}
R^{int}_{(\tilde V_{2h}, 0)}(\tilde V_{2h} \setminus {\hat N}) \ge T(\epsilon_1).
\end{align}
On the other hand, by \eqref{defofh2} above, in view of the definition of ${\check R}(h)$ given before Lemma \ref{rcheckcts},
\begin{align}
\label{ExtHypRadIneq}
R^{ext}_{(\tilde V_{2h(\epsilon_1)}, 0)}\tilde V_{h(\epsilon_1)} = T(\epsilon_1).
\end{align}
Thus, using \emph{(1)} of Corollary \ref{HyperbolicAvoidance}, if we set $X = \tilde V_{h(\epsilon_1)}$, $Y = \tilde V_{2h(\epsilon_1)} \setminus {\hat N}$, we have $\tilde V_{h(\epsilon_1)} \subset \tilde V_{2h(\epsilon_1)} \setminus {\hat N}$ (this latter set clearly being closed) and so we obtain \emph{(1)}.

Again by \emph{(3)} of the statement of Lemma \ref{TargetLemma}, $T(\epsilon_1)\rightarrow \infty$ as $\epsilon_1\rightarrow 0_+$. Lemma \ref{exthypradgoestoinfty}, together with the fact that ${\check R}$ is continuous in view of Lemma \ref{rcheckcts}, then ensure that the hypotheses of Lemma \ref{realanalysislemma} are met.  \eqref{defofh1} and Lemma \ref{realanalysislemma} then imply that $h(\epsilon_1)\rightarrow 0$ as $\epsilon_1 \to 0_+$ as desired, which proves \emph{(2)}. 
\end{proof}

As we remarked earlier, the Fitting Lemma will be essential for proving Phase II.  Basically, part \emph{(1)} of the statement says that, for each $0 < \epsilon_1 \le \tilde \epsilon_1$, the domain $\tilde V_{h(\epsilon_1)}$ `fits' inside ${\tilde V_{2h(\epsilon_1)}} \setminus {\hat N}$ which will allow us to apply the Polynomial Implementation Lemma which we will need to correct the error from the Phase I immediately prior to this. 
On the other hand, part \emph{(2)} of the statement says that $h(\epsilon_1) \rightarrow 0_+$ as $\epsilon_1 \rightarrow 0_+$ which, as we will see, is the key to controlling the loss of domain incurred by the correction of the error from the Phase I immediately prior to this and which is the purpose of Phase II.

Observe that getting $\tilde V_h$ to fit inside ${\tilde V_{2h}} \setminus {\hat N}$ as above is easier if the value of $h$ is \emph{large} while ensuring the loss of domain is small requires a value of $h$ which is \emph{small}. Indeed it is the tension between these competing requirements for $h$ which makes proving Phase II so delicate and why the Target and Fitting Lemmas are so essential.  Before we move on to the statement and proof of Phase II, we state one last technical lemma that will be of use to us later.

\begin{lemma}
\label{domainswallowlemma}
Let $D\subset \C$ be a bounded simply connected domain and let $z_0 \in D$. Then for all $\epsilon > 0$ there exists $R_\epsilon>0$ such that if $X$ is any set containing $z_0$ and contained in $D$ such that $R_{(D,z_0)}^{int}X>R_\epsilon$, then $\mathrm{d}(\partial X, \partial D) < \epsilon$.  
\end{lemma}

\begin{proof}
Define $D_{\epsilon}=\{z  \in D \: : \: \mathrm{d}(z,\partial D)\geq \epsilon/2   \}$. Since $D$ is bounded, $D_{\epsilon}$ is a compact subset of $D$ and we can find $R_{\epsilon} > 0$ such that $D_{\epsilon}\subset \Delta_D(z_0, R_{\epsilon})$.  Then, if $X$ is any set containing $z_0$ and  contained in $D$ such that $R_{(D, z_0)}^{int}X \ge R_\epsilon$, by the definition of internal hyperbolic radius (Definition \ref{intexthypraddef}), for every $z \in D \setminus X$ we have $\rho_D(z_0, z) \ge R_\epsilon$ . It then follows that for every $z \in \partial X \cap D = \partial (D \setminus X) \cap D$, we also have $\rho_D(z_0, z) \ge R_\epsilon$ from which it follows that $z \notin D_{\epsilon}$. Since $\partial X = (\partial X \cap D) \cup (\partial X \cap \partial D)$, it follows that $\partial X \subset \{z  \in \overline D \: : \: \mathrm{d}(z,\partial D)< \epsilon/2   \}$ and from  the compactness of the bounded set $\partial X \subset \overline D$ we get $\mathrm{d}(\partial X, \partial D) \le \epsilon/2 < \epsilon$ as desired.
\end{proof}

\section{Statement and Proof of Phase II}

Recall the scaling factor $\kappa\geq 1$ and upper bound $h_0$ on the value of the Green's function from the statement of Lemma \ref{hypradv2hbddbelow}. Recall also the bounds $0 < r_0 < R_0 \le \tfrac{\pi}{2}$ for $R$ and that the upper bound of $\tfrac{\pi}{2}$ was chosen in the discussion before the Target Lemma (Lemma \ref{TargetLemma}) so that $U_R$ as well as its image under any conformal mapping whose domain of definition contains $U$ is star-shaped (about the appropriate image of $0$).

\begin{lemma}
\label{PhaseII}
(Phase II) Let $\kappa$, $h_0$, $r_0$, and $R_0$ be fixed as above. Then there exist an upper bound ${\tilde \epsilon_1}>0$ and a function $\delta:(0, \tilde \epsilon_1]  \rightarrow (0,\frac{r_0}{4})$, with $\delta(x)\rightarrow 0$ as $x\rightarrow 0_+$, both of which are determined by $h_0$, $r_0$, and $R_0$ such that, for all $\epsilon_1\in (0,{\tilde \epsilon_1}]$, there exists an upper bound ${\tilde \epsilon_2}>0$, determined by $\epsilon_1$, $h_0$, and $r_0$, $R_0$, such that, for all $\epsilon_2\in (0,{\tilde \epsilon_2}]$, all $R \in [r_0, R_0]$, and all functions $\calE$ univalent on $U_R$ with $\calE(0)=0$ and $\rho_U(\calE(z),z)<\epsilon_1$ for $z \in U_R$, there exists a $(17+\kappa)$-bounded composition ${\bf Q}$ of quadratic polynomials which depends on $\kappa$, $\epsilon_1$, $\epsilon_2$, $R$, $h_0$, $r_0$, $R_0$, and ${\mathcal E}$ such that 

\begin{itemize}

\item[i)] ${\bf Q}$ is univalent on a neighborhood of $\overline{U}_{R-\delta(\epsilon_1)}$,

\item[ii)] For all $z\in \overline U_{R-\delta(\epsilon_1)}$, we have 
\begin{align*}
\rho_U({\bf Q}(z),\calE(z))<\epsilon_2,
\end{align*} 

\item[iii)] ${\bf Q}(0)=0$.

\end{itemize}

\end{lemma}

Because we will be using the Polynomial Implementation Lemma repeatedly to construct our polynomial composition, we need to interpolate functions outside of $\calK$, the filled Julia set for $P$.  Indeed, as we saw in the Polynomial Implementation Lemma (Lemma \ref{PIL}), the solutions to the Beltrami equation converge to the identity precisely because the supports of the Beltrami data become small in measure.  However, $\calE$ is only defined on a subset of $U$ and hence we will need to map a suitable subset of $U$ on which $\calE$ is defined to a domain which contains $\calK$, and correct the conjugated error using the Polynomial Implementation Lemma.  The trick to doing this is that we choose our subset of $U$ such that the mapping to blow this subset up to $U$ can be expressed as a high iterate of a map which is defined on the whole of the Green's domain $V_h$, not just on this subset.  This will allow us to interpolate outside $\calK$.  Further, we will then use the Polynomial Implementation Lemma twice more to `undo' the conjugating map and its inverse.  

The two key considerations in the proof are controlling loss of domain (which is measured by the function $\delta$ in the statement above), and showing that the error in our polynomial approximation to the function $\calE$ (measured by the quantity $\epsilon_2$ above) is mild and in particular can be made as small as desired. In controlling loss of domain, one main difficulty will arise in converting between the hyperbolic metrics of different domains, $U$ and $V_{2h}$ and we will deal with this by means of the convergence of the pointed domains $(V_{2h},0)$ to $(U,0)$ in the \Car topology as $h$ tends to zero. One last thing is worth mentioning: since this result involves many functions and quantities which depend on one another, the interested reader is encouraged to make use of the dependency tables in the appendices to help keep track of them.

\afterpage{\clearpage}

\begin{figure}[htbp]
  \centering
  \vspace{0cm}
\includegraphics[scale=0.362]{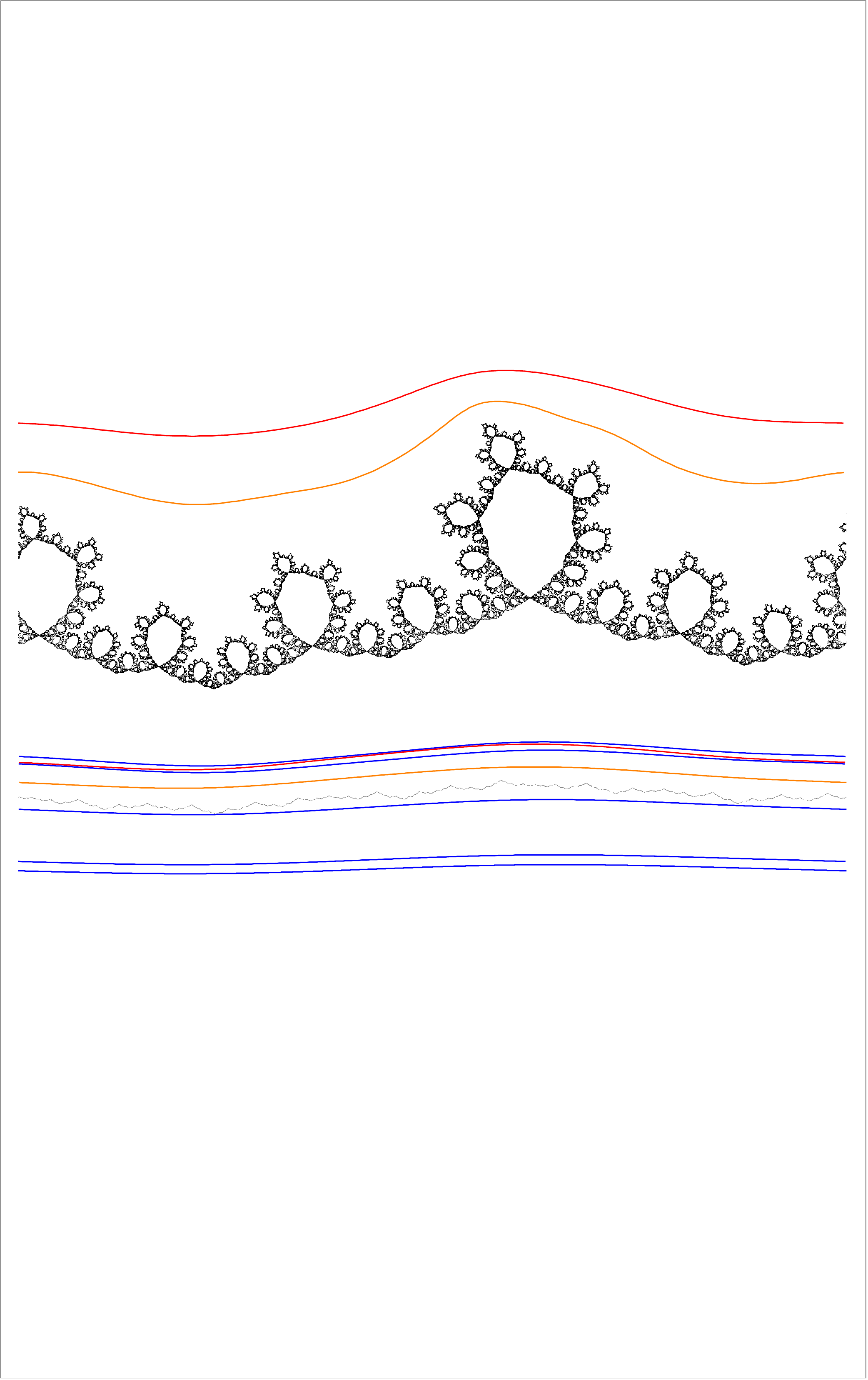}
\caption{The Setup for Phase II in Rotated Logarithmic Coordinates}
  \vspace{0.6cm}

    \unitlength1cm
\begin{picture}(0.01,0.01)
  \put(5.5,16.52){\color{red}\scriptsize$\partial V_{2h}$  }
  \put(-6.05,15.76){\color{orange}\scriptsize$\partial V_{h}$  } 
  \put(5.75,12.65){\scriptsize $\partial U$  } 
  \put(-6.05,11.6){\color{darkblue}\scriptsize $\partial U_R$,  $\partial U_{R'}$} 
  \put(5.5,11.64){\color{red}\scriptsize $\partial \tilde V_{2h}$}
  \put(-6.05,10.35){{\color{orange}\scriptsize $\partial \tilde V_{h}$,}  {\color{gray}\scriptsize $\partial \tilde U$  } }
  \put(5.5,10.43){\color{darkblue}\scriptsize$\partial U_{R''}$  }
  \put(-6.05,10.05){\color{darkblue}\scriptsize $\partial U_{R''-4\epsilon_1}$  }  
  \put(4.9,9.5){\color{darkblue}\scriptsize $\partial U_{R''-5\epsilon_1}$  }  
  \put(-0.1,6.5){\LARGE $U$  } 
 \end{picture}
\end{figure}

\begin{proof}
{\bf Ideal Loss of Domain:}

The techniques for controlling loss of domain will be the Fitting Lemma, and again the fact that $(V_{2h},0)\rightarrow (U,0)$ in the \Car topology as $h$ tends to zero combined with the Target Lemma.  As stated above, we will apply the Polynomial Implementation Lemma to our conjugated version of $\calE$ which will be $\phi_{2h}\circ \calE\circ \phi_{2h}^{-1}$ in what we call the `During' portion of the error calculations. To approximate $\calE$ itself rather than this conjugated version, we then wish to `cancel' the conjugacy, so `During' is bookended by `Up' and `Down' portions, in which we apply the Polynomial Implementation Lemma to get polynomial compositions which are arbitrarily close to $\phi_{2h}$ and $\phi_{2h}^{-1}$, respectively, on suitable domains.  

We begin the proof of Phase II by considering `Ideal Loss of Domain.'  In creating polynomial approximations using Phase I (Lemma \ref{PhaseI}), errors will be created which will have an impact on the loss of domain which occurs.  We first describe the loss of domain that is forced on us before this error is taken into account.  During what follows, the reader might find it helpful to consult Figure 4 below where most of the relevant domains are shown in rotated logarithmic coordinates where the up direction corresponds to increasing distance from the fixed point for $P$ at $0$.

We first turn our attention to controlling loss of domain. Let $R \in [r_0, R_0]$ be arbitrary as in the statement (we consider $R$ for now as varying over the whole of the interval $[r_0, R_0]$ and will fix an (arbitrary) value of $R$ later at the start of the `up' portion of the proof).  Recall the discussion before the statement of Lemma \ref{hypradv2hbddbelow} where we let $\psi:U\rightarrow \D$ be the unique normalized Riemann map from $U$ to $\D$ satisfying $\psi(0)=0$, $\psi'(0)>0$. Recalling the upper bound $h_0$ for the value of the Green's function $G$ for $P$, for $h \in (0, h_0]$ arbitrary, let $\psi_{2h}:V_{2h}\rightarrow \D$ be the unique normalized Riemann map from $V_{2h}$ to $\D$ satisfying $\psi_{2h}(0)=0$, $\psi_{2h}'(0)>0$. Recall also that we had ${\tilde R}=R_{(V_{2h},0)}^{int}U_R$, ${\tilde V_{2h}}=\Delta_{V_{2h}}(0,{\tilde R})$ and $\phi_{2h}:{\tilde V_{2h}}\rightarrow V_{2h}$ which was the unique conformal map from ${\tilde V_{2h}}$ to $V_{2h}$ normalized so that $\phi_{2h}(0)=0$ and $\phi_{2h}'(0)>0$. Now define $R'=R_{(U,0)}^{int}{\tilde V_{2h}}$ and note that the value of this quantity is completely determined by those of $R$ and $h$. We prove the following claim: \label{Rdash}

\begin{claim}
\label{PhaseIIClaim1}
$R-R'\rightarrow 0$ uniformly on $[r_0, R_0]$ as $h \rightarrow 0_+$.
\end{claim}

\begin{claimproof}
By Lemma \ref{v2htoucar}, $(V_{2h},0)\rightarrow(U,0)$ in the \Car topology as $h\rightarrow 0_+$ and thus $\psi_{2h}$ converges locally uniformly to $\psi$ on $U$ in view of Theorem \ref{theorem1.2}.  

Let $\{h_n\}_{n=1}^\infty$ be an arbitrary sequence of positive numbers such that $h_n \rightarrow 0$ as $n \rightarrow \infty$.  By the definitions of ${\tilde V_{2h}}$ and $R'$ and Lemma \ref{equivalenthypradformulation}, there exists $w_{h_n,1} \in \partial{\tilde V_{2h_n}}\cap \partial U_R$ and $w_{h_n,2} \in \partial{\tilde V_{2h_n}}\cap \partial U_{R'}$.  Let $0<s,s'_n,s_n''<1$ be such that $\psi(\partial U_R)=\mathrm{C}(0,s)$, $\psi(\partial U_{R'})=\mathrm{C}(0,s_n')$, and $\psi_{2h_n}(\partial {\tilde V_{2h_n}})=\mathrm{C}(0,s_n'')$.  

Let $\epsilon_0>0$. By the local uniform convergence of $\psi_{2h_n}$ to $\psi$ on $U$, there exists $n_0$, such that for all $n \ge n_0$ 
\begin{align*}
|\psi_{2h_n}(z) - \psi(z)|<\frac{\epsilon_0}{2}, \quad z \in \overline U_{R_0}.
\end{align*}

Thus, for any $n \ge n_0$ and any $R \in [r_0, R_0]$, 
\begin{align*}
|s - s_n''| &= \left ||\psi(w_{h_n,1})| -  |\psi_{2h_n}(w_{h_n,1})| \right| \le |\psi(w_{h_n,1}) - \psi_{2h_n}(w_{h_n,1})| <\frac{\epsilon_0}{2},\\
|s_n'' - s_n'| &= ||\psi_{2h_n}(w_{h_n,2})| - | \psi(w_{h_n,2})|| \le |\psi_{2h_n}(w_{h_n,2}) - \psi(w_{h_n,2})| <\frac{\epsilon_0}{2}
\end{align*}
whence 
\begin{align*}
|s - s_n'| < \epsilon_0.
\end{align*}
Since the sequence $\{h_n\}_{n=1}^\infty$ was arbitrary, the desired uniform convergence then follows on applying the conformal invariance of the hyperbolic metric under $\psi^{-1}$. \end{claimproof}

Now define the {\it Internal Siegel disc}, \label{intsiegeldisc}${\tilde U}:=\phi_{2h}^{-1}(U)$, and set $R''=R_{(U,0)}^{int}{\tilde U}$, noting again that the value of this quantity is completely determined by those of $R$ and $h$.  Next, we show \label{Rdoubledash}

\begin{claim} 
\label{PhaseIIClaim2}
$R-R''\rightarrow 0$  uniformly on $[r_0, R_0]$ as $h \rightarrow 0_+$. 
\end{claim}

\begin{claimproof}
First we show $R_{(V_{2h},0)}^{int}U\rightarrow \infty$ as $h \rightarrow 0_+$ (note that this convergence will be trivially uniform with respect to $R$ on $[r_0, R_0]$ as there is no dependence on $R$). Fix $R_1>0$ and set $X:=U_{R_1}$ and $Y:=U_{R_1+1}$. Then $\psi(X)=\Delta_{\D}(0,R_1)$ and $\psi(Y)=\Delta_{\D}(0,R_1+1)$.  As ${\overline \Delta_{\D}(0,R_1)}\subset \Delta_{\D}(0,R_1+1)$, let $\eta=\mathrm{d}(\partial\Delta_{\D}(0,R_1),\partial\Delta_{\D}(0,R_1+1))>0$.  Now let $z \in \partial Y$ and $w \in \Delta_{\D}(0,R_1)$.  We have that $(V_{2h},0)\rightarrow(U,0)$ as $h\rightarrow 0_+$ in view of Lemma \ref{v2htoucar}, so, by Theorem \ref{theorem1.2}, we have that $\psi_{2h}$ converges to $\psi$ uniformly on compact subsets of $U$ in the sense given in Definition \ref{ContCon}.  Then, for all $h$ sufficiently small, we have 
\begin{align*}
|(\psi_{2h}(z)-w)-(\psi(z)-w)|&= |\psi_{2h}(z)-\psi(z)|\\
&<\frac{\eta}{2} \\
&<\eta \\
&\leq|\psi(z)-w|.
\end{align*}
Thus by Rouch\'e's theorem, since the convergence is uniform and $w \in \Delta_{\D}(0,R_1)$ was arbitrary, $\Delta_{\D}(0,R_1)\subset\psi_{2h}(Y)$.  Then $\psi_{2h}^{-1}(\Delta_{\D}(0,R_1))\subset Y$, so $R^{int}_{(V_{2h},0)}Y\geq R_1$.  We also have that $Y\subset U$ so $R^{int}_{(V_{2h},0)}U\geq R^{int}_{(V_{2h},0)}Y$, and thus $R^{int}_{(V_{2h},0)}U\geq R_1$.  Since $R_1$ was arbitrary, we do indeed have that $R_{(V_{2h},0)}^{int}U\rightarrow \infty$ as $h \rightarrow 0_+$.

For a constant $c > 0$ and a set $X \subset \C$, define the scaled set $cX:= \{z \in \mathbb{C} \: : \: z=cw \text{ for some }w\in X \}$. \label{r2hup} Let $0 < r_{2h} < 1$ be such that $\psi_{2h}({\tilde V_{2h}})=\mathrm{D}(0,r_{2h})$. The quantity $r_{2h}$ then depends on $r_0$, $R_0$, $R$, and $h$ (and thus ultimately on $\epsilon_1$, $R$, $h_0$, $r_0$, and $R_0$ once we make our determination of the function $h = h(\epsilon_1)$ immediately before \eqref{LossofDomain} below), clearly $\frac{1}{r_{2h}}\psi_{2h}({\tilde V_{2h}})=\D$.

By conformal invariance, 
 \begin{align*}
 R_{(\frac{1}{r_{2h}}\psi_{2h} ({\tilde V_{2h}}),0)}^{int}\left (\frac{1}{r_{2h}}\psi_{2h}({\tilde U})\right ) =  R_{(\psi_{2h} ({\tilde V_{2h}}),0)}^{int}(\psi_{2h}({\tilde U})) = R_{({\tilde V_{2h}},0)}^{int}{\tilde U} =R_{(V_{2h},0)}^{int}U .
 \end{align*}
 As $R_{(V_{2h},0)}^{int}U\rightarrow \infty$ as $h \rightarrow 0_+$ from above, it follows that, uniformly on $[r_0, R_0]$,
\begin{align*}
R_{(\frac{1}{r_{2h}}\psi_{2h} ({\tilde V_{2h}}),0)}^{int} \left (\frac{1}{r_{2h}}\psi_{2h}({\tilde U}) \right) =  R_{(V_{2h},0)}^{int}U\rightarrow \infty \quad \mbox{as} \:h \rightarrow 0_+.
\end{align*}

\begin{figure}[htb]

  \includegraphics[width=12.52cm]{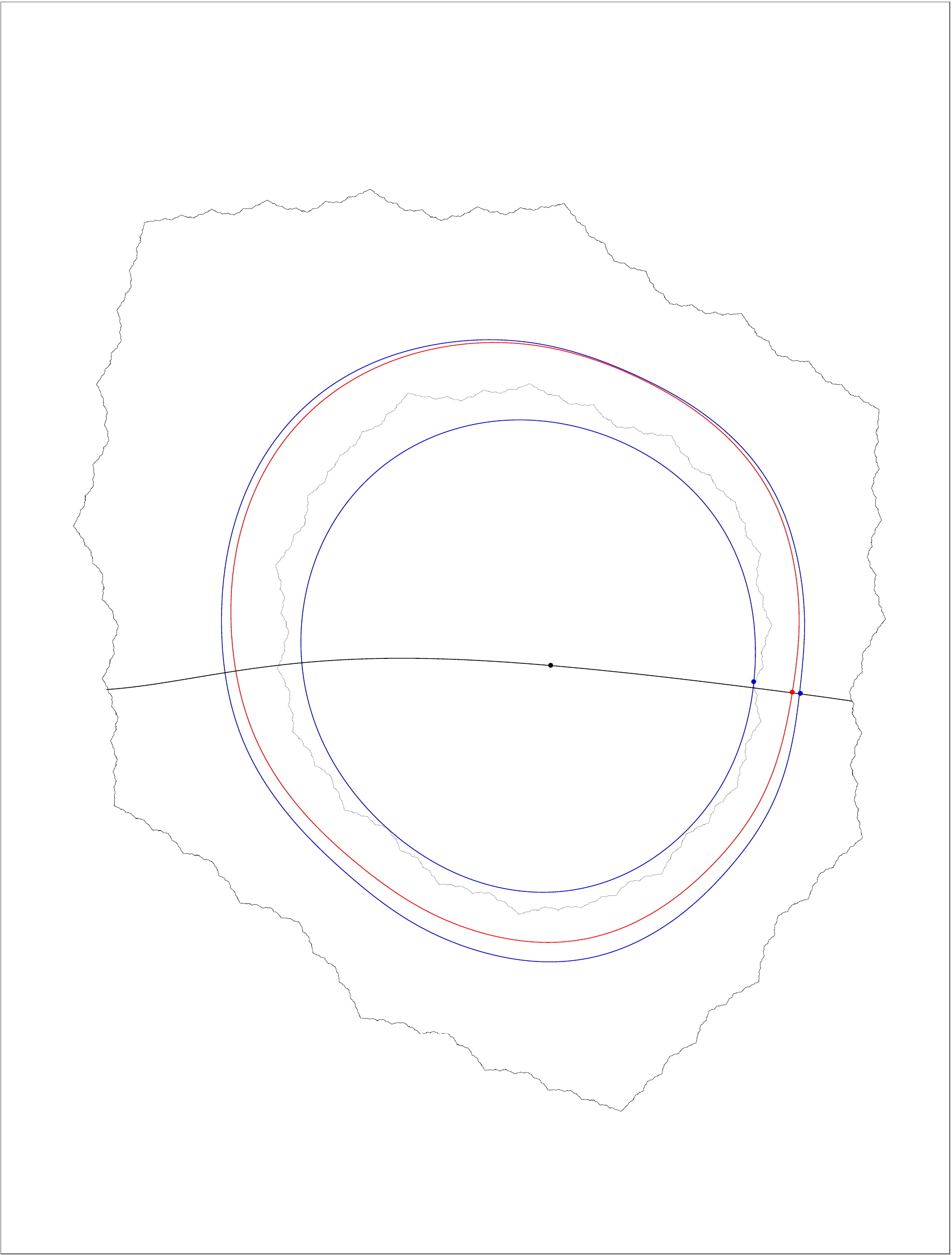}
\begin{picture}(0.01,0.01)
        \put(-210, 228){\footnotesize $\gamma$}
        \put(-159,223.5){\footnotesize $0$}
        \put(-84,216){\color{darkblue}\footnotesize $z$}
        \put(-74,214){\color{red}\footnotesize $w_h$}
         \put(-58,211.5){\color{darkblue}\footnotesize $w$}
     \put(-208,327){\color{red}$\partial \tilde V_{2h}$}
       \put(-221,169){\color{darkblue}$\partial U_{R''}$}
        \put(-264,151){\color{darkblue}$\partial U_{R}$}
      \put(-261,274){\color{gray}$\partial \tilde U$}
       \put(-275,345){\LARGE $U$}
 \end{picture}
\caption{Showing $R - R'' \rightarrow 0$ as $h \to 0_+$}
\end{figure}

We can then apply Lemma \ref{domainswallowlemma} to conclude using $\frac{1}{r_{2h}}\psi_{2h}({\tilde V_{2h}})=\D$ that on letting $h \to 0_+$ we have 
\begin{align*}
\mathrm{d}\left (\partial \left (\frac{1}{r_{2h}}\psi_{2h}({\tilde U})\right ),\partial \left (\frac{1}{r_{2h}}\psi_{2h}({\tilde V_{2h}})\right )\right ) = \mathrm{d}\left (\partial \left (\frac{1}{r_{2h}}\psi_{2h}({\tilde U})\right ),\partial \D\right )\rightarrow 0 
\end{align*}
where again the convergence is uniform on $[r_0, R_0]$. Thus, scaling by $r_h$, we have, again uniformly on $[r_0, R_0]$,
\begin{equation}
\label{boundaryclose}
\mathrm{d}(\partial (\psi_{2h}({\tilde U})),\partial (\psi_{2h}({\tilde V_{2h}})))\rightarrow 0_+\text{ as }h \rightarrow 0_+.
\end{equation}
We observe that, since $r_h$ depends on $R$, this is the first time when the convergence being uniform on $[r_0, R_0]$ is not entirely trivial.

Further, using the Schwarz lemma for the hyperbolic metric (\cite{CG} Theorem I.4.1 or I.4.2), we have that 
\begin{equation}
\label{uniformcontainment}
\psi_{2h}({\tilde U})\subset \psi_{2h}({\tilde V_{2h}})\subset \psi_{2h}(U_R)\subset \psi_{2h}( U_{\tfrac{\pi}{2}})\subset\psi_{2h}(\Delta_{V_{2h}}(0,\tfrac{\pi}{2})) = \Delta_{\D}(0,\tfrac{\pi}{2}),
\end{equation}
(where we use the Schwarz Lemma for the hyperbolic metric for the last inclusion) which shows that both $\psi_{2h}({\tilde U})$ and $\psi_{2h}({\tilde V_{2h}})$ lie within distance $\tfrac{\pi}{2}$ of $0$ within $\D$.  Since $\psi_{2h}^{-1}$ converges to $\psi^{-1}$ uniformly on compact subsets of $\D$ by Theorem \ref{theorem1.2}, using \eqref{boundaryclose}, \eqref{uniformcontainment}, we have that $\mathrm{d}(\partial {\tilde U},\partial {\tilde V_{2h}})\rightarrow 0_+$  uniformly on $[r_0, R_0]$ as $h \rightarrow 0_+$.  Using Lemma \ref{lemma4.3} and the fact that $\tilde U \subset \tilde V_{2h} \subset U_R \subset \Delta_U(0, \tfrac{\pi}{2})$, we see that we can say the same for distances with respect to the hyperbolic metric for $U$ and that 
\begin{align*}
\rho_{U}(\partial {{\tilde U}},\partial {\tilde V_{2h}})\rightarrow 0 \quad \mbox{as} \:\: h \rightarrow 0_+ 
\end{align*}
uniformly on $[r_0, R_0]$.

Fix $\epsilon_0>0$. Using Lemma \ref{equivalenthypradformulation}, pick $z \in \partial {\tilde U}$ such that $\rho_U(0,z)=R''$.  From above, for all $h$ sufficiently small, we can pick $w_h \in \partial {\tilde V_{2h}}$ such that
 \begin{align*}
 \rho_U(z,w_h)<\frac{\epsilon_0}{2}
\end{align*}
(for any $R \in [r_0, R_0]$).  Now let $\gamma$ be the unique geodesic in $U$ passing through $0,w_h$.  As $\gamma$ must eventually leave $U_R$, let $w$ be the first point on $\gamma \cap \partial U_R$ after we pass along $\gamma$ from $0$ to $w_h$.  Then $0$, $w_h$, and $w$ are on the same geodesic and $w_h$ is on the hyperbolic segment $\gamma_U[0,w]$ in $U$ from $0$ to $w$. We now have $\rho_U(0,w)=R$ and $\rho_U(0,w_h)\geq R' = R_{(U,0)}^{int}{\tilde V_{2h}}$ using Lemma \ref{equivalenthypradformulation}.  Then, since $w_h \in \gamma_U[0,w]$, using our Claim \ref{PhaseIIClaim1} above, we have, uniformly on $[r_0, R_0]$,
\begin{align*}
\rho_U(w,w_h)&=\rho_U(0,w)-\rho_U(0,w_h) \\
&\leq R-R' \\
&<\frac{\epsilon_0}{2}
\end{align*}
for $h$ sufficiently small.  Further, we have 
\begin{align*}
R-R'' &= \rho_U(0,w)-\rho_U(0,z) \\
&\leq \rho_{U}(z,w)\\
&\leq \rho_U(z,w_h) + \rho_U(w_h,w) \\
&<\frac{\epsilon_0}{2}+\frac{\epsilon_0}{2} \\
&=\epsilon_0
\end{align*}
for $h$ sufficiently small, and thus $R-R''\rightarrow 0$  as $h\rightarrow 0_+$ while this convergence is uniform on $[r_0, R_0]$ as desired. \hfill \end{claimproof}

By the Fitting Lemma (Lemma \ref{FittingLemma}), there exist $\tilde \epsilon_1 > 0 $ and a function $h$ defined on $(0, \tilde \epsilon_1]$, both of which depend on $h_0$, $r_0$ which we fixed before the statement for which we have (by \emph{(2)} of the statement of this result) that \label{defofh3}
\begin{align}
\label{LossofDomain}
h(\epsilon_1) \rightarrow 0_+ \quad \mbox{as} \: \epsilon_1\rightarrow 0_+.
\end{align}
From this, using using this function and Claim \ref{PhaseIIClaim2} we have that
 \begin{align}
\label{R''0}
R-R''\rightarrow 0 \text{ as }\epsilon_1\rightarrow 0_+ 
\end{align}
while this convergence is uniform on $[r_0, R_0]$.

To conclude this chapter we make our final determination of the upper bound $\tilde \epsilon_1$ and define the function $\delta$ on $(0, \tilde \epsilon_1]$. Using the value of $\tilde \epsilon_1$ above which comes from Lemma \ref{FittingLemma}, for $\epsilon_1 \in (0, \tilde \epsilon_1]$, set 
\begin{align}
\label{deltadef}
\delta(\epsilon_1) := \sup_{[r_0, R_0]}{(R-R'')}+5\epsilon_1
\end{align}
(the justification for this definition of this function will be made clear later). Note that, in view of the above dependencies of the function $h$ on $h_0$ and $r_0$, the function $\delta$ then depends on $h_0$ and the bounds $r_0$, $R_0$ for $R$ (but not on $R$ itself), all of which we regard as fixed in advance. 
It follows from \eqref{R''0} that $\delta(\epsilon_1) \rightarrow 0$ as $\epsilon_1\rightarrow 0_+$ (we remark that this is the point where we require that the convergence above be uniform). We can then make \label{tildeepsilon1}${\tilde \epsilon_1}$ smaller if needed such that
\begin{align}
\label{deltabound}
\sup_{(0, \tilde \epsilon_1]} \delta(\epsilon_1)<\frac{r_0}{4}
\end{align}
which ensures that $U_{R-\delta(\epsilon_1)}\neq \emptyset$. Note from above that $\tilde \epsilon_1$ will therefore also depend on $h_0$ and the bounds $r_0$, $R_0$ for $R$ (in other words, we pick up an extra dependency on $R_0$ from the definition of the value $\delta(\epsilon_1)$ in \eqref{deltadef}), so that we now have the correct dependencies for both the quantity $\tilde \epsilon_1$ and the function $\delta: (0, \tilde \epsilon_1] \mapsto (0, \tfrac{r_0}{4})$ as in the statement. Note in addition that this change in $\tilde \epsilon_1$ may require us to redefine the function $h$ above by restricting its domain of definition. 
Note that restricting $\tilde \epsilon_1$ in this way will not violate \eqref{ExistenceCriterion} in the proof of the Fitting Lemma (Lemma \ref{FittingLemma}) so that we can still define $h(\epsilon_1)$ according to \eqref{defofh1} in the proof and in particular \eqref{defofh2} still holds. In addition, because the function $\delta$ depends on $h_0$, $r_0$, and $R_0$, the redefined function $h$ now depends on $h_0$, $r_0$ and $R_0$ as in the statement. Lastly, note that in particular \eqref{deltabound} implies that 
\begin{equation}
\label{innerdiscnottoosmall}
U_{R - \delta} \supset U_{3r_0/4}.
\end{equation}

{\bf Controlling Error: `Up'} 

Now fix $\epsilon_1 \in (0, \tilde \epsilon_1]$, $h = h(\epsilon_1)$ using the function $h$ introduced before \eqref{LossofDomain} above, and also fix $R \in [r_0, R_0]$ as in the statement. \label{up} Recall from the discussion before Lemma \ref{hypradv2hbddbelow} that we had ${\tilde R}=R_{(V_{2h},0)}^{int}U_R$, ${\tilde V_{2h}}=\Delta_{V_{2h}}(0,{\tilde R})$ and $\phi_{2h}:{\tilde V_{2h}}\rightarrow V_{2h}$ which was the unique conformal map from ${\tilde V_{2h}}$ to $V_{2h}$ normalized so that $\phi_{2h}(0)=0$ and $\phi_{2h}'(0)>0$. Recall also that $\psi_{2h}$ is the unique normalized Riemann map which sends $V_{2h}$ to $\D$. Since $\tilde V_{2h}$ is a hyperbolic disc about $0$ in $V_{2h}$ it follows that, in the conformal coordinates of $V_{2h}$, $\phi_{2h}$ is then a dilation of ${\tilde V_{2h}}$.  To estimate the error in approximating $\phi_{2h}$, we wish to break this dilation into many smaller dilations, and apply the Polynomial Implementation Lemma (Lemma \ref{PIL}) so as to approximate each of these small dilations with a polynomial composition. The key idea here is that conformal dilations by small amounts can have larger domains of definition and, by dilating by a sufficiently small amount, we can ensure this domain of definition includes the filled Julia set and indeed all of the Green's domain $V_h$ which ultimately allows us to apply the Polynomial Implementation Lemma to approximate it to an arbitrarily high degree of accuracy. 

As before, let $r_{2h}\in (0,1)$ be such that $\psi_{2h}({\tilde V_{2h}})=\mathrm{D}(0,r_{2h})$ and recall that $r_{2h}$ depends on $\epsilon_1$, $R$, $h_0$, $r_0$, and $R_0$. Pick \label{sup}$s \in (0,1)$ so that  ${\psi_{2h}(\overline V_h)}\subset \mathrm{D}(0,s)$. Note that $s$ depends immediately on $h$, $\psi_{2h}$, and $r_{2h}$ but does not depend on $\kappa$ by conformal invariance, so that $s$ also depends ultimately on $\epsilon_1$, $R$, $h_0$, $r_0$, and $R_0$. Note also that we must have $s > r_{2h}$ since $V_h \supset U \supset \overline U_R \supset \tilde V_{2h}$.  Now fix \label{Nup}$N$ such that 
\begin{align}
\label{Nineq}
s \sqrt[N]{\frac{1}{r_{2h}}} < \sqrt s.
\end{align}  
and note that this choice of $N$ will depend immediately on $s$ and $r_{2h}$ and thus ultimately on $\epsilon_1$, $R$, $h_0$, $r_0$, $R_0$, from above. However, given the definition of $\tilde V_{2h}$ as the hyperbolic `incircle' of $U_R$ in $V_{2h}$, $r_{2h}$ decreases as $R$ decreases which allows us to eliminate the dependence on $R$ and $R_0$ above so that we can ensure that $N$ depends only on $\epsilon_1$, $h_0$, and $r_0$.

This choice will ensure that our conformal dilations in the composition do not distort $\partial V_h$ so much so that we no longer have a conformal annulus for interpolation when we apply the Polynomial Implementation Lemma.

Next define on $\psi_{2h}^{-1}(\mathrm{D}(0,s))$ the map 
\begin{align}
\label{gdef}
g(z)=\psi_{2h}^{-1}\left(\sqrt[N]{\frac{1}{r_{2h}}}\psi_{2h}(z)\right)
\end{align} 
and note in particular that $g$ is defined and in addition analytic and injective on a neighbourhood of $\overline {V_h}$ as $\psi_{2h}(\overline V_h) \subset \mathrm{D}(0,s)$ by our choice of $s$. Further, since $\psi_{2h}$ fixes $0$, we have $g(0) = 0$ and, given our choice of $N$ in \eqref{Nineq} above, we have 
\begin{align}
\label{PILg}
\overline{g(V_h)}\subset V_{2h}.
\end{align}

By conformal invariance or simply because $g$ corresponds to a dilation by $r_{2h}^{-1/N}$ in the conformal coordinates of $V_{2h}$, recalling that we set $\tilde V_h := \phi_{2h}^{-1}(V_h)$ (immediately before Lemma \ref{rcheckcts}), we must then have that $\psi_{2h}(\tilde U) \subset \psi_{2h}(\tilde V_h) \subset \Delta_\D(0, r_{2h}s)$. Again since $g$ corresponds to a dilation by $r_{2h}^{-1/N}$ in the conformal coordinates of $V_{2h}$, the compositions $g^{\circ j}$, $0 \le j \le N$ (where $g^{\circ \, 0}$ is the identity) are all then defined on $\tilde U$ and in particular we have $g^{\circ N} = \phi_{2h}$ on $\tilde U$. We observe that the functions $g^{\circ j}$ then form (part of) a L\"owner chain on $\tilde U$ in a sense similar to that given in \cite{CDG} (although these authors were working on the unit disc). 

Since $R'' < R \le R_0 \le  \tfrac{\pi}{2}$ is bounded above, the external hyperbolic radius about $0$ of $U_{R''-\epsilon_1}$ inside  $U_{R''}$ (with respect to the hyperbolic metric of this slightly larger domain) can be uniformly bounded above in terms of $\epsilon_1$ and the upper bound $R_0 \le \tfrac{\pi}{2}$ for $R$. By the Schwarz lemma for the hyperbolic metric (e.g. \cite{CG} Theorem I.4.1 or I.4.2), the same is true for the external hyperbolic radius about $0$ of $U_{R''-\epsilon_1}$ inside the larger (than $U_{R''}$) domain $\tilde U$. By conformal invariance under $\phi_{2h}$, the same is also true for the external hyperbolic radius about $0$ of $\phi_{2h}(U_{R''-\epsilon_1})$ inside  $U$. We can then find an upper bound \label{R2}$R_2$ for this external hyperbolic radius 
which depends only on $\epsilon_1$ and the upper bound $R_0$ on $R$. As a result, we have 
\begin{align}
\label{r2bound}
R^{ext}_{(U,0)}\phi_{2h}(U_{R'' - 3\epsilon_1}) < R^{ext}_{(U,0)}\phi_{2h}(U_{R'' - 2\epsilon_1}) < R^{ext}_{(U,0)}\phi_{2h}(U_{R'' - \epsilon_1}) \le R_2.
\end{align}

We note that this upper bound is also in particular independent of $N$ and $h$ (recall that in fact we have $h$ is a function of $\epsilon_1$ in view of \eqref{defofh1} and \eqref{deltabound}). We also note that, at this point, we do not actually require that $R_2$ be independent of $h$. However, we will need this later when we turn to giving upper bound $\tilde \epsilon_2$ for $\epsilon_2$ which has the correct dependencies as listed in the statement. Finally, we note that this upper bound is for the set $U_{R'' - \epsilon_1}$, while all we will need in this section of the proof is a bound on the slightly smaller sets $U_{R'' - 2\epsilon_1}$, $U_{R'' - 3\epsilon_1}$. However, we will need the bound on the larger set when it comes to the `during' part of the proof later on.

Since from above the compositions $g^{\circ j}$, $0 \le j \le N$ are all defined on $\tilde U \supset U_{R'' - 3\epsilon_1}$, we can now set $B := g^{\circ N}(U_{R''-3\epsilon_1})$ and note that, since $g^{\circ N} = \phi_{2h}$ on $U_{R'' - 3\epsilon_1}$, it follows from \eqref{r2bound} above that $g^{\circ N}$ maps $U_{R''-3\epsilon_1}$ inside $U_{R_2}$ which is a relatively compact subset of $U$. 
On the other hand, recalling the normalized Riemann map $\psi$ from $U$ to $\D$ (which was introduced before the start of Lemma \ref{hypradv2hbddbelow}), 
since $R'' \le R_0 \le \tfrac{\pi}{2}$, by Lemma \ref{starshaped}, the set $\psi_{2h}(U_{R''-3\epsilon_1}) = \psi_{2h} \circ \psi^{-1}( \Delta_\D(0, R''-3\epsilon_1))$,  is star-shaped with respect to $0$. Since $g$ corresponds to a dilation by $r_{2h}^{-1/N} >1$ in the conformal coordinates of $V_{2h}$, it therefore follows that the sets $g^{\circ j}(U_{R''-3\epsilon_1})$, $0 \le j \le N$ (which from above are well-defined) are increasing in $j$ and therefore all contained in $B$. Thus any estimate which holds on $B$ will also automatically hold on these sets also.

Now set $A := U_{R_2+1} \supset B$ (we remark that the `extra' $1$ here is due to the fact that $\phi_{2h} = g^{\circ N}$ is a composition of $g$ with itself many times and each of these compositions with $g$ will be approximated so that we need to be able to allow for the total error which arises). The functions $\psi_{2h}^{-1}\left(\sqrt[N]{\frac{1}{r_{2h}}}\psi_{2h}(z)\right)$ clearly converge to the identity locally uniformly on $U$ as $N \to \infty$ (and in particular, for all sufficiently large $N$, are defined on any relatively compact subset of $U$ and map it into another relatively compact subset). 

Since $A$ depends on $R_2$ which from above does not depend on $N$, it follows that, if we fix a constant $K_1 = \tfrac{3}{2}$, we may   
therefore make \label{Nupredef}$N$ larger if needed (without invalidating \eqref{Nineq}) so that, if $\hat A$ is a $1$-hyperbolic neighbourhood of $A$ in the hyperbolic metric of $U$ (which implies that $\hat A = U_{R_2 +2}$), then $g$ still maps $\hat A$ into a relatively compact subset of $U$ and we have 
\begin{align}
\label{ghypder}
\| g^{\natural} \|_{\hat{A}}\leq K_1
\end{align}
where, as usual, we are taking our hyperbolic derivatives with respect to the hyperbolic metric of $U$. Our new choice of $N$ will depend directly on $R_2$, and $g$ in addition to the old dependencies on $\epsilon_1$, $h_0$, and $r_0$ from the discussion after \eqref{Nineq} and so, using \eqref{gdef} and \eqref{r2bound}, ultimately $N$ depends on $\kappa$, $\epsilon_1$, $h_0$, $r_0$, $R_0$, and $R$ (this last being via $r_{2h}$ which we are not allowed to alter at this stage). However, since we are estimating a hyperbolic derivative here, we can eliminate the dependence on $\kappa$ so that $N$ ultimately depends on $\epsilon_1$, $h_0$, $r_0$, $R_0$, and $R$. Note also that the function $g$ defined in \eqref{gdef} will then depend on the six quantities $\kappa$, $\epsilon_1$, $h_0$, $r_0$, $R_0$, and $R$ (and in particular we cannot eliminate the dependence on $\kappa$ since the domain of $g$ depends on $\kappa$ via $\psi_{2h}$).

Note also that by Lemma \ref{stupidfuckinglemma} $\hat A$ is hyperbolically convex which will be useful (though not essential) when we come to apply the hyperbolic M-L estimates (Lemma \ref{hyperbolicML}) later on. Also important to note is that $N$ is fixed from now on which means that we can choose our subsequent approximations using the Polynomial Implementation Lemma with this $N$ in mind.

\label{epsilon2initialbound}
Set $\tilde \epsilon_2 : =1$ and let $0< \epsilon_2 \le \tilde \epsilon_2$ be arbitrary and fixed (note that this upper bound $\tilde \epsilon_2$ is universal, but we will be making further restrictions later in the proof to deduce the upper bound $\tilde \epsilon_2$ with the same dependencies as in the statement). 
Define $\gamma :=\partial V_h$ and $\Gamma := \partial V_{2h}$ (with positive orientations as Jordan curves) and note that, since $g$ is injective and analytic on a neighbourhood of $\overline V_h$ while $\overline{g(V_h)}\subset V_{2h}$ from \eqref{PILg} above, we must have that $g(\gamma)$ lies inside $\Gamma$ (so that $(g, Id)$ is an admissible pair on $(\gamma, \Gamma)$ in the sense given in Definition \ref{admissiblepair}).

Now set $\epsilon$ in the statement of the Polynomial Implementation Lemma (Lemma \ref{PIL}) to be $\frac{\epsilon_2}{3(2K_1)^{N-1}K_2 K_3}$, where $K_2$ and $K_3$ are bounds on hyperbolic derivatives which will be chosen later.  For now we just assume that $K_i>1$ for $i=2,3$ (these are just constants, and we can always choose a larger constant). Note that we have $\frac{\epsilon_2}{3(2K_1)^{N-1}}<1$ which implies that $\epsilon<1$. Further,  note that $\epsilon<\epsilon_2$. Now, since $g(0)=0$, $g(\gamma)$ lies inside $\Gamma$, and we have the estimate \eqref{ghypder} on the hyperbolic derivative of $g$, we can apply the Polynomial Implementation Lemma (Lemma \ref{PIL}), with $\Omega = V_h$, $\Omega' = V_{2h}$, $\gamma=\Gamma_h$, $\Gamma=\Gamma_{2h}$, $f = g$, $A = U_{R_2 + 1}$, $\delta=1$, $M = K_1$, and $\epsilon=\frac{\epsilon_2}{3(2K_1)^{N-1}K_2 K_3}$ as above to $g$ to get $n_{k_0} >0$, and a (17+$\kappa$)-bounded finite sequence of quadratic polynomials $\{Q_m \}_{m=1}^{n_{k_0}}$ such that the composition of these polynomials, $Q_{n_{k_0}}$, is univalent on $A$ and satisfies 
\begin{align}
\label{PILUp1}
\rho_U(Q_{n_{k_0}}(z),g(z))&< \frac{\epsilon_2}{3(2K_1)^{N-1}K_2 K_3} = \epsilon , \hspace{1cm} z\in A, \\
\label{PILUp2}
\|Q^{\natural}_{n_{k_0}}\|_{A}&\leq K_1\left (1+\frac{\epsilon_2}{3(2K_1)^{N-1}K_2 K_3} \right ),\\
\label{PILUpfix0}
Q_{n_{k_0}}(0) &=0.
\end{align}
Note that by Lemma \ref{PIL}, since $M = K_1 = \tfrac{3}{2}$ and $\delta = 1$, $n_{k_0}$ and $Q_{n_{k_0}}$ depend directly on $\kappa$, $K_1$, $\epsilon$, $R_2$, $g$, and $h$ so that one can check that $n_{k_0}$ and $Q_{n_{k_0}}$ ultimately depend on $\kappa$, $\epsilon_1$, $\epsilon_2$, $K_2$, $K_3$, $R$, $h_0$, $r_0$, and $R_0$. 

For $1 \le j \le N$ define $Q_{jn_{k_0}}:=Q^{\circ k}_{n_{k_0}}$.  We prove the following claim, which will allow us to control the error in the `Up' portion of Phase II: 

\begin{claim}
\label{Q1Approx}
For each $1\leq j \leq N$ we have:
\begin{align*}
1.& \; \rho_U(Q_{j n _{k_0}}(z),g^{\circ j}(z) )<\frac{\epsilon_2}{3(2K_1)^{N-j}K_2K_3} < 1, \qquad z \in U_{R''-3\epsilon_1}, \\
2.& \; Q_{j n _{k_0}}(z)\in A, \qquad z \in U_{R''-3\epsilon_1},\\
3. &\; Q_{j n _{k_0}} \quad \mbox{is univalent on}\:\:  U_{R''-3\epsilon_1}.
\end{align*}
\end{claim}

\begin{claimproof}
For the base case $j=1$, recall that, from the definition of $R_2$ given in \eqref{r2bound}, we have that the external hyperbolic radius of $U_{R''-3\epsilon_1} \subset U_{R''-2\epsilon_1}$ inside $\tilde U$ is bounded above by $R_2$. Since $\tilde U \subset U$, by the Schwarz Lemma for the hyperbolic metric (e.g. \cite{CG} Theorem I.4.1 or I.4.2), we have that $U_{R''-3\epsilon_1} \subset U_{R_2} \subset U_{R_2 + 1} = A$. 1. then follows from \eqref{PILUp1}. 

For 2. recall that the sets $g^{\circ j}(U_{R''-3\epsilon_1})$, $0 \le j \le N$ are increasing in $j$ and, in view of \eqref{r2bound}, therefore all contained in $B = g^{\circ N}(U_{R''-3\epsilon_1}) = \phi_{2h}(U_{R''-3\epsilon_1}) \subset U_{R_2}$. Thus $g(z) \in U_{R_2}$ and the result follows from \eqref{PILUp1} on recalling that $\epsilon < 1$ and that $A = U_{R_2 + 1}$ contains a $1$-neighbourhood of $B$ (in the hyperbolic metric of $U$).

Finally, 3. simply follows from the above fact that $Q_{n _{k_0}}$ is univalent on $A$ which we already saw contains $U_{R''-3\epsilon_1}$.

Now assume the claim is true for some $1\leq j <  N$.  For $z \in U_{R''-3\epsilon_1}$ we have 
\begin{multline*}
\rho_U(Q_{(j+1)n_{k_0}}(z),g^{\circ j+1}(z)) \leq \rho_U(Q_{(j+1)n_{k_0}}(z),g\circ Q_{j n_{k_0}}(z))\\ +\rho_U(g\circ Q_{j n_{k_0}}(z),g^{\circ j+1}(z)).
\end{multline*}
Now $Q_{j n_{k_0}}(z)\in A$ by hypothesis so the first term in the inequality above is less than $\epsilon$ by \eqref{PILUp1}.  
In addition to $Q_{j n_{k_0}}(z)\in A$, we also have $g^{\circ j}(z) \in B \subset A$ (we remark that this is a place where we need to make use of the fact that the sets $g^{\circ j}(U_{R''-3\epsilon_1}))$ are increasing in $j$ and thus all contained in $B$) while $\rho_U(Q_{j n_{k_0}}(z),g^{\circ j}(z))<\tfrac{\epsilon_2}{3(2K_1)^{N-j}K_2K_3} < 1$ by hypothesis whence we also have $Q_{j n_{k_0}}(z)$ lies within distance $1$ of $B$ and is thus in $A$. Using \eqref{ghypder} and the hyperbolic convexity of $A$ which follows from Lemma \ref{stupidfuckinglemma}, we see on applying the hyperbolic M-L estimates (Lemma \ref{hyperbolicML}) to $g$ that the 
second term is less than $K_1\tfrac{\epsilon_2}{3(2K_1)^{N-j}K_2 K_3}$.   Thus we have 
\begin{align*}
\rho_U(Q_{(j+1)n_{k_0}}(z),g^{\circ j+1}(z))& < \epsilon+ K_1 \cdot {\frac{\epsilon_2}{3(2K_1)^{N-j}K_2 K_3}} \\
&= \frac{1}{(2K_1)^j}\cdot\frac{\epsilon_2}{3(2K_1)^{N-(j+1)}K_2 K_3}\\
&\hspace{3.4cm}+\frac{1}{2}\cdot\frac{\epsilon_2}{3(2K_1)^{N-(j+1)}K_2 K_3} \\
& < \frac{\epsilon_2}{3(2K_1)^{N-(j+1)}K_2 K_3}\\
& < 1
\end{align*}
which proves 1. in the claim using the fact that $K_1 > 1$ for the second last inequality and $\epsilon_2  \le  1$, $K_1, K_2, K_3 > 1$ for the last inequality above.

Now $Q_{(j+1)n_{k_0}}(U_{R''-3\epsilon_1})$ lies in a 1-neighborhood of $g^{\circ j+1}(U_{R''-3\epsilon_1})$ by 1. above. But $g^{\circ j+1}(U_{R''-3\epsilon_1})\in B$ (where again we note that the sets $g^{\circ j}(U_{R''-3\epsilon_1}))$ are increasing in $j$ and thus all contained in $B$)
while a $1$-neighbourhood of $B$ lies inside $A$ by the definition of $A$ and 
so $Q_{(j+1)n_{k_0}}(z)\in A$ if $z \in U_{R''-3\epsilon_1}$ (note that $j+1\leq N$), which finishes the proof of 2.  To show 3. and see that $Q_{(j+1)n_{k_0}}(z)$ is univalent, we obviously have $Q_{(j+1)n_{k_0}}(z)=Q_{n_{k_0}} \circ Q_{j n_{k_0}}(z)$.  Since by hypothesis we have both that $Q_{j n_{k_0}}$ is univalent on $U_{R''-3\epsilon_1}$ and $Q_{j n_{k_0}}(U_{R''-3\epsilon_1}) \subset A$, while $Q_{n_{k_0}}$ is univalent on $A$ by our application of the Polynomial Implementation Lemma (Lemma \ref{PIL}), we have that $Q_{(j+1)n_{k_0}}$ is univalent on $U_{R''-3\epsilon_1}$.  This completes the proof of the claim. \end{claimproof}

For convenience, set \label{bfQ1}${\bf Q_1} :=Q_{Nn_{k_0}}$ and recall that on $U_{R'' - 3\epsilon_1} \subset \tilde U$ we had $g^{\circ N} = \phi_{2h}$. From above, ${\bf Q_1}$ then depends $\kappa$, $\epsilon_1$, $\epsilon_2$, $K_2$, $K_3$, $R$, $h_0$, $r_0$, and $R_0$ (recall that $N$ depends on $\epsilon_1$, $R$, $h_0$, $r_0$, and $R_0$ while the mapping $\phi_{2h} = g^{\circ N}$ depends on $\kappa$, $\epsilon_1$, $R$, $h_0$, $r_0$, and $R_0$). By 3. of Claim \ref{Q1Approx} above, ${\bf Q_1}$ is univalent on $U_{R''-3\epsilon_1}$ and, on this hyperbolic disc, from 1. of the same claim  and the fact that $g^{\circ N}=\phi_{2h}$ on $U_{R'' - 3\epsilon_1}$, we have (on this set)
\begin{align}
\label{PILQ1app}
\rho_U({\bf Q_1}(z), \phi_{2h}(z))&<\frac{\epsilon_2}{3K_2 K_3},
\end{align}
while, from 2. of this claim and \eqref{PILUpfix0}
\begin{align}
\label{PILQ1loc}
{\bf Q_1}(z)&\in A,\\
\label{PILQ1fix0}
{\bf Q_1}(0) &=0.
\end{align}

The mapping $\phi_{2h}$ obviously maps $U_{R'' - 3\epsilon_1}$ to $\phi_{2h}(U_{R''-3\epsilon_1})$ and, provided the next polynomial in our construction has the desired properties on this set, we will be able to compose in a meaningful way so that the composition also has the desired properties. However, in practice, we are approximating $\phi_{2h}$ with the composition ${\bf Q_1}$ which involves an error and our next step is to show that we can map into the correct set $\phi_{2h}(U_{R''-3\epsilon_1})$ using ${\bf Q_1}$ provided we are wiling to `give up' an extra $\epsilon_1$. First, however, we have the following important estimates which we will need later, especially when it comes to defining the upper bound $\tilde \epsilon_2$ for $\epsilon_2$ to obtain the same dependencies as given in the statement. 

\begin{claim}
\label{etabounds}
There exist \label{etas}$\eta_1, \eta_2  > 0$ depending on $\epsilon_1$, $h_0$, $r_0$, and $R_0$ such that 
\begin{align*}
\eta_1 \le \|( \phi_{2h}^{-1})^{\natural}\|_{U_{R_2 + 2}}\leq \eta_2.
\end{align*}
\end{claim}

\begin{claimproof}
Recall the upper bound $R_2$ on $R^{ext}_{(U,0)}\phi_{2h}(R_{R''-\epsilon_1})$ given in \eqref{r2bound} and the normalized Riemann map $\psi$ from $U$ to  $\D$ which was defined just before Lemma \ref{hypradv2hbddbelow}. $\phi_{2h}^{-1}$ maps $U$ to $\tilde U \subset U$ so that the conjugated mapping $\psi \circ \phi_{2h}^{-1} \circ \psi^{-1}$ is defined on all of $\D$. Using \eqref{deltabound} and \eqref{r2bound} one checks easily that 
\begin{align}
\label{LowerKoebeBound}
\phi_{2h}^{-1}(U) \supset \phi_{2h}^{-1}(U_{R_2+2}) \supset \phi_{2h}^{-1}(\phi_{2h}(U_{R'' - 2\epsilon_1})) = U_{R'' - 2\epsilon_1} \supset U_{3r_0/4}.
\end{align}
Hence $\psi \circ \phi_{2h}^{-1} \circ \psi^{-1}$  maps $\D$ to a domain which is contained in $\D$ and which contains $\Delta_{\D}(0, 3r_0/4)$ and so by 
the Koebe one-quarter theorem  (Theorem \ref{Koebe}) we then obtain (strictly positive) upper and lower bounds for the derivative $|(\psi \circ \phi_{2h}^{-1} \circ \psi^{-1})'(0)|$ which depend only on $h_0$ (because we assumed $h \le h_0$) and $r_0$. Note that in particular these bounds do not depend on the values of $h$ or $R$. 

Since $\psi \circ \phi_{2h}^{-1} \circ \psi^{-1}$ is defined on the whole of the unit disc, on applying the 
distortion theorems (Theorem \ref{distortion}), we obtain strictly positive upper and lower bounds for $|(\psi \circ \phi_{2h}^{-1} \circ \psi^{-1})'|$ on the set $\psi(U_{R_2+2}) = \Delta_{\D}(0, R_2+2)$. Since by \eqref{r2bound} $R_2$ depends on $\epsilon_1$ and $R_0$, these bounds depend on $\epsilon_1$, $h_0$, $r_0$, and $R_0$.

$\phi_{2h}^{-1}$ maps $U_{R_2 + 2}$ inside $\tilde U \subset U_{\pi/2}$ and, as $U_{R_2+2}$ and $U_{\pi/2}$  are both relatively compact subsets of $U$, $|{\psi}'|$ is uniformly bounded above and below away from $0$ on both of these sets. We therefore obtain strictly positive upper and lower bounds for the absolute value of the Euclidean and thus the hyperbolic derivative of $\phi_{2h}^{-1}$ on $U_{R_2 + 2}$. These bounds will depend on $\kappa$, $\epsilon_1$, $h_0$, $r_0$, and $R_0$ (the dependence on the scaling factor $\kappa$ arising via $\psi$). However, since we are estimating hyperbolic derivatives, we can actually eliminate the dependence on $\kappa$ and the claim then follows.
\end{claimproof}

We now define our upper bound \label{epsilon2}$\tilde \epsilon_2$ on $\epsilon_2$ by setting
\begin{align}
\label{e2ub}
\tilde \epsilon_2 = \min \{1, \tfrac{\epsilon_1}{\eta_2}\}
\end{align}

Given the dependencies of $\eta_2$ above (in Claim \ref{etabounds}), this upper bound then depends on $\epsilon_1$, $h_0$, $r_0$, and $R_0$  which is the same as given in the statement.

Now make $\epsilon_2$ smaller if necessary to ensure that $\epsilon_2 \le \tilde \epsilon_2$, (note that this may require us to obtain a new composition ${\bf Q_1}$ as above, but since $\epsilon_2$ and $\tilde \epsilon_2$ in no way depend on ${\bf Q_1}$, there is no danger of circular reasoning). 

\begin{claim}
\label{GivingUp1Epsilon1}
Given $\tilde \epsilon_2$, $0 < \epsilon_2 \le \tilde \epsilon_2$ and ${\bf Q_1}$ as above, we have
\begin{align}
{\bf Q_1}(U_{R''-4\epsilon_1})\subset \phi_{2h}(U_{R''-3\epsilon_1}).
\end{align}
\end{claim}

\begin{claimproof}
Let $z \in U_{R''-4\epsilon_1}$, $w \in \partial U_{R''-3\epsilon_1}$ be arbitrary and note that $\rho_U(z,w) \ge \epsilon_1$ while both $\phi_{2h}(z)$ and $\phi_{2h}(w)$ lie inside $U_{R_2} = \Delta_U(0,R_2)$ in view of \eqref{r2bound}. As $\phi_{2h}$ is a homeomorphism, we also have that $\phi_{2h}(z)\in \mathrm{int} \,\phi_{2h}(U_{R''-3\epsilon_1})$ while $\phi_{2h}(w)\in \partial \phi_{2h}(U_{R''-3\epsilon_1})$. Lemma \ref{stupidfuckinglemma} ensures that $\Delta_U(0,R_2+2)$ is hyperbolically convex and so we may apply the hyperbolic M-L estimates (Lemma \ref{hyperbolicML}) using Claim \ref{etabounds} on  $\Delta_U(0,R_2+2)$ to $\phi_{2h}^{-1}$. Thus we have $\rho_U(\phi_{2h}(z),\phi_{2h}(w)) \ge  \frac{\epsilon_1}{\eta_2}$, which implies the hyperbolic distance from $\phi_{2h}(z)$ to $\partial(\phi_{2h}(U_{R''-3\epsilon_1}))$ is at least $\frac{\epsilon_1}{\eta_2}$.

Again let $z \in U_{R''-4\epsilon_1}$ be arbitrary.  We then have using $K_2, K_3 > 1$, (\ref{PILQ1app}), and \eqref{e2ub} that 
\begin{align*}
\rho_U({\bf Q_1}(z),\phi_{2h}(z))&<\frac{\epsilon_2}{3K_2 K_3} \\
&<\epsilon_2 \\
&\le\frac{\epsilon_1}{\eta_2}
\end{align*}
and since the hyperbolic distance from $\phi_{2h}(z)$ to $\partial(\phi_{2h}(U_{R''-3\epsilon_1}))$ is at least $\frac{\epsilon_1}{\eta_2}$ from above, it follows that ${\bf Q_1}(z)$ misses $\partial \phi_{2h}(U_{R''-3\epsilon_1})$. 
Additionally, as ${\bf Q_1}(U_{R'' - 4\epsilon_1})$ is connected in view of 3. of Claim \ref{Q1Approx} above while ${\bf Q_1}(0) = 0 \in \phi_{2h}(U_{R'' - 3\epsilon_1})$, it follows, since $z \in U_{R'' - 4\epsilon_1}$ was arbitrary, that 
${\bf Q_1}(U_{R'' - 4\epsilon_1}) \subset \phi_{2h}(U_{R''-3\epsilon_1})$ and the claim follows as desired. \end{claimproof}\\

{\bf Controlling Error: `During'} 

\label{during}
Recall that at the start of the last section we fixed a value of $\epsilon_1$ in $(0, \tilde \epsilon_1]$ which in turn fixed the value of $h = h(\epsilon_1)$ and that we also fixed a value of $R \in [r_0, R_0]$. We now fix a function $\calE$ as in the statement which is defined and univalent on $U_R$ with $\calE(0) = 0$ and $\rho_U(\calE(z),z)<\epsilon_1$ for $z \in U_R$. Note that, in addition to $R$, $\calE$ will also depend on $r_0$, $R_0$ (via $R$) and also on $\kappa$, the latter arising from the fact that $\calE$ is defined on the set $U_R$ which depends on $\kappa$. Recall the quantity ${\check R}(h):=R_{(V_{2h},0)}^{ext}V_h = R_{(\tilde V_{2h},0)}^{ext}\tilde V_h  $ introduced before the statement of Lemma \ref{rcheckcts} and the function $T: (0, \tilde \epsilon_1] \mapsto (0, \infty)$ which was introduced in the statement of the Target Lemma (Lemma \ref{TargetLemma}) and which served as a lower bound for $R^{int}_{({\tilde V_{2h}},0)}({\tilde V_{2h}} \setminus {\hat N})$ (where ${\hat N}$ was a $2\epsilon_1$-neighbourhood of $\partial \tilde V_{2h}$ with respect to the hyperbolic metric of $U$). 

Now ${\tilde U}\subset {\tilde V_h}$ (recall that $\tilde V_h = \phi_{2h}^{-1}(V_h)$ was introduced immediately before Lemma \ref{rcheckcts}) while in \eqref{defofh1}, \eqref{defofh2}, and \eqref{deltabound} we chose $h = h(\epsilon_1)$ (where for convenience we suppress the dependence of $h$ on $\epsilon_1$) as small as possible so that $\check R(h) = T(\epsilon_1)$ (cf. \eqref{defofh2}). By \emph{(1)} of the Fitting Lemma (Lemma \ref{FittingLemma}) we have $\tilde V_h \subset \tilde V_{2h} \setminus {\hat N}$ (this latter set clearly being closed). Hence the $2\epsilon_1$-neighbourhood $\hat N$ of $\partial \tilde V_{2h}$ avoids $\tilde V_h$ (and hence also the smaller set $\tilde U$). Thus, an  $\epsilon_1$-neighborhood (in the hyperbolic metric of $U$) of the closure $\overline {\tilde V_h}$ avoids an  $\epsilon_1$-neighborhood of $\partial \tilde V_{2h}$.  In particular, by the hypotheses on $\calE$ in the statement, $\calE(\partial \tilde V_h)$ is a simple closed curve which lies inside $\partial \tilde V_{2h}$. 

\label{Ehat}
Next on $\phi_{2h}( \tilde V_{2h} \setminus {\hat N})$ define ${\hat \calE}=\phi_{2h}\circ \calE \circ \phi_{2h}^{-1}$. Since by the hypotheses on $\calE$ we must have from above that $\calE( \tilde V_{2h} \setminus {\hat N}) \subset \tilde V_{2h}$, it follows that ${\hat \calE}$ is well-defined on $\phi_{2h}( \tilde V_{2h} \setminus {\hat N})$. $\hat \calE$ then depends immediately on the six quantities $\kappa$, $\epsilon_1$, $h$, $R$ (these last two among other things being via the domain $\tilde V_{2h}$), $\phi_{2h}$, and $\calE$ from which one can deduce (e.g. by using the tables in the appendices) that $\hat \calE$ ultimately depends on $\kappa$, $\epsilon_1$, $R$, $h_0$, $r_0$, $R_0$, and $\calE$. As before let $\gamma = \partial V_h$, $\Gamma = \partial V_{2h}$ as positively oriented Jordan curves. Then, from above and again by the hypotheses on $\calE$, since $\hat \calE$ is defined on $\phi_{2h}( \tilde V_{2h} \setminus {\hat N})$ which contains $\gamma = \partial V_h$, we have from above that ${\hat \calE}(\gamma)$ lies inside $\Gamma$ (and so $({\hat \calE}, Id)$ is an admissible pair on $(\gamma,\Gamma)$ in the sense given in Definition \ref{admissiblepair} in Chapter 3 on the Polynomial Implementation Lemma). Lastly, since $\phi_{2h}$ and $\calE$ both fix $0$, we must have that ${\hat \calE}(0) = 0$.

Since $\hat N$ is a $2 \epsilon_1$-neighbourhood of $\partial \tilde V_{2h}$ while $\tilde V_{2h} \supset \tilde U \supset U_{R''}$, it follows that $U_{R''-2\epsilon_1} \subset \tilde V_{2h} \setminus {\hat N}$ and so $\phi_{2h}(U_{R''-2\epsilon_1}) \subset \phi_{2h}( \tilde V_{2h} \setminus {\hat N})$. By the definition of ${\hat \calE}: = \phi_{2h} \circ {\calE} \circ \phi_{2h}^{-1}$ and the hypotheses on $\calE$ in the statement,
\begin{align}
\label{errormove0}
{\calE}(U_{R''-2\epsilon_1})\subset U_{R''-\epsilon_1}
\end{align}
and so from above
\begin{align}
\label{errormove1}
{\hat \calE}(\phi_{2h}(U_{R''-2\epsilon_1}))\subset \phi_{2h}(U_{R''-\epsilon_1}).
\end{align}
Hence, since ${\hat \calE}$ maps the relatively compact subset $\phi_{2h}(U_{R''-2\epsilon_1})$ to another relatively compact subset of $U$, we can fix the value of\label{K2} $1 < K_2 < \infty$ such that 
\begin{align}
\label{ehathd}
|{\hat \calE}^{\natural}(z)|\leq K_2, \hspace{1cm} z \in \phi_{2h}(U_{R''-2\epsilon_1})
\end{align}
where as usual we take our hyperbolic derivative with respect to the hyperbolic metric of $U$. $K_2$ depends immediately on $\epsilon_1$, $R''$, $\phi_{2h}$, and $\hat \calE$. Using the chain rule \eqref{chainrule} for the hyperbolic derivative, it follows from \eqref{r2bound}, Claim \ref{etabounds}, \eqref{errormove0}, and the hypotheses on $\calE$ in the statement that $K_2$ can be bounded uniformly in terms of $\kappa$, $\epsilon_1$, $h_0$, $r_0$, $R_0$, the function $\calE$ as well as the particular value of $R$ (since $\calE$ is defined on all of $U_R$ in the statement). However, the dependence on $\kappa$ (arising via $\phi_{2h}$) can be eliminated since we are estimating a hyperbolic derivative. We also observe that this is the one point where we employ the full force of \eqref{r2bound} and require an upper bound on the external hyperbolic radius of $ \phi_{2h}(U_{R''-\epsilon_1})$ and not just $\phi_{2h}(U_{R''-2\epsilon_1})$ or some smaller set.

Note in particular that this bound has nothing to do with the existence of the composition ${\mathbf Q_1}$ from the last section, and so there is no danger of circular reasoning in fixing the bound $K_2$ at this point.
Note also that this doesn't affect our earlier assertion that $\epsilon<1$ in the previous section on controlling the error for `up'. On the other hand, the same argument as used in the proof of Claim \ref{GivingUp1Epsilon1} shows that if we set $\delta_0 = \tfrac{\epsilon_1}{\eta_2}$ where $\eta_2$ is the upper bound on the hyperbolic derivative of $\phi_{2h}^{-1}$ from Claim \ref{etabounds}, then 
a $\delta_0$-hyperbolic neighbourhood in $U$ of $\phi_{2h}(U_{R''-3\epsilon_1})$ is contained in $\phi_{2h}(U_{R''-2\epsilon_1})$ while $\delta_0$ depends on $\epsilon_1$, $h_0$, and $r_0$, $R_0$. 

Since ${\hat \calE}(\gamma)$ lies inside $\Gamma$ while ${\hat \calE}(0)=0$, using \eqref{ehathd}, we can then apply the Polynomial Implementation Lemma (Lemma \ref{PIL}) with $\Omega = V_h$, $\Omega' = V_{2h}$, $\gamma=\Gamma_h$, $\Gamma=\Gamma_{2h}$, $f = {\hat \calE}$, $A = \phi_{2h}(U_{R''-3\epsilon_1})$, $\delta = \delta_0$, $M = K_2$ and 
$\epsilon = \tfrac{\epsilon_2}{3K_3}$ to construct a (17+$\kappa$)-bounded composition of quadratic polynomials, \label{bfQ2}${\bf Q_2}$, univalent on $\phi_{2h}(U_{R''-3\epsilon_1})$ such that 
\begin{align}
\label{PILQ2app}
\rho_U({\bf Q_2}(z),{\hat \calE}(z))&<\frac{\epsilon_2}{3K_3}, \quad z \in \phi_{2h}(U_{R''-3\epsilon_1}),\\
\label{PILQ2hd}
\| {\bf Q_2}^{\natural}\|_{\phi_{2h}(U_{R''-3\epsilon_1})}&\leq {K_2} \left (1+\frac{\epsilon_2}{3K_3} \right ),\\
\label{PILQ2fix0}
{\bf Q_2}(0) &=0,
\end{align}
where the bound $K_3 > 1$ is to be fixed in the next section. From the statement of Lemma \ref{PIL}, the composition ${\bf Q_2}$ depends directly on $\kappa$, $\epsilon_1$, $\epsilon_2$, $K_2$, $K_3$, $\eta_2$, $\phi_{2h}$, $R''$ (via the set $ \phi_{2h}(U_{R''-3\epsilon_1})$, and $\hat \calE$. From this one checks (e.g. using the tables) that  ${\bf Q_2}$ depends ultimately on 
on $\kappa$, $\epsilon_1$, $\epsilon_2$, $K_3$, $R$, $h_0$, $r_0$, $R_0$, and the function $\calE$.\\

{\bf Controlling Error: `Down'} 

Recall \eqref{errormove1} where we had that that ${\hat \calE}(\phi_{2h}(U_{R''-2\epsilon_1}))\subset \phi_{2h}(U_{R''-\epsilon_1})$. In exactly the same way, we have
\begin{align}
\label{errormove2}
{\hat \calE}(\phi_{2h}(U_{R''-3\epsilon_1}))\subset \phi_{2h}(U_{R''-2\epsilon_1}).
\end{align}
Recall also that from \eqref{r2bound} we had that $R^{ext}_{(U,0)}\phi_{2h}(U_{R''-2\epsilon_1}) \le R_2$. Also, by \eqref{PILQ2app} and \eqref{errormove2}, we have that ${\bf Q_2}(\phi_{2h}(U_{R''-3\epsilon_1}))$ is contained in an $\frac{\epsilon_2}3K_3$-neighborhood of $\phi_{2h}(U_{R''-2\epsilon_1})$ (using the hyperbolic metric of $U$).  Thus  
\begin{align*}
R^{ext}_{(U,0)}{\bf Q_2}(\phi_{2h}(U_{R''-3\epsilon_1}))&\leq  R_2+\frac{\epsilon_2}{3K_3} \\
& < R_2+\epsilon_2,
\end{align*}
(recall that we assumed $K_3 > 1$) and so
\begin{align}
\label{PILQ3A}
R^{ext}_{(U,0)}{\bf Q_2}(\phi_{2h}(U_{R''-3\epsilon_1}))\leq  R_2+1
\end{align} 
as $\epsilon_2<1$ using \eqref{e2ub}. Thus ${\bf Q_2}(\phi_{2h}(U_{R''-3\epsilon_1}))\subset U_{R_2+1}\subset U_{R_2 +2}\subset {\overline U}\subset V_{2h}$ while $\phi_{2h}^{-1}$ maps $U_{R_2+2} \subset U$ inside $\phi_{2h}^{-1}(U)= \tilde U$ which is compactly contained in $U$. Using Claim \ref{etabounds}, if we set \label{K3}$K_3 = \max\{\eta_2, \tfrac{3}{2}\}$
so that $K_3 > 1$, we have that
\begin{align}
\label{phi2h-1hd}
|(\phi_{2h}^{-1})^{\natural}(z)|\leq K_3, \hspace{1cm} z\in U_{R_2+2}.  
\end{align}

Note that, in view of \eqref{r2bound} and Claim \ref{etabounds}, $K_3$ depends on $\epsilon_1$, $h_0$, and $r_0$, $R_0$. Again, note that this bound has nothing to do with the existence of the compositions ${\mathbf Q_1}$, ${\mathbf Q_2}$ from the last sections, and so there is no danger of circular reasoning in fixing the bound $K_3$ at this point. 
Further, $\phi_{2h}^{-1}$ is analytic and injective on a neighbourhood of $\overline V_h$ and maps $\partial V_h$ inside $U \subset V_{2h}$ so that, if we set $\gamma = \partial V_h$, $\Gamma = \partial V_{2h}$ as positively oriented Jordan curves, we have that $\phi_{2h}^{-1}(\gamma)$ lies inside $\Gamma$. Thus
$(\phi_{2h}^{-1}, Id)$ is easily seen to be an admissible pair on $(\gamma,\Gamma)$ as in Definition \ref{admissiblepair} and we also have that $\phi_{2h}^{-1}(0)=0$. Using \eqref{phi2h-1hd}, we can then apply the Polynomial Implementation Lemma (Lemma \ref{PIL}) with $\Omega = V_h$, $\Omega' = V_{2h}$, $\gamma = \Gamma_h$, $\Gamma = \Gamma_{2h}$, $f = \phi_{2h}^{-1}$, $A = U_{R_2+1}$ (so that $\hat A = U_{R_2+2}$), $\delta = 1$, $M = K_3$, and $\epsilon = \frac{\epsilon_2}{3}$ to construct a (17+$\kappa$)-bounded quadratic polynomial composition \label{bfQ3}${\bf Q_3}$ that is univalent on $U_{R_2+1}$ for which we have
\begin{align}
\label{PILQ3app}
\rho_U({\bf Q_3}(z),\phi_{2h}^{-1}(z))&<\frac{\epsilon_2}{3}, \qquad z\in U_{R_2+1},\\
\label{PILQ3hd}
\| {\bf Q_3}^{\natural}\|_{U_{R_2+1}}&\leq K_3 \left (1+\frac{\epsilon_2}{3} \right ),\\
\label{PILQ3fix0}
{\bf Q_3}(0) &=0.
\end{align}
Note that by Lemma \ref{PIL}, since $\delta = 1$, ${\bf Q_3}$ depends directly on $\kappa$, $K_3$, $\epsilon$, $R_2$, $h$ (via the curves $\partial V_h$, $\partial V_{2h}$) and $\phi_{2h}$ so that one can check that ${\bf Q_3}$ ultimately depends on $\kappa$, $\epsilon_1$, $\epsilon_2$, $R$, $h_0$, $r_0$, and $R_0$. \\

{\bf Concluding the Proof of Phase II} 

Now, as ${\bf Q_1}$, ${\bf Q_2}$, and ${\bf Q_3}$ were all constructed using the Polynomial Implementation Lemma, they are all (17+$\kappa$)-bounded compositions of quadratic polynomials. Next define the ($17+\kappa$)-bounded composition  
\begin{align}
\label{Qdef}
{\bf Q}:={\bf Q_3}\circ{\bf Q_2}\circ{\bf Q_1}
\end{align}
$\bf Q$ then has the correct coefficient bound of $17 + \kappa$ as in the statement and, checking the dependencies of each of the compositions ${\bf Q_i}$, $i =1, 2, 3$ as well as those of the constants $K_2$, $K_3$, one see that ${\bf Q}$ depends on $\kappa$, $\epsilon_1$, $\epsilon_2$, $R$, $h_0$, $r_0$, $R_0$, and $\calE$ which is the same as given in the statement. 

Using the definitions of the compositions ${\bf Q_1}$, ${\bf Q_2}$, ${\bf Q_3}$ (on pgs. \pageref{bfQ1}, \pageref{bfQ2}, and \pageref{bfQ3}, respectively), 3. of Claim \ref{Q1Approx}, and  Claim \ref{GivingUp1Epsilon1}, \eqref{PILQ3A} (respectively) we showed the following: 
	
\begin{enumerate}
\item  ${\bf Q_1}$ is univalent on $U_{R''-3\epsilon_1} \supset U_{R''-4\epsilon_1}$ and ${\bf Q_1}(U_{R''-4\epsilon_1})\subset \phi_{2h}(U_{R''-3\epsilon_1})$,

\item ${\bf Q_2}$ is univalent on $\phi_{2h}(U_{R''-3\epsilon_1})$  and ${\bf Q_2}(\phi_{2h}(U_{R''-3\epsilon_1})) \subset U_{R_2 + 1}$, 

\item  ${\bf Q_3}$ is univalent on $U_{R_2 + 1}$.
\end{enumerate}

Combining these three observations, and recalling the definition of $\delta(\epsilon_1) = \sup_{[r_0, R_0]}{(R-R'')} +5\epsilon_1$ which we set in \eqref{deltadef} at the end of the section on ideal loss of domain, we see that the composition  ${\bf Q}$ is univalent on $U_{R''-4\epsilon_1}$ therefore
univalent on a neighborhood of $\overline U_{R''-5\epsilon_1} \supset \overline U_{R -\delta(\epsilon_1)}$ (this is the reason why the function $\delta: (0, \tilde \epsilon_1] \mapsto (0, \tfrac{r_0}{4})$ was defined the way it was and in particular why we needed to include an `extra' $\epsilon_1$ in our definition of $\delta$), which gives \emph{i)} in the statement. As all compositions were created with the Polynomial Implementation Lemma, we have using \eqref{PILQ1fix0}, \eqref{PILQ2fix0}, \eqref{PILQ3fix0} that ${\bf Q}(0)=0$ which gives \emph{iii)} in the statement.  

The last thing we need to do is then establish \emph{ii)} in the statement. Recall that in \eqref{deltabound} we chose ${\tilde \epsilon_1}$ sufficiently small such that in particular $\delta(\epsilon_1)<\frac{r_0}{4}$, which ensured that $U_{R-\delta(\epsilon_1)}\neq \emptyset$.

Then for $z \in \overline U_{R-\delta(\epsilon_1)} \subset \overline U_{R''-5\epsilon_1} \subset U_{R''-4\epsilon_1}$, we have
\begin{align}
\notag
\rho_U({\bf Q}(z),\calE(z))\leq &\rho_U({\bf Q_3}\circ{\bf Q_2}\circ{\bf Q_1}(z),\phi_{2h}^{-1}\circ{\bf Q_2}\circ{\bf Q_1}(z)) + \\
\notag
&\rho_U(\phi_{2h}^{-1}\circ{\bf Q_2}\circ{\bf Q_1}(z),\phi_{2h}^{-1}\circ{\hat \calE}\circ{\bf Q_1}(z)) + \\
\label{PIITRI}
&\rho_U(\phi_{2h}^{-1}\circ{\hat \calE}\circ{\bf Q_1},\calE(z)).
\end{align}
We now estimate the three terms on the right hand side of the inequality above.  We have that $z\in \overline U_{R''-5\epsilon_1} \subset U_{R''-4\epsilon_1}$, so ${\bf Q_1}(z)\in \phi_{2h}(U_{R''-3\epsilon_1})$ by Claim \ref{GivingUp1Epsilon1}.  Then ${\bf Q_2}\circ{\bf Q_1}(z)\in U_{R_2+1}$ by \eqref{PILQ3A}.  Thus 
\begin{align}
\label{PIITRI1}
\rho_U({\bf Q_3}\circ{\bf Q_2}\circ{\bf Q_1}(z),\phi_{2h}^{-1}\circ{\bf Q_2}\circ{\bf Q_1}(z))< \frac{\epsilon_2}{3}
\end{align}
by \eqref{PILQ3app}.  For the second term, we still have ${\bf Q_1}(z)\in \phi_{2h}(U_{R''-3\epsilon_1})$ and ${\bf Q_2}\circ{\bf Q_1}(z)\in U_{R_2+1}\subset U_{R_2 +2}$ as above.  Also, we have ${\hat \calE}\circ{\bf Q_1}(z) \in \phi_{2h}(U_{R''-2\epsilon_1})\subset U_{R_2}\subset U_{R_2 +2}$ by  \eqref{r2bound} and \eqref{errormove2}.  Thus, using the hyperbolic convexity lemma (Lemma \ref{stupidfuckinglemma}) and the hyperbolic M-L estimates (Lemma \ref{hyperbolicML}) applied to $\phi_{2h}^{-1}$ on $ U_{R_2 +2}$, 
by \eqref{PILQ2app} and \eqref{phi2h-1hd}, we have 
\begin{align}
\notag
\rho_U(\phi_{2h}^{-1}\circ{\bf Q_2}\circ{\bf Q_1}(z),\phi_{2h}^{-1}\circ{\hat \calE}\circ{\bf Q_1}(z))&< K_3\cdot\frac{\epsilon_2}{3K_3} \\
\label{PIITRI2}
&<\frac{\epsilon_2}{3}.
\end{align}

For the third term we note that $\calE(z)=\phi_{2h}^{-1}\circ{\hat \calE}\circ\phi_{2h}$ on the set $\tilde V_{2h} \setminus \hat{N} \supset \tilde V_h \supset \overline {\tilde U}  \supset \overline U_{R'' - 5\epsilon_1} \supset \overline U_{R-\delta(\epsilon_1)}$ so that $\calE$ and $\phi_{2h}^{-1}\circ{\hat \calE}\circ{\bf Q_1}$ differ in the first mapping of the composition.  We still have ${\bf Q_1}(z)\in \phi_{2h}(U_{R''-3\epsilon_1})$ by Claim \ref{GivingUp1Epsilon1}, and clearly $\phi_{2h}(z)\in \phi_{2h}(\overline U_{R''-5\epsilon_1})\subset \phi_{2h}(U_{R''-3\epsilon_1})$. We need to take care to ensure that we have at least a local version of hyperbolic convexity when it comes to applying the hyperbolic M-L estimates for ${\hat \calE}$ and $\phi_{2h}^{-1}$. 
By \eqref{PILQ1app}  (and the fact that $K_2, K_3 > 1$) we have that ${\bf Q_1}(z) \in \Delta_U (\phi_{2h}(z), \epsilon_2)$. Since by \eqref{e2ub} $\epsilon_2 \le \tilde \epsilon_2 \le 1$, it follows from \eqref{r2bound} that this hyperbolic disc is in turn contained in $U_{R_2 + 1}$. 

Recall that by Claim \ref{etabounds} we had $\eta_2$ depending only on $\epsilon_1$, $h_0$, $r_0$, and $R_0$ for which we had in particular $\|(\phi_{2h}^{-1})^{\natural}\|_{\Delta_U(0,R_2+2)}\leq \eta_2$. If we now apply  the hyperbolic convexity lemma (Lemma \ref{stupidfuckinglemma}) and the hyperbolic M-L estimates (Lemma \ref{hyperbolicML}) for the function $\phi_{2h}^{-1}$ on the ball $\Delta_U (\phi_{2h}(z), \epsilon_2)$, we have that 
 $\phi_{2h}^{-1} (\Delta_U (\phi_{2h}(z),\epsilon_2)) \subset  \Delta_U(z, \eta_2 \epsilon_2) \subset  \Delta_U(z, \epsilon_1)$, the last inclusion following from \eqref{e2ub} which implies that $\epsilon_2 \le \tilde \epsilon_2 \le \tfrac{\epsilon_1}{\eta_2}$. Thus $\phi_{2h}^{-1} (\Delta_U (\phi_{2h}(z),\epsilon_2)) \subset \Delta_U(z, \epsilon_1) \subset U_{R'' - 4 \epsilon_1} \subset  U_{R'' - 3 \epsilon_1}$ so that $ \Delta_U (\phi_{2h}(z), \epsilon_2) \subset \phi_{2h}(U_{R'' - 3 \epsilon_1})$. We also know using \eqref{ehathd} that $|{\hat \calE}^{\natural}|$ is bounded above on $ \phi_{2h}(U_{R''-2\epsilon_1}) \supset  \phi_{2h}(U_{R''-3\epsilon_1}) \supset \Delta_U (\phi_{2h}(z), \epsilon_2)$. 
 
Thus by \eqref{r2bound} and \eqref{errormove2}, we have ${\hat \calE}(\Delta_U (\phi_{2h}(z), \epsilon_2)) \subset {\hat \calE}(\phi_{2h}(U_{R''-3\epsilon_1})) \subset \phi_{2h}(U_{R''-2\epsilon_1})\subset U_{R_2}\subset  U_{R_2 + 2}$ so that in particular ${\hat \calE}(\phi_{2h}(z))$ and ${\hat \calE}({\bf Q_1}(z))$ both lie in $U_{R_2 +2}$ while we know know $|(\phi_{2h}^{-1})^{\natural}|$ is bounded above on $ U_{R_2 + 2}$ using \eqref{phi2h-1hd}.  Then, using \eqref{PILQ1app}, \eqref{ehathd}, \eqref{phi2h-1hd}, and combining the hyperbolic convexity lemma (Lemma \ref{stupidfuckinglemma})  and the hyperbolic M-L estimates (Lemma \ref{hyperbolicML}), applied first to  ${\hat \calE}$ on $\Delta_U (\phi_{2h}(z),\epsilon_2) \subset  \phi_{2h}(U_{R''-2\epsilon_1})$ and then to $\phi_{2h}^{-1}$ on $U_{R_2+2}$, we have
\begin{align}
\notag
\rho_U(\phi_{2h}^{-1}\circ{\hat \calE}\circ{\bf Q_1},\calE(z))&<K_3\cdot K_2\cdot \frac{\epsilon_2}{3K_2 K_3} \\
\label{PIITRI3}
&<\frac{\epsilon_2}{3}.
\end{align}
Finally, using \eqref{PIITRI}, \eqref{PIITRI1}, \eqref{PIITRI2}, and \eqref{PIITRI3}, we have
\begin{align*}
\rho_U({\bf Q}(z),\calE(z))<\epsilon_2
\end{align*}
which establishes \emph{ii)} in the statement and completes the proof of Phase II.
\end{proof}

Before going on to the next chapter, we close with a couple of observations. It is possible if one wishes to find a bound on the absolute value of the hyperbolic derivative of the composition ${\bf Q}$ above on $\overline U_{R'' - \delta(\epsilon_1)}$ which is uniform in terms of the constants $\kappa$, $\epsilon_1$, $\epsilon_2$, $h_0$, $r_0$, and $R_0$ (the hardest part of this is controlling the hyperbolic derivative of ${\bf Q_1}$ which can best be done using \eqref{PILQ1app} and Claim \ref{etabounds} combined with Lemma \ref{lemma4.3} and the version of Cauchy's integral formula for derivatives - e.g. \cite{Con} Corollary IV.5.9). 

However, we do not actually require estimates on the size of ${\bf Q}^\natural$. The reason for this is that the purpose of Phase II is to correct the error from a previous Phase I (Lemma \ref{PhaseI}) approximation which essentially resets the error we need to keep track of. However, as we saw, this Phase II correction itself generates an error which is then passed through the next Phase I approximation. In order to control this, then, we do need an estimate on the hyperbolic derivative of the Phase I composition (which is \emph{(4)} in the statement of Phase I).

%% file: MainProof.tex
\chapter{Proof of the Main Theorem}

In this chapter we prove Theorem \ref{thetheorem}.  The proof of the theorem will follow from a large inductive argument. First, however, we need one more technical lemma. Recall the Siegel disc $U$ for $P$ and that, for $R > 0$, $U_R = \Delta_U(0,R)$ is used to denote the hyperbolic disc of radius $R$ about $0$ with respect to the hyperbolic metric of $U$.
\begin{lemma}
\label{thejordancurveargument}
(The Jordan Curve Argument) Let $U$ and $U_R$ be as above.  Given $0<\epsilon<R$, suppose $g$ is a univalent function defined on a neighbourhood of $\overline U_R$ such that $g(0)=0$ and $\rho_U(g(z),z) \le \epsilon$ on $\partial U_R$.  Then $g(U_R)\supset U_{R-\epsilon}$.
\end{lemma}

\begin{proof}
The function $g$ is a homeomorphism and is bounded on $\overline U_R$, so that it maps $\partial U_R$ to $g(\partial U_R)=\partial(g(U_R))$ which is a Jordan curve in $\C$, while $U_R$ gets mapped to the bounded complementary component of this Jordan curve in view of the Jordan curve theorem (e.g. \cite{Mun} Theorem 13.4 or \cite{New} Theorem V.10.2). Then $0 = g(0)$ lies in $g(U_R)$ and thus inside $\partial (g(U_R))$ and, since this curve avoids $U_{R-\epsilon}$, all of the connected set $U_{R-\epsilon}$ lies inside $\partial (g(U_R))$.  Hence $U_{R-\epsilon}\subset g(U_R)$.  
\end{proof}

\begin{lemma} \label{mediuminductionlemma}
There exist 
\begin{description}
\item[a)] a sequence of positive real numbers $\{\epsilon_k \}_{k=1}^{\infty}$ which converges to $0$, 
\item[b)]  a sequence $\{J_i\}_{i=1}^{\infty}$ of natural numbers, a positive constant $\kappa_0 \ge 576$ and a sequence of compositions of quadratic polynomials $\{{\bf Q^i} \}_{i=1}^{\infty}$,
\item[c)] a sequence of strictly decreasing hyperbolic radii $\{R_i \}_{i=0}^{\infty}$, and 
\item[d)] a sequence of strictly increasing hyperbolic radii $\{S_i \}_{i=0}^{\infty}$, 
\end{description}
such that 

\begin{enumerate}
\item For each $i\geq 0$, $S_i<\frac{1}{10}<\frac{1}{5}<R_i$, 
\item For each $i \ge 1$, ${\bf Q^i}$ is a composition of $J_i$ (17+$\kappa_0$)-bounded quadratic polynomials with ${\bf Q^i}(0)=0$,
\item For each $i \ge 1$, ${\bf Q^i}\circ \cdots \cdots \circ {\bf Q^1}(U_{\frac{1}{20}})\subset U_{S_i}\subset U_{\frac{1}{10}} $, and 
\item For each $i \ge 1$ and $1 \leq m \le J_i$, if ${\bf Q_{m}^{i}}$ denotes the partial composition of the first $m$ quadratics of ${\bf Q^i}$, then, for all $f \in \calS$ and for all $i =2k+1$ odd, there exists $1\leq m_k \leq J_i$ such that, for all $z\in U_{\frac{1}{20}}$, we have 
\begin{align*}
\rho_U({\bf Q^i_{m_k}}\circ {\bf Q^{i-1}}\circ \cdots \cdots \circ {\bf Q^1},f(z))<\epsilon_{k+1}.
\end{align*}
\end{enumerate}
\end{lemma}

Let $J_i$ be the integers and ${\bf Q}^i$ the polynomial compositions from {\bf b)} of the statement above. For $i=0$, set $T_0 = 0$ and, for each $i \ge 1$, set $T_i = \sum_{j=1}^i J_j$. Given this, we define a sequence $\Pm$ in the following natural way: for $m \ge 1$, let $i \ge 1$ be the largest index such that $T_{i-1} < m$ so that $T_{i-1} < m \le T_i = T_{i-1} + J_i$. Then simply let $P_m$ be the $(m - T_{i-1})$th quadratic in the composition ${\bf Q}^i$ (which is a composition of $J_i$ quadratic polynomials). 

The next lemma then follows as an immediate corollary (using \emph{(2)}, \emph{(3)}, and \emph{(4)} above):

\begin{lemma} \label{smallinductionlemma}
There exists a sequence of quadratic polynomials $\{P_m\}_{m=1}^{\infty}$ such that the following hold:
\begin{enumerate}
\item $\{P_m\}_{m=1}^{\infty}$ is (17+$\kappa_0$)-bounded, 
\item $Q_m(U_{\frac{1}{20}})\subset U_{\frac{1}{10}}$ for infinitely many $m$, 
\item For all $f\in \calS$, there exists a subsequence $\{Q_{m_k}\}_{k=1}^{\infty}$ which converges uniformly to $f$ on $U_{\frac{1}{20}}$ as $k \rightarrow \infty$.  
\end{enumerate}
\end{lemma}

\begin{proof}[Proof of Lemma \ref{mediuminductionlemma}]

We begin by fixing the values of the constants in the statements of Phases I and II (Lemmas \ref{PhaseI} and \ref {PhaseII}). Starting with Phase II, let $h_0 = 1$ be the maximum value for the Green's function $G$ and 
let $r_0 = \tfrac{1}{20}$, $R_0 = \tfrac{1}{4} < \tfrac{\pi}{2}$ be the upper and lower bounds for the hyperbolic radii we consider in applying Phase II.  
We will also use $R_0 = \tfrac{1}{4}$ when we apply Phase I and we set $\kappa = \kappa_0 = \kappa_0(\tfrac{1}{4}) \ge 576$ for both Phases I and II.

Let $C:=7$ be the bound on the hyperbolic derivative from part \emph{(4)} of the statement of Phase I and let {$\tilde \epsilon_1 > 0$ and $\delta(x)$ be the function defined on $(0, \tilde \epsilon_1]$} measuring loss of hyperbolic radius from the statement of Phase II, both of which are determined by the values of $h_0$, $r_0$, and $R_0$ which we have just fixed. The proof of Lemma \ref{mediuminductionlemma} will follow quickly from the following claim, which we prove by induction. 

\begin{claim}
\label{inductionclaim}
There exist inductively defined infinite sequences of positive real numbers $\{\epsilon_k \}_{k=1}^{\infty}$, $\{\eta_k \}_{k=1}^{\infty}$, and $\{\sigma_k \}_{k=1}^{\infty}$, sequences of hyperbolic radii $\{R_i \}_{i=0}^{\infty}$ and $\{S_i \}_{i=0}^{\infty}$, integers $\{J_i \}_{i=1}^{\infty}$, and polynomial compositions $\{{\bf Q^i} \}_{i=1}^{\infty}$ such that, for each $n \in \N$, the following hold.

\begin{enumerate}

\item[\namedlabel{ind:i}{i)}] The sequences $\{\epsilon_k \}_{k=1}^{n}$, $\{\eta_k \}_{k=1}^{n}$, and $\{\sigma_k \}_{k=1}^{n}$ satisfy \\
\[  \eta_k= \left\{
\begin{array}{ll}
     \frac{4\epsilon_1}{3} + \delta(\epsilon_1),  & i=1, \\
     (\frac{4}{3}+\frac{1}{3C})\epsilon_k + \delta(\epsilon_k), & 2 \leq k \leq n, \\
\end{array} 
\right. \]

\[ \hspace{-1.2cm}\sigma_k= \left\{
\begin{array}{ll}
     \frac{4\epsilon_1}{3},   & i=1, \\
     (\frac{4}{3}+\frac{1}{3C})\epsilon_k,  & 2 \leq k \leq n, \\
\end{array} 
\right. \]
where in addition we require that $0< \epsilon_k < \sigma_k<\eta_k<\frac{1}{40\cdot 2^k}$ and that $\epsilon_k \le \tilde \epsilon_1$ for each $1\leq k \leq n$.

\item[\namedlabel{ind:ii}{ii)}] The sequence $\{R_i \}_{i=0}^{2n-1}$ is strictly decreasing and is given by $R_0=\frac{1}{4}$,\\ $R_1=\frac{1}{4}-\frac{\epsilon_1}{3}$, and 
\[  \hspace{-.4cm}R_i= \left\{
\begin{array}{ll}
     \frac{1}{4}-(\sum_{j=1}^{k}\eta_j) -\frac{\epsilon_{k+1}}{3C},  & i=2k \; \mathrm{ for \; some } \; 1\leq k \leq n-1, \\
     &\\
     \frac{1}{4}-(\sum_{j=1}^{k}\eta_j) - (\frac{1}{3}+\frac{1}{3C})\epsilon_{k+1}, & i=2k+1 \: \mathrm{ for \; some }\\ &   1\leq k \leq n-1.\\
\end{array} 
\right. \]

The sequence $\{S_i \}_{i=0}^{2n-1}$ is strictly increasing and is given by $S_0=\frac{1}{20}$,\\ $S_1=\frac{1}{20}+\frac{\epsilon_1}{3}$, and 
\[  S_i= \left\{
\begin{array}{ll}
     \frac{1}{20}+(\sum_{j=1}^{k}\sigma_j) +\frac{\epsilon_{k+1}}{3C},  & i=2k \; \mathrm{ for \; some } \; 1\leq k \leq n-1, \\
     &\\
     \frac{1}{20}+(\sum_{j=1}^{k}\sigma_j) + (\frac{1}{3}+\frac{1}{3C})\epsilon_{k+1}, & i=2k+1 \; \mathrm{ for \; some } \\ &1\leq k \leq n-1.\\
\end{array} 
\right. \]

\item[\namedlabel{ind:iii}{iii)}] $\frac{1}{20} \le S_i<\frac{1}{10}<\frac{1}{5}<R_i \le \frac{1}{4}$ for each $0 \leq i \leq 2n-1$.  

\item[\namedlabel{ind:iv}{iv)}] For each $1\leq i \leq 2n-1$, ${\bf Q^i}$ is a (17+$\kappa_0$)-bounded composition of $J_i$ quadratic polynomials with ${\bf Q^i}(0)=0$.   

\item[\namedlabel{ind:v}{v)}] For each $1 \le i \le 2n-1$, the branch of $({\bf Q^i})^{-1}$ which fixes $0$ is well-defined and univalent on $U_{R_i}$ and maps $U_{R_i}$ inside $U_{R_{i-1}}$. The branch of $({\bf Q^i}\circ \cdots \cdots \circ {\bf Q^2} \circ {\bf Q^1})^{-1}$ which fixes $0$ is then also well-defined and univalent on $U_{R_{i}}$.  

\item[\namedlabel{ind:vi}{vi)}] For each $1\leq i \leq 2n-1$, ${\bf Q^i}$ is univalent on $U_{S_{i-1}}$ and
\begin{align*}
{\bf Q^i}(U_{S_{i-1}}) \subset U_{S_{i}}.
\end{align*}
Thus ${\bf Q^i} \circ \cdots \cdots \circ {\bf Q^1}$ is univalent on $U_{\frac{1}{20}}$ and
\begin{align*}
{\bf Q^i} \circ \cdots \cdots \circ {\bf Q^1}(U_{\frac{1}{20}})\subset U_{S_{i}}\subset U_{\frac{1}{10}}. 
\end{align*}

\item[\namedlabel{ind:vii}{vii)}] If $i=2k$  with $1\leq k \leq n-1$ is even, and $z \in U_{R_{i-1}-\delta(\epsilon_k)}$, 

\begin{align*}
\rho_U({\bf Q^{i}}(z),({\bf Q^{i-1}}\circ \cdots \cdots \circ {\bf Q^{1}})^{-1}(z) )<\frac{\epsilon_{k+1}}{3C}
\end{align*}
where we use the same branch of $({\bf Q^{i-1}}\circ \cdots \cdots \circ {\bf Q^{1}})^{-1}$ which fixes $0$ from {\ref{ind:v}} above.

\hspace{-2cm} For the final two hypotheses, let $i=2k+1$ with $0\leq k \leq n-1$ be odd.  

\item[\namedlabel{ind:viii}{viii)}] If $z \in U_{R_i}$, using the same inverse branch mentioned in {\ref{ind:v}} we have
\begin{align*}
\rho_U(({\bf Q^{i}}\circ \cdots \cdots \circ {\bf Q^{1}})^{-1}(z),z)<\epsilon_{k+1}.
\end{align*}

\item[\namedlabel{ind:ix}{ix)}] If, for each $1\leq m \leq J_i$, ${\bf Q_{m}^{i}}$ denotes the partial composition of the first $m$ quadratics of ${\bf Q^{i}}$, then for all $f \in \calS$ there exists $1\leq m \leq J_i$, such that, for all $z \in U_{\frac{1}{20}}$, we have 
\begin{align*}
\rho_U({\bf Q^i_m}\circ {\bf Q^{i-1}}\circ \cdots \cdots \circ {\bf Q^{1}}(z), f(z))<\epsilon_{k+1}. 
\end{align*}
\end{enumerate}
\end{claim}

Remarks: 

\begin{enumerate}

\item Statements \emph{\ref{ind:i}}-\emph{\ref{ind:iii}} are designed for keeping track of the domains on which estimates are holding and in particular to ensure that these domains do not get too small and that the constants $\epsilon_i$ which keep track of the accuracy of the approximations do indeed tend to $0$. The outer radii $R_i$ are chosen primarily so that the image of $U_{R_i}$ under the inverse branch of ${\bf Q^i}$ which fixes $0$ is contained in $U_{R_{i-1}}$ (this is \emph{\ref{ind:v}} above). This allows us to compose the inverses of these compositions and then approximate this composition of inverses by means of Phase II. The inner radii $S_i$ are chosen primarily so that the image of $U_{S_{i-1}}$ under the polynomial composition ${\bf Q^i}$ lies inside  $U_{S_i}$ (this is \emph{\ref{ind:vi}} above). This allows us to compose these polynomial compositions and gives us our iterates which remain bounded and approximate the elements of $\calS$. 


\item Statement \emph{\ref{ind:vii}} is a `Phase II' statement regarding error correction using Phase II of the inverse of an earlier polynomial composition. Effectively, the Phase II correction compensates for the error in the previous Phase I composition, whose deviation from the identity is measured in \emph{\ref{ind:viii}} above.

\item Statements \emph{\ref{ind:viii}}  and \emph{\ref{ind:ix}} are  `Phase I' statements. Statement \emph{\ref{ind:viii}} is a bound on the error to be corrected by the next Phase II approximation. Statement \emph{\ref{ind:ix}} is the key element for proving Theorem \ref{thetheorem}.

\item It follows readily from \emph{\ref{ind:i}} that the sequence $\{\epsilon_i\}_{i=1}^\infty$ converges to $0$ exponentially fast, which gives a) in the statement of the Lemma. b) follows from \emph{\ref{ind:iv}} and our choice of $\kappa_0$, while c) and d) follow from \emph{\ref{ind:ii}}.

\item Part \emph{(1)} of the second part of the statement of the Lemma follows from \emph{\ref{ind:iii}} above while \emph{(2)} of the statement follows from \emph{\ref{ind:iv}}. Lastly, \emph{(3)} follows from \emph{\ref{ind:vi}} while \emph{(4)} follows from \emph{\ref{ind:ix}}. 

\end{enumerate}


\begin{center}
\begin{figure}[ht]
\label{blockdiagram}

\scalebox{.48}{
\begin{tikzpicture}

\draw (0,3) -- (22,3);
	\node at (1,3.2) {$k=0$}; 
	\node at (4,3.2) {$k=1$};
	\node at (12,3.2) {Block $k$};
	\node at (20,4.2) {Last Block};
	\node at (20,3.7) {$k=n-1$};
	\node at (20,3.2) {($n$ blocks)};
	
\draw (0,3.2) -- (0,2.8);	
\draw (2,3.2) -- (2,2.8);
\draw (6,3.2) -- (6,2.8);
\draw (10,3.2) -- (10,2.8);
\draw (14,3.2) -- (14,2.8);
\draw (18,3.2) -- (18,2.8);
\draw (22,3.2) -- (22,2.8);


\draw (0,0) -- (22,0);
	\node at (-2.68,3.2) {Index:};
	\node at (-2.4,2.3) {Domains:};
	\node at (-2.1,.5) {Polynomials:};
	\node at (-2.65,-.5) {Phase:};
	\node at (-2.28,-1.7) {Number of};
	\node at (-1.98,-2.2) {Compositions:};
	\node at (1,.5) {${\bf Q^1}$};
	\node at (1,-.5) {Phase I};
	\node at (3,.5) {${\bf Q^2}$};
	\node at (3,-.5) {Phase II};
	\node at (5,.5) {${\bf Q^3}$};
	\node at (5,-.5) {Phase I};
	\node at (7,.5) {$\cdots$};
	\node at (9,.5) {$\cdots$};
	\node at (11,.5) {${\bf Q^{2k}}$};
	\node at (11,-.5) {Phase II};
	\node at (13,.5) {${\bf Q^{2k+1}}$};
	\node at (13,-.5) {Phase I};
	\node at (15,.5) {$\cdots$};
	\node at (17,.5) {$\cdots$};
	\node at (19,.5) {${\bf Q^{2n-2}}$};
	\node at (19,-.5) {Phase II};
	\node at (21,.5) {${\bf Q^{2n-1}}$};
	\node at (21,-.5) {Phase I};
	
\draw (0,2) -- (0,-2);
	\node at (0,2.3) {$U_{R_0}$};
	\node at (0,-2.2) {$i=0$};
	
\draw (2,2)  -- (2,-2); 
	\node at (2,2.3) {$U_{R_1}$};
	\node at (2,-2.2) {$i=1$};
	
\draw[dashed] (4,2) -- (4,-2);
	\node at (4,2.3) {$U_{R_2}$};
	\node at (4,-2.2) {$i=2$};
	
\draw (6,2) -- (6,-2);
	\node at (6,2.3) {$U_{R_3}$};
	\node at (6,-2.2) {$i=3$};	
	
\draw[dashed] (8,2) -- (8,-2);
		
\draw (10,2) -- (10,-2);
	\node at (10,2.3) {$U_{R_{2k-1}}$};
	\node at (10,-2.2) {$i=2k-1$};
	
\draw[dashed] (12,2) -- (12,-2);
	\node at (12,2.3) {$U_{R_{2k}}$};
	\node at (12,-2.2) {$i=2k$};

\draw (14,2) -- (14,-2);
	\node at (14,2.3) {$U_{R_{2k+1}}$};
	\node at (14,-2.2) {$i=2k+1$};
	
\draw[dashed] (16,2) -- (16,-2);

\draw (18,2) -- (18,-2);	
	\node at (18,2.3) {$U_{R_{2n-3}}$};
	\node at (18,-2.2) {$i=2n-3$};

\draw[dashed] (20,2) -- (20,-2);
	\node at (20,2.3) {$U_{R_{2n-2}}$};
	\node at (20,-2.2) {$i=2n-2$};

\draw (22,2) -- (22,-2);
	\node at (22,2.3) {$U_{R_{2n-1}}$};
	\node at (22,-2.2) {$i=2n-1$};	
	
\end{tikzpicture}
}
\caption{A block diagram illustrating the induction scheme.}
\end{figure}
\end{center}


\begin{claimproof}

{\bf Base Case:} $n=1$. 

Recall the bound {$\tilde \epsilon_1 > 0$ and function $\delta(x)$ defined on $(0, \tilde \epsilon_1]$} whose existence are given by Phase II (recall that we have fixed the values of $h_0$, $r_0$, $R_0$ at the start of the proof) and that {$\delta(x) \to 0$ as $x \to 0_+$. We can then pick} $0< \epsilon_1 \, \le\, \tilde \epsilon_1$ such that if we set 
\begin{align*}
\eta_{1}&=\frac{4}{3}\epsilon_{1}+\delta(\epsilon_1), \\
\sigma_{1}&=\frac{4}{3}\epsilon_1,
\end{align*}
then we can ensure that $0< \epsilon_1 < \,\, \sigma_1<\eta_1<\frac{1}{40\cdot 2} = \tfrac{1}{80}$.  This verifies \emph{\ref{ind:i}}. Now recall that we already set $R_0=\frac{1}{4}$, let $S_0=\frac{1}{20}$, and set
\begin{align*}
R_1&= \frac{1}{4}-\frac{\epsilon_1}{3},\\
S_1&= \frac{1}{20}+\frac{\epsilon_1}{3},
\end{align*}
which verifies \emph{\ref{ind:ii}} and then \emph{\ref{ind:iii}} follows easily. Applying Lemma \ref{net}, we choose an $\frac{\epsilon_1}{3}$-net $\{f_0, f_1, \cdots,  f_{N_1+1} \}$ for $\calS$ (consisting of elements of $\calS$) on $U_{\frac{1}{2}}$, where $N_1 = N_1 (\epsilon_1) \in \mathbb{N}$, and with $f_0=f_{N_1+1}=\mathrm{Id}$. Apply Phase I (Lemma \ref{PhaseI}) for this collection of functions with $R_0=\frac{1}{4}$, $\varepsilon=\frac{\epsilon_1}{3}$, to obtain $M_1 \in \N$, and a ($17 + \kappa_0$)-bounded finite sequence $\{P_m\}_{m=1}^{(N_1 +1)M_1}$ of quadratic polynomials both of which depend directly on $R_0$, $\kappa_0$,  $N_1$, the functions $\{f_i\}_{i=0}^{N_1 + 1}$, and $\epsilon$
and thus ultimately on $\epsilon_1$ and  $\{f_i\}_{i=0}^{N_1 + 1}$ such that, for $1\leq i \leq N_1+1$, if we let $\bf Q_m^1$, $1 \le m \le J_1$ denote the composition of the first $m$ polynomials of this sequence, we have

\begin{enumerate}

\item ${\bf Q^1_{iM_1}}(0)=0$,

\item ${\bf Q^1_{iM_1}}$ is univalent on $U_{\frac{1}{2}}$,

\item $\rho_{U}(f_i(z), {\bf Q^1_{iM_1}}(z))<\frac{\epsilon_1}{3}$ on $U_{\frac{1}{2}}$,

\item $\|({\bf Q^1_{iM_1}})^{\natural} \|_{U_{\frac{1}{4}}}\leq C$.

\end{enumerate} 

Now set ${\bf Q^1}=Q_{(N_1+1)M_1}$. By (1) ${\bf Q^1}(0)=0$ and, as Phase I guarantees ${\bf Q^1}$ is ($17 + \kappa_0$)-bounded, on setting $J_1 = (N_1+1)M_1$, \emph{\ref{ind:iv}} is verified.  

Now we have that each ${\bf Q^1_{iM_1}}$ is univalent on $U_{\frac{1}{2}}\supset {\overline U_{\frac{1}{4}}}={\overline U_{R_0}}$ by (2) above. Further, by (3), if $\rho_U(0,z)=\frac{1}{4}$, then $\rho_U({\bf Q^1}(z),z)< \frac{\epsilon_1}{3}$, so by (1) and the Jordan curve argument (Lemma \ref{thejordancurveargument}), ${\bf Q^1}(U_{R_0})\supset U_{R_1}$. The branch of $({\bf Q^1})^{-1}$ which fixes $0$ is then well-defined and univalent on $U_{R_1}$ and maps this set inside $U_{R_0}$. With this we have verified \emph{\ref{ind:v}}. 

Likewise, if $\rho_U(0,z) < \frac{1}{20}$, then $\rho_U({\bf Q^1}(z),0)<\frac{1}{20}+\frac{\epsilon_1}{3}$. This implies ${\bf Q^1}(U_{S_0})\subset U_{S_1}$ and, since by \emph{\ref{ind:iii}}, $S_1 < \tfrac{1}{10}$ while by (2) above ${\bf Q^1}$ is univalent on $U_{\frac{1}{2}}\supset U_{\frac{1}{10}}$, which verifies \emph{\ref{ind:vi}}.  We observe that hypothesis \emph{\ref{ind:vii}} is vacuously true as it is concerned only with Phase II.  

Now let $z \in U_{R_1}$. Using the same branch of $({\bf Q^1})^{-1}$  as in \emph{\ref{ind:v}}, it follows from \emph{\ref{ind:v}} that we can write $z={\bf Q^{1}}(w)$ for some $w\in U_{R_0}$.  Since $f_{N_1+1}= Id$, it follows from (3) above that
\begin{align*}
\rho_U(({\bf Q^1})^{-1}(z),z)&=\rho_U(w,{\bf Q^1}(w)) \\
&<\frac{\epsilon_1}{3} \\
&<\epsilon_1
\end{align*}
which verifies \emph{\ref{ind:viii}}. 

Finally, let $z \in U_{\frac{1}{20}}$.  For $f \in \calS$, let $f_i$ be a member of the net for which $\rho_U(f(w),f_i(w))<\frac{\epsilon_1}{3}$ on $U_{\frac{1}{2}}\supset U_{\frac{1}{20}}$, and, using (3), let ${\bf Q_{iM_1}^1} $ be a partial composition which satisfies $\rho_U({\bf Q_m^1}(w),f_i(w))<\frac{\epsilon_1}{3}$ on $U_{\frac{1}{2}}\supset U_{\frac{1}{20}}$
Then, on setting $m = iM_1$, 
\begin{align*}
\rho_U({\bf Q_m^1}(z),f(z))&\leq \rho_U({\bf Q_m^1}(z),f_i(z))+\rho_U(f_i(z),f(z)) \\
&\leq \frac{\epsilon_1}{3}+\frac{\epsilon_1}{3} \\
&<\epsilon_1,
\end{align*}
which verifies \emph{\ref{ind:ix}} and completes the base case. \\

{\bf Induction Hypothesis:}  Assume \emph{\ref{ind:i}}-\emph{\ref{ind:ix}} hold for some arbitrary $n \ge 1$.  \\

{\bf Induction Step:}  We now show this is true for $n+1$. 

Since the above hypotheses hold for $n$, we have already defined $R_{2n-1}=R_{2n-2}-\frac{\epsilon_n}{3}$. Using \emph{\ref{ind:viii}} for $n$ with $i=2n-1$ we have 
\begin{align}
\label{PhaseIIhypothesisn}
\rho_U&(({\bf Q^{2n-1}}\circ \cdots \cdots \circ {\bf Q^{1}})^{-1}(z),z)<\epsilon_n, \hspace{-2cm} &z \in U_{R_{2n-1}}
\end{align}  
where of course we are using the branch of $({\bf Q^{2n-1}}\circ \cdots \cdots \circ {\bf Q^{1}})^{-1}$ from \emph{\ref{ind:v}} which fixes $0$. 

Recalling that the function $\delta: (0, \tilde \epsilon_1] \mapsto (0, \tfrac{r_0}{4})$ in Phase II (Lemma \ref{PhaseII}) has a limit of $0$ from the right, we can pick $\epsilon_{n+1}>0$ sufficiently small such that $\epsilon_{n+1} \le \tilde \epsilon_1$ and if we set
\begin{align}
\label{etasigmadef1}
\eta_{n+1}&=\left( \frac{4}{3}+\frac{1}{3C}\right)\epsilon_{n+1}+\delta(\epsilon_{n+1}), \\
\sigma_{n+1}&=\left( \frac{4}{3}+\frac{1}{3C} \right)\epsilon_{n+1},
\label{etasigmadef2}
\end{align}
then we can ensure
\begin{align}
\label{epsilonetasigmasmall}
0&< \epsilon_{n+1} < \sigma_{n+1}<\eta_{n+1}<\frac{1}{40\cdot 2^{n+1}}
\end{align}
which verifies \emph{\ref{ind:i}} for n+1.  If we now apply Phase II, with $\kappa_0$, $h_0$, $r_0$, $R_0$ as above, $R=R_{2n-1}$, $\varepsilon_1=\epsilon_n$ (recall that $\epsilon_n \le \tilde \epsilon_1$ in view of hypothesis \emph{\ref{ind:i}} for $n$), $\varepsilon_2=\frac{\epsilon_{n+1}}{3C}$, and $\calE = ({\bf Q^{2n-1}}\circ \cdots \circ {\bf Q^{1}})^{-1}$ and make use of \eqref{PhaseIIhypothesisn}, we can find a $(17 + \kappa_0)$-bounded composition of quadratic polynomials ${\bf Q^{2n}}$ which depends immediately on $\kappa_0$, $\epsilon_n$, $\epsilon_{n+1}$, $R$ and $\calE$ and thus ultimately on $\{\epsilon_k\}_{k=1}^{n+1}$ and $\{{\bf Q^i}\}_{i=1}^{2n-1}$, such that ${\bf Q^{2n}}$ is univalent on a neighborhood of $\overline U_{R_{2n-1}-\delta(\epsilon_n)}$, satisfies ${\bf Q^{2n}}(0)=0$, and
\begin{align}
\label{PhaseIIconclusionn+1}
\rho_U&({\bf Q^{2n}}(z),({\bf Q^{2n-1}}\circ \cdots \circ {\bf Q^1})^{-1}(z))<\frac{\epsilon_{n+1}}{3C}, \hspace{-1cm}  &z\in \overline U_{R_{2n-1}-\delta(\epsilon_n)},
\end{align}
which verifies \emph{\ref{ind:vii}} for $n+1$.  Note that, because of the upper bound ${\tilde \epsilon_2}$ in the statement of Phase II, we may need to make $\varepsilon_{n+1}$ smaller, if necessary.  However, this does not affect the estimates on $\eta_{n+1}$ or $\sigma_{n+1}$ or any of the other dependencies for ${\bf Q^{2n}}$ above. Finally ${\bf Q^{2n}}(0)=0$ from above, so that, if we let $J_{2n}$ be the number of quadratics in ${\bf Q^{2n}}$, we see that the first half of \emph{\ref{ind:iv}} for $n+1$ is also verified.  Now set
\begin{align}
\label{R2ndef}
R_{2n}&=R_{2n-1}-\epsilon_n- \delta(\epsilon_n) - \frac{\epsilon_{n+1}}{3C}, \\
\label{S2ndef}
S_{2n}&=S_{2n-1}+\epsilon_n+\frac{\epsilon_{n+1}}{3C}.
\end{align}
We observe that the $\epsilon_n$ change in radius above is required in view of \emph{\ref{ind:viii}} which measures how much the function  $({\bf Q^{2n-1}}\circ \cdots \cdots \circ {\bf Q^1})^{-1}$ which we are approximating moves points on $U_{R_{2n-1}-\delta(\epsilon_n)}$, the $\delta(\epsilon_n)$ change is loss of domain incurred by Phase II, while the additional $\tfrac{\epsilon_{n+1}}{3C}$ is to account for the error in the Phase II approximation (the factor of $C$ arising from the fact that this error needs to be passed through a subsequent Phase I to verify \emph{\ref{ind:ix}} for $n+1$). 

A final observation worth making is that here we are dealing with a loss of radius in passing from $R_{2n-1}$ to $R_{2n}$ arising from two distinct sources - the initial loss of domain by an amount $\delta(\epsilon_n)$ arising from the need to make a Phase II approximation and the subsequent losses of $\epsilon_n$ and $\tfrac{\epsilon_{n+1}}{3C}$ which arise via the Jordan curve argument (Lemma \ref{thejordancurveargument}) due the amount that ${\bf Q^{2n}}$ moves points on $U_{R_{2n-1}-\delta(\epsilon_n)}$ (for details see \eqref{firsthalfvvin+1} below as well as the discussions immediately preceding and succeeding this inequality).

One easily checks that using hypotheses \emph{\ref{ind:i}} and \emph{\ref{ind:ii}} for $n$ that
\begin{align}
\label{R2nEqn}
R_{2n}&=\left(\frac{1}{4}-\left(\sum_{j=1}^{n-1}\eta_j\right) - \left(\frac{1}{3}+\frac{1}{3C}\right)\epsilon_n \right)-\epsilon_n-\delta(\epsilon_n)-\frac{\epsilon_{n+1}}{3C} \nonumber \\
&=\frac{1}{4}-\left(\sum_{j=1}^{n}\eta_j\right) -\frac{\epsilon_{n+1}}{3C}, \\
\label{S2nEqn}\
S_{2n}&=\left(\frac{1}{20}+\left(\sum_{j=1}^{n-1}\sigma_j\right) + \left(\frac{1}{3}+\frac{1}{3C}\right)\epsilon_n \right) +\epsilon_n+\frac{\epsilon_{n+1}}{3C} \nonumber\\
&=\frac{1}{20}+\left(\sum_{j=1}^{n}\sigma_j\right) +\frac{\epsilon_{n+1}}{3C},
\end{align}
which verifies the first half of \emph{\ref{ind:ii}} for $n+1$. We also observe at this stage that one can verify that the total loss of radius on passing from $R_{2n-2}$ to $R_{2n}$ is $\tfrac{\epsilon_n}{3} + \left ( \epsilon_n + \delta (\epsilon_n) + \tfrac{\epsilon_{n+1}}{3C} \right ) = \left ( \tfrac{4\epsilon_n}{3} + \delta(\epsilon_n) \right ) + \tfrac{\epsilon_{n+1}}{3C}$ which explains the form of the constants $\eta_i$ in part \emph{\ref{ind:i}} of the induction hypothesis. A similar argument also accounts for the other constants $\sigma_i$ in \emph{\ref{ind:i}}. Further, clearly $R_{2n} \le \tfrac{1}{4}$ and, using \emph{\ref{ind:iii}} for $n$ and \eqref{epsilonetasigmasmall},
\begin{align*}
R_{2n}&=\frac{1}{4}-\left( \sum_{j=1}^{n}\eta_j \right) -\frac{\epsilon_{n+1}}{3C}\\
&>\frac{1}{4}-\left(\sum_{j=1}^{n}\frac{1}{40\cdot 2^j}\right) -\frac{1}{40\cdot 2^{n+1}} \\
&=\frac{1}{4}-\frac{1}{40}\left (1-\frac{1}{2^n}-\frac{1}{2^{n+1}}\right ) \\
&>\frac{1}{5}.
\end{align*}
The calculation for $S_{2n}$ is similar, and thus we have verified the first half of \emph{\ref{ind:iii}} for $n+1$.  Combining \eqref{PhaseIIhypothesisn} and \eqref{PhaseIIconclusionn+1} we have, on $\overline U_{R_{2n-1}-{\delta(\epsilon_n)}}$,
\begin{align}
\label{firsthalfvvin+1}
\rho_U({\bf Q^{2n}}(z),z)&\leq \rho_U({\bf Q^{2n}}(z),({\bf Q^{2n-1}}\circ \cdots \circ {\bf Q^1})^{-1}(z)) \nonumber \\
&\hspace{1.11cm}+\rho_U(({\bf Q^{2n-1}}\circ \cdots \circ {\bf Q^1})^{-1}(z),z) \nonumber\\
&<\frac{\epsilon_{n+1}}{3C}+\epsilon_n
\end{align}
This, combined with the Jordan curve argument (Lemma \ref{thejordancurveargument}) and \eqref{R2ndef} above implies that 
\begin{align}
\label{PhaseIIContainment}
{\bf Q^{2n}}(U_{R_{2n-1}-\delta(\epsilon_n)}) \supset U_{R_{2n-1}-\delta(\epsilon_n)- \epsilon_n - \tfrac{\epsilon_{n+1}}{3C}} =  U_{R_{2n}},  
\end{align}
and, since from above $\bf Q^{2n}$ is univalent on a neighbourhood of $\overline U_{R_{2n-1} - \delta(\epsilon_n)}$, the branch of $({\bf Q^{2n}})^{-1}$ which fixes $0$ is well-defined on $U_{R_{2n}}$ and maps this set inside $U_{R_{2n-1}-\delta(\epsilon_n)} \subset U_{R_{2n-1}}$ which verifies the first half of \emph{\ref{ind:v}} for $n+1$. 

By \eqref{R2ndef}, the first half of \emph{\ref{ind:iii}} for $n+1$, and \emph{\ref{ind:iii}} for $n$, $R_{2n-1}-\delta(\epsilon_n) > R_{2n} > \tfrac{1}{5} > S_{2n-1}$ so that ${\bf Q^{2n}}$ is univalent on $U_{S_{2n-1}}$. It then follows using 
\eqref{S2ndef}, \eqref{firsthalfvvin+1} that 
\begin{align*}
{\bf Q^{2n}}(U_{S_{2n-1}})\subset U_{S_{2n-1}+\epsilon_n+\frac{\epsilon_{n+1}}{3C}}=U_{S_{2n}}.
\end{align*}
Since, by the first half of \emph{\ref{ind:iii}} for $n+1$, $S_{2n} < \tfrac{1}{10}$ and, together with \emph{\ref{ind:vi}} for $n$, this verifies half of \emph{\ref{ind:vi}} for $n+1$ and finishes the Phase II portion of the induction step. \\

Now again apply Lemma \ref{net} to construct an $\frac{\epsilon_{n+1}}{3}$-net $\{f_0, f_1, \cdots, f_{N_{n+1}+1} \}$ for $\calS$ (which again consists of elements of $\calS$) on $U_{\frac{1}{2}}$, where we obtain $N_{n+1}=N_{n+1}(\epsilon_{n+1})\in \mathbb{N}$ and require $f_0=f_{N_{n+1}+1}=\mathrm{Id}$. We apply Phase I (Lemma \ref{PhaseI}) with $R_0 =\frac{1}{4}$ and $\varepsilon=\frac{\epsilon_{n+1}}{3}$ for this collection of functions to obtain $M_{n+1} \in \N$, and a (17+$\kappa_0$)-bounded sequence of quadratic polynomials $\{P_m \}_{m=1}^{(N_{n+1}+1)M_{n+1}}$ both of which depend directly on $R_0$, $\kappa_0$, $N_{n+1}$, the functions $\{f_i\}_{i=0}^{N_{n+1} + 1}$, and $\epsilon$ 
and thus ultimately on $\{\epsilon_k\}_{k=1}^{n+1}$ and  $\{f_i\}_{i=0}^{N_{n+1} + 1}$. 

Now let $J_{2n+1}=M_{n+1}(N_{n+1}+1)$ be the number of quadratics and denote similarly to before the composition of the first $m$ of these quadratics by ${\bf Q^{2n+1}_m}$. By Phase I these compositions satisfy, for each $1\le i \le N_{n+1}+1$ 

\begin{enumerate}
\item ${\bf Q^{2n+1}_{iM_{n+1}}}(0)=0$,
\item ${\bf Q^{2n+1}_{iM_{n+1}}}$ is univalent on $U_{\frac{1}{2}}$,
\item $\rho_U(f_i(z),{\bf Q^{2n+1}_{iM_{n+1}}}(z))<\frac{\epsilon_{n+1}}{3}$, $z\in U_{\frac{1}{2}}$,
\item $\| ({\bf Q_{iM_{n+1}}^{2n+1}})^{\natural} \|_{U_{\frac{1}{4}}} \leq C$.
\end{enumerate}

Now set ${\bf Q^{2n+1}}:={\bf Q^{2n+1}_{(N_{n+1}+1)M_{n+1}}}$. The polynomial composition ${\bf Q^{2n+1}}$ is then a $(17+\kappa_0)$-bounded composition of $J_{2n+1}$ quadratic polynomials which by (1) satisfies ${\bf Q^{2n+1}}(0)=0$. This then verifies \emph{\ref{ind:iv}} for $n+1$. 

Next, we define  
\begin{align}
\label{R2nplus1def}
R_{2n+1}&=R_{2n}-\frac{\epsilon_{n+1}}{3}, \\
\label{S2nplus1def}
S_{2n+1}&=S_{2n}+\frac{\epsilon_{n+1}}{3}.
\end{align}

We observe that the $\tfrac{\epsilon_{n+1}}{3}$ change in radius above is required in view of (3) above. One easily checks, using the above and \eqref{R2nEqn}, \eqref{S2nEqn},
\begin{align*}
R_{2n+1}&= \left (\frac{1}{4}-\sum_{j=1}^{n}\eta_j -\frac{\epsilon_{n+1}}{3C} \right )-\frac{\epsilon_{n+1}}{3} \\
&=\frac{1}{4}-\sum_{j=1}^{n}\eta_j - \left (\frac{1}{3} +  \frac{1}{3C} \right )  \epsilon_{n+1}, \\
S_{2n+1}&= \left (\frac{1}{20}+\sum_{j=1}^{n}\sigma_j +\frac{\epsilon_{n+1}}{3C} \right ) +\frac{\epsilon_{n+1}}{3} \\
&=\frac{1}{20}+\sum_{j=1}^{n}\sigma_j + \left (\frac{1}{3}+\frac{1}{3C} \right )\epsilon_{n+1}.
\end{align*}
Thus we have verified \emph{\ref{ind:ii}} for $n+1$ and a similar calculation (again using the first half of \emph{\ref{ind:iii}} for $n+1$ and \eqref{epsilonetasigmasmall}) to that for verifying the first half of \emph{\ref{ind:iii}} for $n+1$ allows us to complete the verification of \emph{\ref{ind:iii}} for $n+1$. 

By (1) and (3) above applied to the function $f_{N_{n+1}+1}= Id$, together with \emph{\ref{ind:iii}} for $n+1$, \eqref{R2nplus1def}, and Lemma \ref{thejordancurveargument}, we have 
\begin{align}
\label{PhaseIContainment}
{\bf Q^{2n+1}}(U_{R_{2n}})\supset U_{R_{2n}-\frac{\epsilon_{n+1}}{3}} =  U_{R_{2n+1}},
\end{align}
while ${\bf Q^{2n+1}}$ is univalent on a neighbourhood of this set by (2). Hence, the branch of $({\bf Q^{2n+1}})^{-1}$ which fixes $0$ is well-defined and univalent on $U_{R_{2n+1}}$ and maps $U_{R_{2n+1}}$ inside $U_{R_{2n}}$ which then verifies \emph{\ref{ind:v}} for $n+1$.  

By (2) above and the first half of \emph{\ref{ind:iii}} for $n+1$, ${\bf Q^{2n+1}}$ is univalent on $U_{\tfrac{1}{2}} \supset U_{\tfrac{1}{10}} \supset U_{S_{2n}} $. Again by (3) applied to the function $f_{N_{n+1}+1}=Id$, \emph{\ref{ind:iii}} for $n+1$ and \eqref{S2nplus1def}, we see
\begin{align*}
{\bf Q^{2n+1}}(U_{S_{2n}})\subset U_{S_{2n}+\frac{\epsilon_{n+1}}{3}}=U_{S_{2n+1}}.
\end{align*}
By \emph{\ref{ind:iii}} for $n+1$, we have $U_{S_{2n+1}} \subset U_{\tfrac{1}{10}}$ and, together with \emph{\ref{ind:vi}} for $n$, this verifies \emph{\ref{ind:vi}} for $n+1$.

Now let $w \in U_{R_{2n}}$.  Using the same branch of $({\bf Q^{2n}})^{-1}$ which fixes $0$ as in the first part of \emph{\ref{ind:v}} for $n+1$, by \eqref{PhaseIIContainment}, $({\bf Q^{2n}})^{-1}(w) = \zeta$ for some $\zeta \in U_{R_{2n-1}-\delta(\epsilon_n)}$ and thus $w={\bf Q^{2n}}(\zeta)$. Then, by \eqref{PhaseIIconclusionn+1}, 
\begin{align}
\label{inversetrick1}
\rho_U(({\bf Q^{2n}}\circ {\bf Q^{2n-1}}\circ \cdots \circ {\bf Q^1})^{-1}(w),w) \hspace{-6cm} & \nonumber \\
&= \rho_U(({\bf Q^{2n-1}}\circ \cdots \circ {\bf Q^1})^{-1}\circ({\bf Q^{2n}})^{-1}(w),w)   \nonumber \\
&=\rho_U(({\bf Q^{2n-1}}\circ \cdots \circ {\bf Q^1})^{-1}(\zeta),{\bf Q^{2n}}(\zeta))  \hspace{0cm}  \nonumber \\
&<\frac{\epsilon_{n+1}}{3C}. 
\end{align}
By \eqref{PhaseIContainment}, if now $z \in U_{R_{2n+1}} \subset {\bf Q^{2n+1}}(U_{R_{2n}})$, then, using the same inverse branch as in \emph{\ref{ind:v}} which fixes $0$, $({\bf Q^{2n+1}})^{-1}(z) = w$ for some $w\in U_{R_{2n}}$ and thus  $z={\bf Q^{2n+1}}(w)$. By (3) above and  \emph{\ref{ind:iii}} for $n+1$, since $f_{N_{n+1}+1}= Id$,  
\begin{align}
\label{inversetrick2}
\rho_U(({\bf Q^{2n+1}})^{-1}(z),z)& = \rho_U(w, {\bf Q^{2n+1}}(w)) \nonumber \\
& <\frac{\epsilon_{n+1}}{3}.
\end{align}

Then, if we now take $z \in U_{R_{2n+1}}$ and we let $w=({\bf Q^{2n+1}})^{-1}(z)\in U_{R_{2n}}$ again as above using the branch of $({\bf Q^{2n+1}})^{-1}$ which fixes $0$, using \eqref{inversetrick1}, \eqref{inversetrick2} we have 
\begin{align}
\label{xiiin+1}
\rho_U(({\bf Q^{2n+1}}\circ {\bf Q^{2n}}\circ \cdots \circ {\bf Q^1})^{-1}(z),z) \hspace{6.8cm}& \nonumber \\ 
=\rho_U(({\bf Q^{2n}}\circ \cdots \circ {\bf Q^1})^{-1}\circ({\bf Q^{2n+1}})^{-1}(z),z)  \hspace{3.1cm}& \nonumber\\
\leq \rho_U(({\bf Q^{2n}}\circ \cdots \circ {\bf Q^1})^{-1}\circ ({\bf Q^{2n+1}})^{-1}(z),({\bf Q^{2n+1}})^{-1}(z)) \hspace{1.1cm}&\nonumber\\ 
+\rho_U(({\bf Q^{2n+1}})^{-1}(z),z) \hspace{5.7cm} & \nonumber \\
=\rho_U(({\bf Q^{2n}}\circ {\bf Q^{2n-1}}\circ \cdots \circ {\bf Q^1})^{-1}(w),w)+\rho_U(({\bf Q^{2n+1}})^{-1}(z),z)  \hspace{-0.1cm} & \nonumber \\
<\frac{\epsilon_{n+1}}{3C}+\frac{\epsilon_{n+1}}{3}<\epsilon_{n+1}. \hspace{6cm}& 
\end{align}
This verifies \emph{\ref{ind:viii}}.

Now by \emph{\ref{ind:iii}}, \emph{\ref{ind:vi}} for $n$ together with \eqref{R2ndef} we have 
\begin{align}
\label{AbletodoPhaseII}
({\bf Q^{2n-1}}\circ \cdots \cdots \circ {\bf Q^1})(U_{\tfrac{1}{20}}) \subset U_{S_{2n-1}} \subset U_{\tfrac{1}{10}} \subset U_{R_{2n}} \subset U_{R_{2n-1}- \delta(\epsilon_n)}. 
\end{align} 
while again by \emph{\ref{ind:vi}} for $n$ the forward composition ${\bf Q^{2n-1}} \circ \cdots \circ {\bf Q^1}$ is univalent on $U_{\tfrac{1}{20}}$. Lastly, applying \emph{\ref{ind:v}} for $n$, we see that the branch of $({\bf Q^{2n-1}} \circ \cdots \circ {\bf Q^1})^{-1}$ which fixes $0$ is well-defined and univalent on $U_{R_{2n-1}} \supset \overline U_{R_{2n} - \delta(\epsilon_n)}$. Combining these three observations, we have the cancellation property 
\begin{align}
\label{Cancellation}
({\bf Q^{2n-1}} \circ \cdots \cdots \circ {\bf Q^1})^{-1}\circ ({\bf Q^{2n-1}} \circ \cdots \cdots \circ {\bf Q^1}) = Id \quad \mbox{on} \quad U_{\tfrac{1}{20}}.
\end{align} 

Let $z \in U_{\frac{1}{20}}$ and set $\zeta =   ({\bf Q^{2n-1}} \circ \cdots \circ {\bf Q^1})(z)$. Then from \eqref{Cancellation} above we have $({\bf Q^{2n-1}}\circ \cdots \circ {\bf Q^1})^{-1}(\zeta)=z$ while by \eqref{AbletodoPhaseII} $\zeta \in U_{R_{2n-1}-\delta(\epsilon_n)}$. We then calculate, using \eqref{PhaseIIconclusionn+1},
\begin{align}
\label{eqctysetup}
\rho_U({\bf Q^{2n}}\circ \cdots \circ {\bf Q^1}(z),z)&=\rho_U({\bf Q^{2n}}(\zeta), ({\bf Q^{2n-1}}\circ \cdots \circ {\bf Q^1})^{-1}(\zeta)) \nonumber \\
&<\frac{\epsilon_{n+1}}{3C}.
\end{align}

Now let $f \in \calS$ be arbitrary. Let $f_i  \in \calS$ be an element of the $\frac{\epsilon_{n+1}}{3}$-net which approximates $f$ to within $\frac{\epsilon_{n+1}}{3}$ on $U_{\frac{1}{2}}\supset U_{\frac{1}{20}}$.  Let ${\bf Q^{2n+1}_{iM_{n+1}}}$ be a partial composition of ${\bf Q^{2n+1}}$ which approximates $f_i$ to within $\frac{\epsilon_{n+1}}{3}$ also on $U_{\frac{1}{2}}\supset U_{\frac{1}{20}}$ using (3) above and let $m = iM_{n+1}$ so that ${\bf Q^{2n+1}_m}= {\bf Q^{2n+1}_{iM_{n+1}}}$. 

Applying \emph{\ref{ind:vi}} for $n+1$ gives us that ${\bf Q^{2n}} \circ \cdots \circ {\bf Q^1}(z) \in U_{\tfrac{1}{10}} \subset U_{\tfrac{1}{4}}$. Then, using the hyperbolic convexity of $U_{\tfrac{1}{4}}$ (which follows from Lemma \ref{stupidfuckinglemma}), the hyperbolic M-L estimates (Lemma \ref{hyperbolicML}), \eqref{eqctysetup}, (3), (4), and the fact that $f_i$ approximates $f$, we have
\begin{align*}
&\rho_U({\bf Q^{2n+1}_m}\circ {\bf Q^{2n}} \circ \cdots \cdots  \circ {\bf Q^1}(z),f(z)) \\
&\leq \rho_U({\bf Q^{2n+1}_m}\circ {\bf Q^{2n}} \circ \cdots \cdots  \circ {\bf Q^1}(z),{\bf Q^{2n+1}_m}(z))\\ 
&\hspace{1cm}+\rho_U({\bf Q^{2n+1}_m}(z),f_i(z))+\rho_U(f_i(z),f(z)) \\
&\leq C\cdot \frac{\epsilon_{n+1}}{3C}+\frac{\epsilon_{n+1}}{3}+\frac{\epsilon_{n+1}}{3}\\
&= \epsilon_{n+1}
\end{align*}
which verifies \emph{\ref{ind:ix}}.  Note that the first term uses Lemmas \ref{stupidfuckinglemma}, \ref{hyperbolicML}, \eqref{eqctysetup}, and (4), the second uses (3), and the third uses the net approximation.  This completes the proof of the claim from which Lemma \ref{mediuminductionlemma} follows. \end{claimproof} 
\end{proof}

We are now finally in a position to prove the main result of this paper. 


\begin{proof}[Proof of Theorem \ref{thetheorem}]
Let $f \in \calS$ be arbitrary. Let $\{P_m \}_{m=1}^{\infty}$ be the sequence of quadratic polynomials which exists in view of Lemma \ref{smallinductionlemma} and which is bounded by part \emph{(1)} of the statement. By Proposition \ref{VerySimpleProposition} and part \emph{(2)} of the statement, $U_{\frac{1}{20}}$ is contained in a bounded Fatou component $V$ for this sequence. By part \emph{(3)} of the statement, there exists a subsequence $\{Q_{m_k} \}_{k=1}^{\infty}$ of $\{Q_m \}_{m=1}^{\infty}$ such that the sequence of compositions $\{Q_{m_k}\}_{k=1}^{\infty}$ converges locally uniformly to $f$ on $U_{\frac{1}{20}}$.  Since $\{Q_{m_k}\}_{k=1}^{\infty}$ is normal on $V$, we may pass to a further subsequence, if necessary, to ensure this subsequence of iterates will converge locally uniformly on all of $V$.  By the identity principle, the limit must then be $f$. In fact, since every such convergent subsequence must have limit $f$, it follows readily that $\{Q_{m_k}\}_{k=1}^{\infty}$ converges locally uniformly to $f$ on all of $V$. 
\end{proof}

Finally, we arrive at the last result of this paper. 

{\it Proof of Theorem \ref{abitmore}}: Let $r > 0$ be such that ${\mathrm D}(z_0, r) \subset \Omega$. Then the function 
\begin{align}
\label{gdefn}g(w) = \frac{f(rw +z_0) - f(z_0)}{rf'(z_0)}, \qquad w \in \D
\end{align}
belongs to ${\mathcal S}$ while $f$ can clearly be recovered from $g$ using the formula
\begin{align}
\label{frecover}
f(z) = rf'(z_0) g\left( \tfrac{z-z_0}{r} \right ) + f(z_0), \qquad z \in {\mathrm D}(z_0, r).
\end{align}

Since ${\mathcal N}$ is locally bounded and all limit functions are non-constant, using Hurwitz's Theorem e.g. Theorem VII.2.5  in \cite{Con} and also Corollary IV.5.9 from the same source we can find $K \ge 1$ such that, for all $f \in {\mathcal N}$, we have 
\begin{align}
\label{Kbounds}
\frac{1}{K}   \le   |f'(z_0)| \le K,  \qquad |f(z_0)| \le K.
\end{align}

Then, if we let $X$ be the subset of $\C^2$ given by $X = \{(f'(z_0), f(z_0)), f \in {\mathcal N}\}$, we can clearly pick a sequence $\{(\alpha_n, \beta_n)\}_{n=1}^\infty$ which densely approximates all of $X$ and such that, for all $n$,
\begin{align*}
\frac{1}{2K}   \le   |\alpha_n| \le 2K,  \qquad |\beta_n| \le 2K.
\end{align*}

We next wish to apply a suitable affine conjugacy to the polynomial sequence $\Pm$ of Theorem \ref{thetheorem} to construct the sequence $\Pmt$ needed to prove the current result. To this end, define $\phi_0(w) = rw + z_0$, and $\phi_n(w) = r\alpha_n w + \beta_n$ for $n \ge 1$. 
Recall the compositions $\{{\bf Q^i} \}_{i=1}^{\infty}$ from Lemma \ref{mediuminductionlemma} and that each ${\bf Q^i}$ was a $(17 + \kappa_0)$-bounded composition of $J_i$ quadratic polynomials. 

As we did before the proof of Lemma \ref{smallinductionlemma}, for $i=0$, set $T_0 = 0$ and, for each $i \ge 1$, set $T_i = \sum_{j=1}^i J_j$. Recall that these compositions $\{{\bf Q^i} \}_{i=1}^{\infty}$ then gave rise the the polynomial sequence $\Pm$ of Lemma \ref{smallinductionlemma} and ultimately Theorem \ref{thetheorem}. 

For $m=1$ we define $\tilde P_1 = \phi_1 \circ P_1 \circ \phi_0^{-1}$. For $m > 1$ let $i \ge 1$ be the largest index such that $T_{i-1} < m$. For $i=2k$ even, we define $\tilde P_m$ by 
\begin{align}
\label{Pmeven}
 \tilde P_m= \left\{
\begin{array}{ll}
     \phi_{k+1}\circ P_m \circ \phi_k^{-1}, & m =T_{i-1} + 1, \\
     \phi_{k+1} \circ P_m \circ {\phi_{k+1}}^{-1}, & T_{i-1} + 1 < m \le T_i \end{array} 
\right. 
\end{align}
while for $i = 2k+1$ odd we set 
\begin{align}
\label{Pmodd}
\tilde P_m =  \phi_{k+1} \circ P_m \circ {\phi_{k+1}}^{-1}, \quad T_{i-1} + 1 \le m \le T_i.
\end{align}
Then, (whether $i$ is even or odd), if as usual we let  $Q_m = P_m \circ \cdots \cdots  \circ P_2 \circ P_1$ and $\tilde Q_m = \tilde P_m \circ \cdots \cdots \circ \tilde P_2 \circ \tilde P_1$, then
\begin{align}
\label{Qtilde}
\tilde Q_m =  \phi_{k+1}\circ Q_m \circ \phi_0^{-1}.
\end{align}

Recall the Fatou component $V \supset U_{\frac{1}{20}} \ni 0$ from the proof of Theorem \ref{thetheorem}. Since the family $\{\phi_n\}_{n=0}^\infty$ is bi-equicontinuous in the sense that the family $\{\phi_n\}_{n=0}^\infty$ as well as the family of inverses $\{\phi_n^{-1}\}_{n=0}^\infty$ are both equicontinuous and locally bounded on $\C$, it follows from \cite{Com3} Proposition 2.1 that 
$W = \phi_0(V)$ is a bounded Fatou component for the sequence $\Pmt$ which contains $\phi_0(U_{\frac{1}{20}})$. 

Let $\epsilon > 0$. It follows from applying the local equivalence of the Euclidean and hyperbolic metrics from Lemma \ref{lemma4.3} to {\bf a)}, {\bf b)} and part \emph{(4)} of Lemma \ref{mediuminductionlemma} there exists $j_0$ such that, for each $j \ge j_0$, there exists $\tilde m_j$, $1 \le \tilde m_j \le J_{2j+1}$ such that for $w \in U_{\frac{1}{20}}$, we have 
\begin{align}
\label{gapprox}
|{\bf Q_{\tilde  m_j}^{2j+1}}\circ {\bf Q^{2j}}\circ \cdots \circ {\bf Q^1}(w) - g(w)| < \frac{\epsilon}{2Kr}
\end{align}
where as before ${\bf Q_{\tilde  m_j}^{2j+1}}$ denotes the partial composition of the first $\tilde m_j$ quadratics of ${\bf Q^{2j+1}}$.

Next, using the approximation property of the sequence $\{(\alpha_n \beta_n)\}_{n=1}^\infty$ to all of the set $X$ above, we can find a subsequence $\{(\alpha_{n_k}, \beta_{n_k})\}_{j=1}^\infty$ which converges to $(f'(z_0), f(z_0))$. Hence we can find $k_0$ such that for all $k \ge k_0$, if $|w| \le \tfrac{1}{288} \le \frac{2}{\kappa_0}$, we have 
\begin{align}
\label{phiapprox}
|\phi_{n_k}(w) - (rf'(z_0)w + f(z_0))| < \frac{\epsilon}{2}.
\end{align}

Now let $z \in \phi_0(U_{\frac{1}{20}})$ be arbitrary, and let $k_0$ be sufficiently large so that $n_{k_0} \ge  j_0$. Then, for each $k \ge k_0$ if we let $i = 2n_k +1$ so that $i = 2j +1$ where $j = n_k$, so that by \eqref{Qtilde} and the construction of the sequence $\Pm$ from just before Lemma \ref{smallinductionlemma} 
\begin{align*}
\tilde Q_{T_{2n_k} + \tilde  m_{n_k}} = \phi_{n_k + 1}  \circ Q_{T_{2n_k} + \tilde  m_{n_k}} \circ \phi_0^{-1} =  \phi_{n_k + 1} \circ {\bf Q_{\tilde  m_{n_k}}^{2n_k+1}}\circ {\bf Q^{2n_k}}\circ \cdots \circ {\bf Q^1} \circ  \phi_0^{-1}
\end{align*}
and, using \eqref{frecover} and \eqref{Qtilde},
\begin{align*}
|\tilde Q_{T_{2n_k} + \tilde  m_{n_k}}(z) - f(z)|\\ 
&\hspace{-4.25cm} = |\phi_{n_k + 1} \circ Q_{T_{2n_k} + \tilde  m_{n_k}} \circ \phi_0^{-1}(z) - (rf'(z_0)\circ g \circ \phi_0^{-1}(z) + f(z_0))|\\
&\hspace{-4.25cm} \le  |\phi_{n_k+1} \circ Q_{T_{2n_k} + \tilde  m_{n_k}} \circ \phi_0^{-1}(z) -  (rf'(z_0)\circ Q_{T_{2n_k} + \tilde  m_{n_k}} \circ \phi_0^{-1}(z) + f(z_0))|\\
&\hspace{-4cm} +  | (rf'(z_0)\circ Q_{T_{2n_k} + \tilde  m_{n_k}} \circ \phi_0^{-1}(z) + f(z_0)) - (rf'(z_0)\circ g \circ \phi_0^{-1}(z) + f(z_0))|.
\end{align*}

Recall that we chose $\kappa_0 \ge 576$ in Lemma \ref{mediuminductionlemma}. From this it follows that $Q_{T_{2n_k} + \tilde  m_{n_k}} \circ \phi_0^{-1}(z) \in Q_{T_{2n_k} + \tilde  m_{n_k}}(U_{\frac{1}{20}}) \subset Q_{T_{2n_k} + \tilde  m_{n_k}}(V)
\subset {\mathrm D}(0, \tfrac{1}{288})$ so that the first term on the right hand side of the above is less than $\tfrac{\epsilon}{2}$ in view of \eqref{phiapprox} above. In addition, it follows from \eqref{gapprox} that the second term is bounded above by $r|f'(z_0)|\tfrac{\epsilon}{2Kr} \le \tfrac{\epsilon}{2}$ in view of \eqref{Kbounds} above. Thus, if for $k \ge 1$ we set $m_k : = T_{2n_k} + \tilde  m_{n_k}$, then for $k \ge k_0$ we have 
\begin{align*}
|\tilde Q_{m_k}(z) - f(z)| < \epsilon
\end{align*} and, since $\epsilon >0$ was arbitrary, $\{\tilde  Q_{m_k}\}_{k=1}^\infty$ converges uniformly to $f$ on $\phi_0(U_{\frac{1}{20}})$. The same argument using the identity principle as at the end of proof of Theorem \ref{thetheorem} above shows that $\{\tilde  Q_{m_k}\}_{k=1}^\infty$ converges locally uniformly on $W$ to $f$ and, as $f \in {\mathcal N}$ was arbitrary, this completes the argument.  \hfill $\square$\\

%% file: Appendices.tex
\chapter{Appendices}

\section{Glossary of Symbols}

We will be using many different symbols repeatedly throughout this exposition.  For clarity of exposition, we have gathered them into the following table.

\begin{center}
\begin{tabular}{ || c | m{8.2cm}|m{1.5cm} || } 
 \hline 
 {\bf Symbol} & {{\bf Description}}&{\bf Defined on Page}  \\ 
 \hline \hline 
 
 $\lambda$ & $\lambda=e^{2\pi i \frac{\sqrt{5}-1}{2}}$: the irrational multiplier for the fixed point $0$ in the Siegel disc polynomial &\pageref{Plambda}, \pageref{RestrictionforPIL}, \pageref{overviewpi1}, \pageref{overviewpii1}, \pageref{PhaseIIKappa}\\ 
 \hline
 $P_\lambda$ & $P_\lambda (z)=\lambda z(1-z)$: the unscaled Siegel disc polynomial &\pageref{Plambda}, \pageref{RestrictionforPIL}, \pageref{overviewpi1}, \pageref{overviewpi2}, \pageref{overviewpii1}, \pageref{PhaseIIKappa}\\ 
 \hline 
 $\calK_\lambda$ & The filled Julia set for $P_\lambda$  &\pageref{RestrictionforPIL}\\ 
 \hline 
 $U_\lambda$ & The Siegel disc for $P_\lambda$ &\pageref{Plambda}, \pageref{RestrictionforPIL}, \pageref{overviewpi2}\\ 
 \hline 
 ${\tilde U_R}$ & The disc of hyperbolic radius $R>0$ about $0$ in $U$ & \pageref{overviewpi2}, \pageref{overviewpii2}\\ 
 \hline 
 ${\tilde r_0}$ & ${\tilde r_0}:=\mathrm{dist}(\partial {\tilde U_R}, \partial U_\lambda)$: the Euclidean distance between $\partial{U_\lambda}$ and $\partial {\tilde U}$&\pageref{tilder0}\\ 
 \hline 
 $\psi_\lambda$ & The unique Riemann map from $U_\lambda$ to $\D$ satisfying $\psi_\lambda(0)=0$ and $\psi_\lambda'(0)>0$ & \pageref{overviewpi2}\\ 
 \hline 
 $\kappa$ & Scaling factor & \pageref{RestrictionforPIL}, \pageref{tilder0}, \pageref{fcircghdestimates}, \pageref{PhaseI}, \pageref{PhaseIIKappa}\\ 
 \hline
 $P$ & $P(z)=\frac{1}{\kappa}{P_\lambda}(\kappa z)=\lambda z(1- \kappa z)$: the scaled version of $P_\lambda$ &\pageref{overview1}, \pageref{RestrictionforPIL}, \pageref{overviewpi2}\\ 
 \hline
 $\calK$ & The filled Julia set for $P$ &\pageref{overview2}, \pageref{RestrictionforPIL}, \pageref{overviewpi2}\\ 
 \hline 
 $U$ & The Siegel disc for $P$ &\pageref{overview1}, \pageref{RestrictionforPIL}, \pageref{overviewpi2} \\ 
 \hline 
 $U_R$ & The disc of hyperbolic radius $R>0$ about $0$ in $U$ &\pageref{overviewpi2} \\ 
 \hline 
 $r_0$ & $r_0:=\mathrm{dist}(\partial U_R, \partial U)$: the Euclidean distance between $\partial U_R$ and $\partial U$& \pageref{r0}\\ 
 \hline
 $\psi$ & The unique Riemann map from $U$ to $\D$ satisfying $\psi(0)=0$ and $\psi'(0)>0$ & \pageref{overviewpi2}, \pageref{phi2h} \\ 
 \hline
 $G$ & $G(z)$: the Green's function for $P$ &\pageref{psi0N-1est}, \pageref{overviewpii1}, \pageref{PhaseIIKappa}\\ 
 \hline 
 $V_{h}$ & The Green's domain $\{z\in \C \: : \: G(z)<h \}$ (for $h>0$)& \pageref{PhaseIIKappa}\\ 
 \hline 
 $V_{2h}$ & The Green's domain $\{z\in \C \: : \: G(z)<2h \}$ (for $h>0$) & \pageref{PhaseIIKappa}, \pageref{overviewpii2}\\ 
 \hline 
 ${\tilde R}$ & ${\tilde R}:=R_{(V_{2h},0)}^{int}U_R$: the internal hyperbolic radius of $U_R$ in $V_{2h}$ about $0$ & \pageref{overviewpii2}\\ 
 \hline 
 ${\tilde V_{2h}}$ & ${\tilde R}:={\tilde V_{2h}}:=\Delta_{V_{2h}}(0,{\tilde R})$: the hyperbolic disc of radius of ${\tilde R}$ in $V_{2h}$ about $0$& \pageref{overviewpii2}\\ 
 \hline 
 $\phi_{2h}$ & The unique Riemann map from ${\tilde V_{2h}}$ to $V_{2h}$ satisfying $\phi_{2h}(0)=0$, $\psi_{2h}'(0)>0$. & \pageref{phi2h}, \pageref{defofh3} \\ 
 \hline 
 $\psi_{2h}$ & The unique Riemann map from $V_{2h}$ to $\D$ satisfying $\psi_{2h}(0)=0$, $\psi_{2h}'(0)>0$. & \pageref{phi2h}, \pageref{defofh3} \\ 
 \hline 
 $R'$ & $R':=R_{(U,0)}^{int}{\tilde V_{2h}}$: the internal hyperbolic radius of ${\tilde V_{2h}}$ in $U$ about $0$ &\pageref{Rdash},  \pageref{defofh3}, \pageref{up}\\ 
 \hline
 ${\tilde U}$ & ${\tilde U}:=\phi_{2h}^{-1}(U)$: the inverse image of $U$ under $\phi_{2h}$ &\pageref{intsiegeldisc}\\ 
 \hline
 $R''$ & $R'':=R_{(U,0)}^{int}{\tilde U}$: the internal hyperbolic radius of ${\tilde U}$ in $U$ about $0$ &\pageref{Rdoubledash}, \pageref{defofh3}, \pageref{up}\\ 
 \hline
\end{tabular}
\end{center}

\section{Dependency Tables}

The proofs of the three key steps in this paper, namely the Polynomial Implementation Lemma (Lemma \ref{PIL}), Phase I (Lemma \ref{PhaseI}), and Phase II (Lemma \ref{PhaseII}) involve many quantities and functions which are defined in terms of other quantities introduced earlier (and occasionally later) in the proofs of these results. In order to fully understand these quantities and avoid any danger of circular reasoning, we feel it is therefore important if not indispensable that we provide full tables for all three of these results detailing the dependencies of the most important objects in their statements and proofs. 

The objects in each table are for the most part listed in the order in which they appear in the proof of the corresponding result as well as the statements and proofs of the supporting lemmas which lead up to it. The tables for Polynomial Implementation Lemma (Lemma \ref{PIL}) and Phase I each have five columns. To determine the dependencies for a given object (given as ultimate dependency in the third column), one looks at the immediate dependencies (second column) for that row. One then reads off the dependencies for each object in this column from the column entries for ultimate dependencies (third column) for the earlier lines in the table for these objects and the combined list of these dependencies for every object then forms the new entry in the third column. 

Due to the more complicated nature of the proof, the table for Phase II has an extra column for intermediate dependencies. However, the determination of the ultimate dependencies is done similarly to before where instead one looks at the ultimate dependencies (fourth column) for each quantity in the second and third columns for the row containing a given object and the combined list gives the entry for the ultimate dependencies for that object (which is in the fourth column). One exception to this is where, for objects depending on the constants $K_2$, $K_3$, one needs to look at later entries in the table for these constants (as explained in the proof of Lemma \ref{PhaseII}, there is no danger of circular reasoning here). For intermediate dependencies, when listed these are the same as ultimate dependencies but involve extra quantities which are then eliminated by monotonicity or some uniformization procedure such as taking a maximum (e.g. $\eta_1$, $\eta_2$) or by being determined later as is the case with $K_2$, $K_3$. Another exception is when some of the ultimate dependencies appear to be `missing'. Most of these are instances of where the scaling factor $\kappa$ is omitted because an estimate involving the hyperbolic metric of $U$ which will not depend on $\kappa$. However, there is one exception to this which is the first instance of the quantity $N$ which appears at \eqref{Nineq} in the text where a monotonicity argument allows us to dispense with the dependence on the quantities $R$ and $R_0$.

One final remark concerns those objects which appear in the columns for ultimate dependencies. As a matter of logical necessity these fall into two categories - objects which are defined or whose value is determined before the proof (e.g. $\kappa$, $R_0$) and universally quantified objects appearing in the statement (e.g. $\epsilon_1$, $\epsilon_2$). Both of these are indicated in the tables by their appearance in the first column being in green. However, these objects also generally come with bounds which have dependencies of their own which then forces these objects to inherit these dependencies. A good way to think of this is to view these results of determining these ultimate dependencies as methods or routines within a computer programme where the objects in the ultimate dependencies are then free objects which are either defined outside the method and then passed to it as parameters or alternatively set within the method itself.


\begin{landscape}
	
	\begin{center}
	{\large \bf Dependencies Table for the Polynomial Implementation Lemma (Lemma \ref{PIL})}
	\vspace{.3cm}
		
		\begin{longtable}{ | p{1.20in} | p{1.00in} | p{1.10in} | p{3in} | p{0.67in} | }
			\hline
			\textbf{Quantity/} & \textbf{Immediate} & \textbf{Ultimate} & \textbf{Role} & \textbf{Defined}   \\
			\textbf{Function} & \textbf{Dependency} & \textbf{Dependency} & &  \textbf{on Page}   \\ \hline \hline

			 $\color{darkgreen}\gamma$, $\color{darkgreen}\Gamma$ &  &  &  Boundary curves of Jordan domains $\Omega$, $\Omega'$ & \pageref{IntrotoPIL}, \pageref{RestrictionforPIL}  \\ \hline
			$\color{darkgreen} f$ & $\gamma$ &   $\gamma$  &  Function defined and univalent on a neighbourhood of $\overline \Omega$ &  \pageref{IntrotoPIL}, \pageref{RestrictionforPIL}  \\ \hline
                          $F$ &   $\gamma$, $\Gamma$, $f$ &  $\gamma$, $\Gamma$, $f$ &  Quasiconformal interpolation between $f$ and $Id$ &  \pageref{IntrotoPIL}, \pageref{qcextension},  \pageref{RestrictionforPIL}  \\ \hline
                          $\color{darkgreen} \kappa$ & &   &  Scaling factor for $P_\lambda$ & \pageref{RestrictionforPIL}  \\ \hline
                          $\color{darkgreen}\gamma$, $\color{darkgreen}\Gamma$ {\color{darkgreen}(scale fixed)} &  &  &  Boundary curves of Jordan domains $\Omega$, $\Omega'$ & \pageref{IntrotoPIL}, \pageref{RestrictionforPIL}  \\ \hline
			$\color{darkgreen} f$ {\color{darkgreen}(scale fixed)} & $\kappa$, $\gamma$ &   $\kappa$, $\gamma$  &  Function defined and univalent on a neighbourhood of $\overline \Omega$ &  \pageref{IntrotoPIL}, \pageref{RestrictionforPIL}  \\ \hline
                          $F\,$ (scale fixed) &   $\kappa$, $\gamma$, $\Gamma$, $f$ &  $\kappa$,  $\gamma$, $\Gamma$, $f$ &  Quasiconformal interpolation between $f$ and $Id$ &  \pageref{IntrotoPIL}, \pageref{RestrictionforPIL}  \\ \hline
                             $\color{darkgreen}N$ & &   &  Iterative time at which interpolation is constructed & \pageref{PmnDef} \\ \hline
			 $\{\psi_m^N\}_{m=0}^\N$ &  $\kappa$,  $\gamma$, $\Gamma$, $f$, $N$  & $\kappa$, $\gamma$, $\Gamma$, $f$, $N$   &Quasiconformal changes of coordinates which fix $0$ & \pageref{PmnDef}    \\ \hline
			 $\{P_m^N\}_{m=0}^\N$ & $\kappa$, $\gamma$, $\Gamma$, $f$, $N$  & $\kappa$, $\gamma$, $\Gamma$, $f$, $N$ & Conjugated version of $P$ which fixes $0$ & \pageref{PmnDef}    \\ \hline

		\end{longtable}
		
	\end{center}		
	
\end{landscape}

\begin{landscape}
	
         \begin{center}
		
		\begin{longtable}{ | p{1.20in} | p{1.00in} | p{1.10in} | p{3in} | p{0.67in} | }
			\hline
			\textbf{Quantity/} & \textbf{Immediate} & \textbf{Ultimate} & \textbf{Role} & \textbf{Defined}   \\
			\textbf{Function} & \textbf{Dependency} & \textbf{Dependency} & &  \textbf{on Page}   \\ \hline \hline

			  $\color{darkgreen} A$ & $\kappa$, $f$ & $\kappa$, $f$  &  Relatively compact set on which polynomial approximation is constructed. & \pageref{PIL}  \\ \hline
			    $\color{darkgreen} \delta$ & $f$ & $f$  &  Enlargement of hyperbolic neighbourhood & \pageref{PIL}  \\ \hline
			    $\color{darkgreen} M$ & $f$ & $f$ &  Bound on hyperbolic derivative & \pageref{PIL}  \\ \hline
			    $\color{darkgreen} \epsilon$ & &   &  Error bound of polynomial approximation & \pageref{PIL}  \\ \hline
			  $k_0$ &  $\kappa$, $\gamma$, $\Gamma$, $f$, $A$, $\delta$, $M$, $\epsilon$, $\{(\psi_0^{n_k})^{-1}\}_{k=1}^\infty$ & $\kappa$, $\gamma$, $\Gamma$, $f$, $A$, $\delta$, $M$, $\epsilon$ & Minimum number of compositions of $P$ required for desired approximations & \pageref{PIL}    \\ \hline
			 $\color{darkgreen}k_1$ & $k_0$ & $\kappa$, $\gamma$, $\Gamma$, $f$, $A$, $\delta$, $M$, $\epsilon$ & Length of finite polynomial sequence & \pageref{PIL}    \\ \hline			
			  $\{P_m^{n_{k_1}}\}_{m=1}^{n_{k_1}}$ & $\kappa$, $\gamma$, $\Gamma$, $f$, $A$, $\delta$, $M$, $\epsilon$, $k_0$, $k_1$, $\{\psi_m^{n_{k_1}}\}_{m=0}^{n_{k_1}}$  & $\kappa$, $\gamma$, $\Gamma$, $f$, $A$, $\delta$, $M$, $\epsilon$, $k_1$ & Finite polynomial sequence possessing desired approximation properties & \pageref{PIL}    \\ \hline
			
		\end{longtable}
		
	\end{center}		
	
\end{landscape}


\begin{landscape}

	\begin{center}
	{\large \bf Dependencies Table for Phase I (Lemma \ref{PhaseI})}
	\end{center}
	\vspace{.3cm}
		\begin{longtable}{ | p{1.30in} | p{0.95in} | p{0.95in} | p{2.75in} | p{0.67in} | }
			\hline
			\textbf{Quantity/} & \textbf{Immediate} & \textbf{Ultimate} & \textbf{Role} & \textbf{Defined}   \\
			\textbf{Function} & \textbf{Dependency} & \textbf{Dependency} & & \textbf{on Page} \\ \hline \hline

			{\color{darkgreen}$\kappa$ {\color{darkgreen}(first version)}} &  &   &  Scaling factor for $P_\lambda$ & \pageref{tilder0}  \\ \hline
			{\color{darkgreen}$R_0$} &    &   & Radius of hyperbolic disc about $0$ in $U$ & \pageref{localdistortion}, \pageref{PhaseI}    \\ \hline
			$\tilde r_0$ {\color{darkgreen}(function of $R$)} &   &   & $\mathrm{d}(\partial \tilde U_{R}, \partial \tilde U$) &  \pageref{tilder0}    \\ \hline
			$r_0$ {\color{darkgreen}(function of $R$)} & $\kappa$    &  $\kappa$   & $\mathrm{d}(\partial U_{R}, \partial U)$ &  \pageref{r0}    \\ \hline
			$\kappa_0$ &   $R_0$ &  $R_0$ &  Minimum necessary scaling to control distortion & \pageref{fcircghdestimates}, \pageref{kappa0},  \pageref{PhaseI} \\ \hline
			$\color{darkgreen}\kappa$ {\color{darkgreen}(redefined)} & $\kappa_0$ & $R_0$ & $\kappa \ge \kappa_0$ &  \pageref{fcircghdestimates}, \pageref{kappa0},  \pageref{PhaseI}  \\ \hline
			{\color{darkgreen}$\epsilon$ {\color{darkgreen}(initial version)}} &  &  & Upper bound on error of approximation &  \pageref{PhaseI}    \\ \hline
			{\color{darkgreen}$N+1$} &  &  & Number of functions being approximated  & \pageref{PhaseI}    \\ \hline
			{\color{darkgreen}$\{f_i\}_{i=1}^{N+1}$} &  &  & The finite sequence of  functions being approximated &  \pageref{PhaseI}    \\ \hline
			{\color{darkgreen}$\epsilon$ {\color{darkgreen}(first redefinition)}} & $R_0$ & $R_0$ & $\epsilon < R_0$ &  \pageref{epsilonbound1}    \\ \hline
		       $M_N$ & $\kappa$, $R_0$, $N$, $\{f_i\}_{i=1}^{N+1}$, $\epsilon$ & $\kappa$, $R_0$, $N$, $\{f_i\}_{i=1}^{N+1}$, $\epsilon$ & Number of quadratics needed to approximate each $f_{i}\circ f_{i-1}^{-1}$ &  \pageref{PhaseI}    \\ \hline
		       $\{P_m\}_{m=1}^{(N+1)M_N}$ & $\kappa$, $R_0$, $N$, $\{f_i\}_{i=1}^{N+1}$, $\epsilon$ & $\kappa$, $R_0$, $N$, $\{f_i\}_{i=1}^{N+1}$, $\epsilon$   & Finite polynomial sequence with desired approximation properties &  \pageref{PhaseI}    \\ \hline
			$\sigma$ & $R_0$  & $R_0$  & Inf. of hyperbolic density $\sigma_U$ on $U_{4R_0}$ &  \pageref{sigmabound}    \\ \hline
			$\delta_0$ & $R_0$  & $R_0$  & $\mathrm{d}(\partial U_{R_0}, \partial  U_{\tfrac{3}{2}R_0}$) &  \pageref{distboundary}    \\ \hline
			{\color{darkgreen}$\epsilon$ {\color{darkgreen}(second redefinition)}} & $\kappa$, $R_0$ & $\kappa$, $R_0$ & Continuity estimate for $\sigma_U$ on $U_{4R_0}$ &  \pageref{epsilonbound2}    \\ \hline
		\end{longtable}

\end{landscape}


\begin{landscape}
	
	\begin{center}
	{\large \bf Dependencies Table for Phase II (Lemma \ref{PhaseII})}
	\vspace{.3cm}

		\begin{tabular}{ | p{1.2in} | p{0.95in} | p{1.00in} | p{0.95in} | p{1.95in} | p{0.7in} | }
			\hline
			\textbf{Quantity/} & \textbf{Immediate} & \textbf{Intermediate} & \textbf{Ultimate} & \textbf{Role} & \textbf{Defined}   \\
			\textbf{Function} & \textbf{Dependency} & \textbf{Dependency} & \textbf{Dependency} & & \textbf{on Page}   \\ \hline \hline
			
			$\color{darkgreen}\kappa$ &  &  &  &  Scaling factor for $P_\lambda$ &  \pageref{PhaseIIKappa}  \\ \hline
			$\color{darkgreen}r_0$ &  &  &  &  Lower bound for hyperbolic radius  &  \pageref{r0Def}, \pageref{R0Def}  \\ \hline
			$d_0$ &   $\kappa$, $r_0$ &  & $\kappa$, $r_0$ &  $d(0, \partial U_R)\geq d_0$ &  \pageref{dnaught}  \\ \hline
			$\color{darkgreen}h_0$ &  &  &  &  Upper bound for value of Green's function for $P$ &   \pageref{rhonaught},  \pageref{h0set}, \pageref{PhaseII}  \\ \hline
			$D_0$ &   $\kappa$, $h_0$ &  & $\kappa$, $h_0$ &  $\delta_{V_{2h}}(z)\leq D_0, \, z \in U, h \le h_0$ &  \pageref{rhonaught} \\ \hline				$\rho_0$ & $d_0$, $D_0$ &  & $h_0$, $r_0$ & $R_{(V_{2h},0)}({\tilde V_{2h}})\geq \rho_0$ &  \pageref{rhonaught}    \\ \hline
			
			${\check R}$ (function of $h$) &  $h_0$ &  & $h_0$ & ${\check R}(h)=R^{ext}_{(V_{2h},0)}V_h$  &  \pageref{rcheck}    \\ \hline
			$\color{darkgreen}R_0$ &  &  &  &  Upper bound for hyperbolic radius  &  \pageref{R0Def}  \\ \hline
			$\color{darkgreen}R$ & $r_0$, $R_0$ &  & $r_0$, $R_0$  &  Hyperbolic radius of disc on which $\mathcal E$ is defined (value is fixed at the start of `up') &  \pageref{R0Def}, \pageref{PhaseII}, \pageref{up}  \\ \hline
			$\color{darkgreen}\calE$ & $\kappa$, $R$, $r_0$, $R_0$ &  & $\kappa$, $R$, $r_0$, $R_0$  &  Error we wish to correct (fixed at the start of `during') &  \pageref{PhaseII}, \pageref{during}  \\ \hline
			${\tilde \epsilon_1}$ (first version) & $\rho_0$ &  & $h_0$, $r_0$ & $\tilde V_{2h} \setminus \hat N \neq \emptyset$  &   \pageref{FirstEpsilonOne}   \\ \hline
			$\color{darkgreen}\epsilon_1$ & ${\tilde \epsilon_1}$ &  &  $h_0$, $r_0$ &  $\epsilon_1 \in (0,{\tilde \epsilon_1}]$ &   \pageref{TargetLemma}, \pageref{FirstEpsilonOne}   \\ \hline

		\end{tabular}
		
	\end{center}
	
\end{landscape}

\begin{landscape}
	
	\begin{center}
		
		\begin{tabular}{  | p{1.35in} | p{1.02in} | p{1.0in} | p{0.95in} | p{1.80in} | p{0.75in} | }
			\hline
			\textbf{Quantity/} & \textbf{Immediate} & \textbf{Intermediate} & \textbf{Ultimate} & \textbf{Role} & \textbf{Defined}   \\
			\textbf{Function} & \textbf{Dependency} & \textbf{Dependency} & \textbf{Dependency} & & \textbf{on Page}   \\ \hline \hline

			${\tilde \epsilon_1}$ (first restriction) & $\tilde \epsilon_1$ (previous), $\kappa$, $d_0$ &  & $h_0$, $r_0$ & $\tilde \epsilon_1<\frac{\kappa d_0}{8D_1}$ &   \pageref{TargetLemma}, \pageref{FirstRestriction}   \\ \hline
			$\color{darkgreen} \epsilon_1$ {\color{darkgreen}(restricted)} & ${\tilde \epsilon_1}$ &  &  $h_0$, $r_0$ &  $\epsilon_1 \in (0,{\tilde \epsilon_1}]$ &   \pageref{TargetLemma},  \pageref{FirstRestriction}   \\ \hline	
			$T$ (function of $\epsilon_1$) & $\kappa$, $\tilde \epsilon_1$, $d_0$ &  &  $h_0$, $r_0$ & $R^{int}_{({\tilde V_{2h}},0)}({\tilde V_{2h}} \setminus {\hat N})\geq T(\epsilon_1)$, &  \pageref{TargetLemma}, \pageref{DefofT}    \\ \hline			 
			$\tilde \epsilon_1$ (second restriction) & $\tilde \epsilon_1$ (previous), $T$, $\check R$, $h_0$ & & $h_0$, $r_0$ & $\min_{\,0 < \epsilon_1 \le \tilde \epsilon_1} T(\epsilon_1) \ge $  &  \pageref{FittingLemma}, \pageref{ExistenceCriterion}\\ 
	                 &  &  &  &  $\min_{\,0 < h \le h_0}{\check R}(h)$ & \\ \hline
                          $\color{darkgreen} \epsilon_1$ {\color{darkgreen}(restricted)} & ${\tilde \epsilon_1}$  &  &  $h_0$, $r_0$ & $\epsilon_1 \in (0,{\tilde \epsilon_1}]$ &  \pageref{ExistenceCriterion}  \\ \hline
			$h$ (function of $\epsilon_1$) &  ${\tilde \epsilon_1}$, $T$, $\check R$, $h_0$ &  & $h_0$, $r_0$ & $\tilde V_{h(\epsilon_1)} \subset \tilde V_{2h(\epsilon_1)} \setminus \hat N, \quad  0 < \epsilon_1 \le \tilde \epsilon_1$ &  \pageref{FittingLemma}, \pageref{defofh1} \\ \hline
			$\delta$ (function of $\epsilon_1$) & $\tilde \epsilon_1$, $h$, $r_0$, $R_0$ &   &  $h_0$, $r_0$, $R_0$  & Measures total loss of $\quad$ domain  &  \pageref{deltadef}   \\ \hline
			${\tilde \epsilon_1}$ (third restriction)  & $\tilde \epsilon_1$ (previous), $r_0$, $\delta$ & $\tilde \epsilon_1$ (previous), $r_0$, & $h_0$, $r_0$, $R_0$ & $U_{R-\delta(\epsilon_1)} \neq  \emptyset$  &  \pageref{tildeepsilon1}    \\
			&  &  $\sup_{(0, \tilde \epsilon_1]}\delta(\epsilon_1)$ &  &  &     \\ \hline 
			$\color{darkgreen} \epsilon_1$ {\color{darkgreen}(restricted)} & ${\tilde \epsilon_1}$ &  &  $h_0$, $r_0$, $R_0$ &  Maximum size of error to be corrected in Phase II &  \pageref{PhaseII}, \pageref{tildeepsilon1}  \\ 
			&  &   &  &  (fixed at the start of `up')&     \\ \hline 
			$h$ (domain restricted) &  ${\tilde \epsilon_1}$, $T$ &  & $h_0$, $r_0$, $R_0$ & $\tilde V_{h(\epsilon_1)} \subset \tilde V_{2h(\epsilon_1)} \setminus \hat N, \quad  0 < \epsilon_1 \le \tilde \epsilon_1$&  \pageref{FittingLemma},  \pageref{defofh1}, \pageref{tildeepsilon1} \\  \hline 
			$\psi$ & $\kappa$ && $\kappa$ & Normalized Riemann $\quad$ map from $U$ to $\D$ &  \pageref{phi2h}  \\ \hline

		\end{tabular}
		
	\end{center}
	
\end{landscape}

\begin{landscape}
	
	\begin{center}
		
		\begin{tabular}{ | p{1.00in} | p{1.00in} | p{1.00in} | p{1.25in} | p{1.80in} | p{0.80in} | }
			\hline
			\textbf{Quantity} & \textbf{Immediate} & \textbf{Intermediate} & \textbf{Ultimate} & \textbf{Role} & \textbf{Defined}   \\
			& \textbf{Dependency} & \textbf{Dependency} & \textbf{Dependency} & &  \textbf{on Page}   \\ \hline \hline

			$\psi_{2h}$ & $\kappa$, $h$  &  &  $\kappa$, $\epsilon_1$, $h_0$, $r_0$, $R_0$ & Normalized Riemann $\quad$ map from $V_{2h}$ to $\D$ & \pageref{phi2h}, \pageref{defofh3}    \\ \hline
			$\phi_{2h}$ &  $\kappa$, $h$, $R$ &  &  $\kappa$, $\epsilon_1$, $R$, $h_0$, $r_0$, $R_0$ & Normalized Riemann $\quad$ map from ${\tilde V_{2h}}$ to $V_{2h}$ &  \pageref{phi2h}, \pageref{defofh3}      \\ \hline
			$R'$ & $h$, $R$ &  & $\epsilon_1$, $R$, $h_0$, $r_0$, $R_0$ & $R' = R^{int}_{(U,0)} \tilde V_{2h}$ & \pageref{Rdash},  \pageref{defofh3}, \pageref{up} \\ \hline 
			$R''$ & $h$, $R$, $\phi_{2h}$ & $\kappa$, $\epsilon_1$, $R$, $h_0$, $r_0$, $R_0$ & $\epsilon_1$, $R$, $h_0$, $r_0$, $R_0$ & $R'' = R^{int}_{(U,0)} \tilde U$ & \pageref{Rdoubledash}, \pageref{defofh3}, \pageref{up} \\ \hline 
			$r_{2h}$ & $h$, $R$ &  & $\epsilon_1$, $R$, $h_0$, $r_0$, $R_0$ & $\psi_{2h}({\tilde V_{2h}})= \mathrm{D}(0,r_{2h})$ &  \pageref{r2hup}, \pageref{Nineq}  \\ \hline
			$s$ & $h$, $\psi_{2h}$  &  & $\epsilon_1$, $R$, $h_0$, $r_0$, $R_0$ & $\psi_{2h}({\overline V_h})\subset \mathrm{D}(0,s)$ &  \pageref{sup}    \\ \hline
			$N$ & $r_{2h}$, $s$  & $\epsilon_1$, $R$, $h_0$, $r_0$, $R_0$ & $\epsilon_1$, $h_0$, $r_0$ & $s\sqrt[N]{\frac{1}{r_{2h}}}< \sqrt{s}$ &  \pageref{Nup}    \\ \hline	
			$g$ & $\psi_{2h}$, $r_{2h}$, $s$, $N$ &  & $\kappa$, $\epsilon_1$, $R$, $h_0$, $r_0$, $R_0$ & Conjugated expansion map &  \pageref{gdef}    \\ \hline
			$R_2$ & $\epsilon_1$, $R_0$ &  & $\epsilon_1$, $R_0$ & $R^{ext}_{(U,0)}\phi_{2h}(U_{R''-\epsilon_1})\leq R_2$ &  \pageref{R2}    \\ \hline
			$K_1$ & & & & $K_1 = \frac{3}{2}$ & \pageref{Nupredef} \\ \hline
			$N$ (redefined) & $g$ (previous), $R_2$  & $\kappa$, $\epsilon_1$, $R$, $h_0$, $r_0$, $R_0$  & $\epsilon_1$, $R$, $h_0$, $r_0$, $R_0$ & $\|g^{\natural} \|_{{\hat A}}\leq K_1$ &  \pageref{Nupredef}  \\ \hline
		        $g$ (redefined) &  $g$ (previous), $R_2$, &  & $\kappa$, $\epsilon_1$, $R$, $h_0$, $r_0$, $R_0$ & $\|g^{\natural} \|_{{\hat A}}\leq K_1$ &  \pageref{Nupredef}    \\ \hline
			${\tilde \epsilon_2}$ (first version) &  &  &  & $\tilde \epsilon_2 = 1$ &  \pageref{epsilon2initialbound}    \\ \hline
			$\color{darkgreen} \epsilon_2$ & ${\tilde \epsilon_2}$ &  & ${\tilde \epsilon_2}$ & $\epsilon_2 \in (0, \tilde \epsilon_2]$ &  \pageref{epsilon2initialbound}    \\ \hline

		\end{tabular}
		
	\end{center}
	
\end{landscape}

\begin{landscape}
	
	\begin{center}
		
		\begin{tabular}{ | p{1.00in} | p{1.00in} | p{1.00in} | p{1.25in} | p{1.80in} | p{0.80in} | }
			\hline
			\textbf{Quantity} & \textbf{Immediate} & \textbf{Intermediate} & \textbf{Ultimate} & \textbf{Role} & \textbf{Defined}   \\
			& \textbf{Dependency} & \textbf{Dependency} & \textbf{Dependency} & &  \textbf{on Page}  \\ \hline \hline
			
			                          ${\bf Q_1}$ & $\kappa$, $\epsilon_2$, $K_1$, $K_2$, $K_3$, $R_2$, $h$, $g$   &  $\kappa$, $\epsilon_1$, $\epsilon_2$, $K_1$, $K_2$, $K_3$, $R$, $h_0$, $r_0$, $R_0$ & $\kappa$, $\epsilon_1$, $\epsilon_2$, $R$, $h_0$, $r_0$, $R_0$ & Approximates $\phi_{2h}$  &  \pageref{bfQ1}    \\ \hline
			$\eta_1$, $\eta_2$ & $\phi_{2h}$, $R_2$ & $\kappa$, $\epsilon_1$, $R$, $h_0$, $r_0$, $R_0$ & $\epsilon_1$, $h_0$, $r_0$, $R_0$ & $\eta_1 \le \|( \phi_{2h}^{-1})^{\natural}\|_{U_{R_2 + 2}}\leq \eta_2$ &  \pageref{etas}    \\ \hline
			${\tilde \epsilon_2}$ (restricted) & $\kappa$, $\epsilon_1$, $\eta_2$ &  & $\epsilon_1$, $h_0$, $r_0$, $R_0$ & Maximum size of error of approximation using $\mathbf Q$ &  \pageref{epsilon2}    \\ \hline
			$\color{darkgreen} \epsilon_2$ & ${\tilde \epsilon_2}$  &  & $\epsilon_1$, $h_0$, $r_0$, $R_0$ & Upper bound for error of approximation using $\mathbf Q$&  \pageref{epsilon2}    \\ \hline
			${\hat \calE}$ & $\kappa$, $\epsilon_1$, $h$, $R$, $\phi_{2h}$, $\calE$ & & $\kappa$, $\epsilon_1$, $R$, $h_0$, $r_0$, $R_0$, $\calE$  & 
			Conjugated error ${\hat \calE}=\phi_{2h}\circ \calE \circ \phi_{2h}^{-1}$ & \pageref{Ehat} \\ \hline
			$K_2$ & $\epsilon_1$, $R''$, $\phi_{2h}$, $\hat \calE$ &  & $\epsilon_1$, $R$, $h_0$, $r_0$, $R_0$, $\calE$  & $|{\hat \calE}^{\natural}(z)|\leq K_2$ for  $z \in \phi_{2h}(U_{R''-2\epsilon_1})$ &  \pageref{K2}    \\ \hline
			${\bf Q_2}$ & $\kappa$, $\epsilon_1$, $\epsilon_2$, $K_2$, $K_3$, $\eta_2$, $\phi_{2h}$, $R''$, $\hat \calE$  &  $\kappa$, $\epsilon_1$, $\epsilon_2$, $K_2$, $K_3$, $R$, $h_0$, $r_0$, $R_0$, $\calE$ & $\kappa$, $\epsilon_1$, $\epsilon_2$, $R$, $h_0$, $r_0$, $R_0$, $\calE$  & Approximates the conjugated error ${\hat \calE}$ on $\phi_{2h}(U_{R''-3\epsilon_1})$ &  \pageref{bfQ2}    \\ \hline
			$K_3$ & $\eta_2$, $R_2$  &  & $\epsilon_1$, $h_0$, $r_0$, $R_0$  & $|(\phi_{2h}^{-1})^{\natural}(z)|\leq K_3$ for  $z \in U_{R_2+2}$ &  \pageref{K3}    \\ \hline
			${\bf Q_3}$ & $\kappa$, $K_3$, $\epsilon_2$, $R_2$, $h$, $\phi_{2h}$ & $\kappa$, $\epsilon_1$, $\epsilon_2$, $K_3$, $R$, $h_0$, $r_0$, $R_0$ & $\kappa$, $\epsilon_1$, $\epsilon_2$, $R$, $h_0$, $r_0$, $R_0$ & Approximates $\phi_{2h}^{-1}$ &  \pageref{bfQ3}    \\ \hline
			${\bf Q}$ & ${\bf Q_1}$, ${\bf Q_2}$, ${\bf Q_3}$ & & $\kappa$, $\epsilon_1$, $\epsilon_2$, $R$, $h_0$, $r_0$, $R_0$, $\calE$ & ${\bf Q} := {\bf Q_3}\circ{\bf Q_2}\circ{\bf Q_1}$ &  \pageref{Qdef} \\ \hline
			
		\end{tabular}
		
	\end{center}
	
\end{landscape}